\documentclass[reqno, 10pt]{amsart}
\usepackage{amsmath,amscd,amsthm,amsxtra, textcomp}
\usepackage{mathtools}
\usepackage{epsfig,graphics,color,colortbl}
\usepackage{amssymb,latexsym}
\usepackage{mathabx}
\usepackage{mathrsfs,eucal,upgreek}
\usepackage[poly,all]{xy}
\usepackage{hyperref}
\usepackage{tikz-cd}
\usepackage{arydshln}
\usepackage{ytableau}
\usepackage[draft]{todonotes}
\usepackage{enumerate}
\usepackage{mathdots}

\makeatletter
\@namedef{subjclassname@2020}{\textup{2020} Mathematics Subject Classification}
\makeatother

\newtheorem{thm}{\bf Theorem}[section]
\newtheorem{defn}[thm]{\bf Definition}
\newtheorem{prop}[thm]{\bf Proposition}

\newtheorem{lem}[thm]{\bf Lemma}
\newtheorem{rem}[thm]{\bf Remark}
\newtheorem{ex}[thm]{\bf Example}

\numberwithin{equation}{section}

\newcommand{\partb}{\beta} % partial b (originally $\beta_{s,t}$)
\newcommand{\sumb}{\Gamma} % sum of b (originally $B_{s, t}$)
\begin{document}
\title[A new description of the bicrystal $B(\infty)$ and the extended crystal]
{A new description of the bicrystal $B(\infty)$ and the extended crystal}
\author{TAEHYEOK HEO}

\address{(T. Heo) Research Institute of Mathematics, Seoul National University, Seoul 08826, Korea}
\email{gjxogur123@snu.ac.kr}

% 05E10 Combinatorial aspects of representation theory

% 17B10 Representations of Lie algebras and Lie superalgebras, algebraic theory (weights)

% 17B37 Quantum groups (quantized enveloping algebras) and related deformations

\thanks{The research of Taehyeok Heo was supported by Basic Science Research Program through the National Research Foundation of Korea (NRF) funded by the Ministry of Education (RS-2023-00241542).}

\begin{abstract}
We define new crystal maps on $B(\infty)$ using its polyhedral realization, 
and show that the crystal $B(\infty)$ equipped with the new crystal maps is isomorphic to Kashiwara's $B(\infty)$ as bicrystals. 
In addition, we combinatorially describe the bicrystal structure of $B(\infty)$, which is called a sliding diamond rule. 
Using the bicrystal structure on $B(\infty)$, we define the extended crystal and show that it is isomorphic to the extended crystal introduced by Kashiwara and Park. 
\end{abstract}

\maketitle
\setcounter{tocdepth}{1} % 1 = section level, 2 = subsection level
\tableofcontents

\section{Introduction}

The crystal base $B(\infty)$ of $U_q^-(\mathfrak{g})$ is introduced by Lusztig \cite{Lus90} for the Lie algebra $\mathfrak{g}$ of type $ADE$ and Kashiwara \cite{Kas91} for a Kac-Moody algebra $\mathfrak{g}$. 
To understand $B(\infty)$ thoroughly, we need to consider its star crystal, which is induced by an anti-involution $\ast$ on $U_q(\mathfrak{g})$ 
as many properties of $B(\infty)$ can be described using the star crystal structure (see, e.g.,\cite{Kas91,Kas93,Kas95}).
For example, for $i \in I$, there exists a strict embedding $\Psi_i : B(\infty) \to B(\infty) \otimes {\bf B}_i$ of crystals, which gives the value $\varepsilon_i^\ast(b)$ for $b \in B(\infty)$ (see Section \ref{subsec:crystal}). 

As combinatorial approaches to understanding $B(\infty)$, many models are introduced such as path models \cite{LZ11, LS19}, % Littelmann's model is the LS path
tableau models \cite{Cliff98,HongLee08,HongLee12}, (modified) Nakajima monomials \cite{KKS07,KS08}, rigged configurations \cite{SS15,SS17,SS18}, PBW bases \cite{Saito94} and so on. 
In this paper, we focus on the polyhedral realization of $B(\infty)$ introduced by Nakashima and Zelvinsky \cite{NZ97}. 
For a given sequence $\iota = (\dots, i_2, i_1)$ of indices, we embed $B(\infty)$ into the $\mathbb{Z}$-lattice $\mathbb{Z}^\infty$ by applying $\Psi_\iota := \dots \circ \Psi_{i_2} \circ \Psi_{i_1}$, 
and a natural question is to describe the image $\mathcal{B}(\infty)$ of $B(\infty)$ under the embedding $\Psi_\iota$. 
It is proved in \cite{NZ97} that $\mathcal{B}(\infty)$ is the set of points in $\mathbb{Z}_+^\infty$ satisfying some linear inequalities (see Theorem \ref{thm:polyhedral realization}), 
which is the reason why we call it the polyhedral realization of $B(\infty)$. 
The defining inequalities are determined by the underlying Lie algebra $\mathfrak{g}$ and the sequence $\iota$, and some results are known. 
When we choose $\iota$ as the fixed one $\iota_0$ (see \eqref{eqn:iota}), the inequalities are given in \cite{NZ97} for $\mathfrak{g}$ of type $A$ or rank $2$, 
\cite{Ho05} for finite types $\mathfrak{g}$, \cite{Ho13} for $\mathfrak{g}$ of types $A_{2n-1}^{(2)}, A_{2n}^{(2)}, B_n^{(1)}, C_n^{(1)}, D_n^{(1)}$, and $D_{n+1}^{(2)}$. 
For other choices of $\iota$, the authors of \cite{KaNa20} introduce the notion of an adapted sequence and give the defining inequalities for an adapted sequence $\iota$. 

The crystal structure of $\mathcal{B}(\infty)$ is well-known (cf. \cite{NZ97}). 
In particular, we compute the associated crystal maps by evaluating special linear functions $\partb_{u,v}$ at $b \in \mathcal{B}(\infty)$, 
and by comparing the values $\sumb_{(s)}(b)$ (denoted by $\sigma_{s,t}(b)$ in \cite{NZ97}) which is a sum of some $\partb_{u,v}(b)$'s. 
In contrast, one can find a few results on the star crystal structure of $\mathcal{B}(\infty)$; for example, for $b \in \mathcal{B}(\infty)$, 
the crystal operators $\widetilde{x}_{i_1}^\ast(b)$ ($x = e$ or $f$) \cite[Theorem 2.2.1]{Kas93} and the function $\varepsilon_i^\ast(b)$ \cite[Theorem 4.2]{Nas99}. \vspace{1em}

The purpose of the paper is to describe the star crystal structure of $\mathcal{B}(\infty)$ (with respect to $\iota_0$) when $\mathfrak{g}$ is of type $A_n, B_n, D_n$. 
To achieve the goal, we introduce tableaux $T_i^\ast$ (Definition \ref{def:partition set}) for $i \in I$ whose cells are filled with 
linear functions $\partb_{u,v}^\ast$ which play a role of $\partb_{u,v}$ for the star crystal. 
As the usual crystal does, we calculate the value $\sumb_{\lambda}^\ast(b)$ as a sum of some $\partb_{u,v}^\ast(b)$'s, for $\lambda \in \Pi_i^\ast$. 
In fact, each $\lambda \in \Pi_i^\ast$ designates which $\partb_{u,v}^\ast(b)$ is in the summand of $\sumb_\lambda^\ast(b)$ 
and it is done by reading entries contained in the (unique) sub-tableau of $T_i^\ast$ of shape $\lambda$. 
Now, we can define new crystal operators $\widetilde{x}_i^\ast$ ($x = e$ or $f$) and maps $\varepsilon_i^\ast, \varphi_i^\ast$ on $b \in \mathcal{B}(\infty)$ 
by comparing the values $\sumb_\lambda^\ast(b)$ for $\lambda \in \Pi_i^\ast$ (Definition \ref{def:crystal star}). 
We remark that we similarly introduce tableaux $T_i$ and define crystal operators $\widetilde{x}_i$ and functions $\varepsilon_i, \varphi_i$ on $b \in \mathcal{B}(\infty)$ 
by observing values $\sumb_{(s)}(b)$ for $(s) \in \Pi_i$ (Definition \ref{def:crystal usual}). 
We can directly check that these crystal maps are exactly the same as those in \cite{NZ97}. 

The main result of the paper is to show that the crystal $\mathcal{B}(\infty)$ endowed with the (star) crystal structure we define is isomorphic to $B(\infty)$ of Kashiwara as bicrystals (Theorem \ref{thm:bicrystal}). 
Here, a bicrystal means a crystal with additional crystal structure such that their weight functions coincide (cf. \cite[Definition 2.5]{CT15}). 
In fact, $B(\infty)$ is characterized by its two crystal structures. 
To be precise, it is known that if a crystal $B$ equipped with another crystal structure satisfies some conditions (see Proposition \ref{prop:bicrystal}), 
then $B$ is isomorphic to $B(\infty)$ as bicrystals and the additional crystal structure of $B$ coincides with the star crystal structure of $B(\infty)$. 
Thus, to prove Theorem \ref{thm:bicrystal}, it is sufficient to show that our bicrystal maps satisfy all conditions in Proposition \ref{prop:bicrystal}.
From now on, $\mathcal{B}(\infty)$ is considered to be the bicrystal.

Meanwhile, in order to calculate $\partb_{s,t}(b)$ or $\partb_{s,t}^\ast(b)$ for $b \in \mathcal{B}(\infty)$ easily, we introduce a combinatorial rule, which we call a sliding diamond rule (Definition \ref{def:SD}). 
We put coordinates of $b$ in a certain configuration, and find diamonds $\Diamond_{s,t}$ and $\Diamond_{s,t}^\ast$ in the configuration. 
Then we obtain $\partb_{s,t}(b)$ and $\partb_{s,t}^\ast(b)$ as a linear combination of integers the corresponding diamond contains. 
In particular, when we compute associated crystal maps on $\mathcal{B}(\infty)$, it is enough to consider diamonds contained in a specific line of the configuration. \vspace{1em}

Using the description of the bicrystal structure on $\mathcal{B}(\infty)$, we combinatorially describe the extended crystal $\widehat{B}(\infty)$ introduced by Kashiwara and Park \cite{KP22}. 
For a Kac-Moody algebra $\mathfrak{g}$, the extended crystal $\widehat{B}(\infty) (= \widehat{B}_\mathfrak{g}(\infty))$ is the set
\[ \widehat{B}(\infty) = \left\{ {\bf b} = (b^{(k)})_{k \in \mathbb{Z}} \in \prod_{k \in \mathbb{Z}} B(\infty) \,:\, b^{(k)} = {\bf 1} \mbox{ for all but finitely many }k \right\} \]
equipped with extended crystal operators $\widetilde{E}_{(i,k)}, \widetilde{F}_{(i,k)}$ and functions $\widehat{\rm wt}, \widehat{\varepsilon}_{(i,k)}$ for $(i, k) \in I \times \mathbb{Z}$, 
where ${\bf 1}$ is the unique highest weight vector of $B(\infty)$ (see Definition \ref{def:extended crystal}). 
By definition, the extended crystal maps are completely determined by the bicrystal maps acting on each $b^{(k)} \in B(\infty)$, 
and hence we apply our results on $\mathcal{B}(\infty)$ to the extended crystal for type $A_n, B_n$, and $D_n$, which results in a combinatorial description of $\widehat{B}(\infty)$ (Theorem \ref{thm:extended crystal iso}). 
In addition, we obtain the extended sliding diamond rule when $\mathfrak{g}$ is of type $A_n$ (Definition \ref{def:extended SD}). 

As an application, we use the combinatorial description on $\widehat{B}(\infty)$ in order to understand the representation theory of quantum affine Lie algebras. 
Let $\mathcal{C}_\mathfrak{g}^0$ be a Hernandez-Leclerc category (cf. \cite{HL10} and \cite[Section 2.2]{KP22}) over a quantum affine algebra $U_q'(\mathfrak{g})$, 
and let $\mathscr{B}(\mathfrak{g})$ be the set of isomorphism classes of simple modules in $\mathcal{C}_\mathfrak{g}^0$. 
It is proved in \cite{KP22} that $\mathscr{B}_\mathcal{D}(\mathfrak{g}) = \mathscr{B}(\mathfrak{g})$ has the categorical crystal structure 
depending on a complete duality datum $\mathcal{D}$ (see \cite{KKOP24} or Section \ref{subsec:categorical crystal review}) and 
it is isomorphic to the extended crystal $\widehat{B}_{\mathfrak{g}_{\rm fin}}(\infty)$ with the isomorphism (see \eqref{eqn:extended category crystal iso})
\[ \Phi_\mathcal{D} : \widehat{B}_{\mathfrak{g}_{\rm fin}}(\infty) \to \mathscr{B}_\mathcal{D}(\mathfrak{g}), \] 
where $\mathfrak{g}_{\rm fin}$ is the finite simple Lie algebra associated with $\mathfrak{g}$ (see \cite[Table 1]{KP22}). 
As we can explain the extended crystal $\widehat{B}_{\mathfrak{g}_{\rm fin}}(\infty)$ from the sliding diamond rule, 
we apply the combinatorial description using the sliding diamond rule to the categorical crystal $\mathscr{B}_\mathcal{D}(\mathfrak{g})$, 
which gives a combinatorial way to understand the representation theory of quantum affine Lie algebras. \vspace{1em}

This paper is organized as follows. We review the notion of crystals and the extended crystal in Section \ref{sec:prelim}. 
In Section \ref{sec:bicrystal}, we explicitly define a new crystal structure on $\mathcal{B}(\infty)$ based on the polyhedral realization, 
and explain the combinatorial model to describe $\mathcal{B}(\infty)$. 
As a result, we show that the bicrystal $\mathcal{B}(\infty)$ is isomorphic to the bicrystal $B(\infty)$ of Kashiwara, 
which is proved in Section \ref{sec:bicrystal proof} (with Appendix). 
In Section \ref{sec:extended}, we introduce the extended crystal analogue $\widehat{\mathcal{B}}(\infty)$ of $\mathcal{B}(\infty)$, 
and show that $\widehat{\mathcal{B}}(\infty)$ is a new combinatorial realization of the extended crystal. 
As an application to representation theory of quantum affine algebras, we give examples for the action of the extended crystal operators 
on the categorical crystal in Section \ref{sec:categorical crystal}. 

{\bf Acknowledgement.} The author would like to appreciate Euiyong Park for valuable discussions and kind comments on the draft. 
% The author also thanks the anonymous referees for careful reading and comments.

\section{Preliminaries} \label{sec:prelim}
\subsection{Crystals} \label{subsec:crystal}
Let $\mathbb{Z}_+$ be the set of nonnegative integers. 
A partition $\lambda = (\lambda_1, \lambda_2, \dots) = (\lambda_i)_{i \geq 1}$ is a sequence of weakly decreasing nonnegative integers. 
The length $\ell(\lambda)$ of a partition $\lambda$ is the number of nonzero entries of $\lambda$. 
In particular, a partition is said to be strict if its nonzero entries are all distinct. 
We usually illustrate a partition as a Young diagram in English convention, which is a collection of cells arranged in left-justified rows (cf. \cite{Fulton}). 
For partitions $\lambda$ and $\mu$, the union $\lambda \cup \mu$ of $\lambda$ and $\mu$ is the partition $(\max\{\lambda_i, \mu_i\})_{i \geq 1}$, 
and the intersection $\lambda \cap \mu$ of $\lambda$ and $\mu$ is the partition $(\min\{\lambda_i, \mu_i\})_{i \geq 1}$. 
The $(i, j)$-th cell of a Young diagram means the $j$-th cell from the left in the $i$-th row from the top. 
In particular, we write $(i, j) \in \lambda$ if a partition $\lambda$ contains $(i, j)$-th cell, or equivalently $1 \leq i \leq \ell(\lambda)$ and $1 \leq j \leq \lambda_i$. 
For a partition $\lambda$, a tableau $T$ of shape $\lambda$ is a filling of cells of the partition $\lambda$. 
In particular, we denote by $T(i, j)$ the entry occupying the $(i, j)$-th cell of $T$ for $(i, j) \in \lambda$.

Let $U_q(\mathfrak{g})$ be the quantum group associated with a symmetrizable Kac-Moody algebra $\mathfrak{g}$, 
in other words, the associative $\mathbb{Q}(q)$-algebra generated by $e_i, f_i$ ($i \in I$), and $q^h$ ($h \in P^\ast$) with some defining relations (cf. \cite{Kas91}). 
Here, we denote by $I$ the index set, by $P$ the weight lattice of $\mathfrak{g}$, and let $\{ \alpha_i \,|\, i \in I \}$ be the set of simple roots.
Let $\ast$ be the $\mathbb{Q}(q)$-linear anti-automorphism on $U_q(\mathfrak{g})$ defined by
\[ e_i^\ast = e_i, \quad f_i^\ast = f_i, \quad \mbox{and}\quad (q^h)^\ast = q^{-h}. \]

\begin{defn} \label{def:crystal def}
    A ($\mathfrak{g}$-)crystal is the set $B$ equipped with the maps ${\rm wt} : B \to P$, $\widetilde{e}_i, \widetilde{f}_i : B \to B \cup \{ {\bf 0} \}$, 
    $\varepsilon_i, \varphi_i : B \to \mathbb{Z} \cup \{ -\infty \}$ ($i \in I$) satisfying the following conditions. For $i \in I$ and $b, b' \in B$,
    \begin{enumerate}
        \item $\varphi_i(b) = \varepsilon_i(b) + \langle h_i, {\rm wt}(b) \rangle$,
        \item Whenever $\widetilde{e}_i(b) \neq {\bf 0}$, we have
        \[ {\rm wt}(\widetilde{e}_ib) = {\rm wt}(b) + \alpha_i, \quad \varepsilon_i(\widetilde{e}_ib) = \varepsilon_i(b) -1, \quad \varphi_i(\widetilde{e}_ib) = \varphi_i(b) +1, \]
        \item Whenever $\widetilde{f}_i(b) \neq {\bf 0}$, we have
        \[ {\rm wt}(\widetilde{f}_ib) = {\rm wt}(b) - \alpha_i, \quad \varepsilon_i(\widetilde{f}_ib) = \varepsilon_i(b) +1, \quad \varphi_i(\widetilde{f}_ib) = \varphi_i(b) -1, \]
        \item $\widetilde{f}_ib = b'$ if and only if $b = \widetilde{e}_ib'$,
        \item If $\varepsilon_i(b) = -\infty$, then $\widetilde{e}_ib = \widetilde{f}_ib = {\bf 0}$.
    \end{enumerate}
    Here, ${\bf 0}$ is the formal symbol and $-\infty$ is another formal symbol such that $-\infty < n$ and $(-\infty) + n = -\infty$ for any $n \in \mathbb{Z}$ with $(-\infty)+(-\infty) = -\infty$. 
\end{defn}
As an example of crystals, we introduce the elementary crystal ${\bf B}_i$ ($i \in I$), which is the set
\[ {\bf B}_i = \{ (b)_i \,|\, b \in \mathbb{Z} \} \]
where ${\rm wt}((b)_i) = b\alpha_i$, $\widetilde{e}_i((b)_i) = (b+1)_i, \widetilde{f}_i((b)_i) = (b-1)_i$, 
$\varepsilon_i((b)_i) = -b, \varphi_i((b)_i) = b$, $\widetilde{e}_j((b)_i) = \widetilde{f}_j((b)_i) = {\bf 0}$ for $j \neq i$.

Our main object is the crystal $B(\infty)$ corresponding to the negative half $U_q^-(\mathfrak{g})$ of the quantum group, 
which is the subalgebra of $U_q(\mathfrak{g})$ generated by $f_i$'s ($i \in I$). 
When we want to indicate the associated Lie algebra $\mathfrak{g}$ explicitly, we write it as a subscript, e.g., such as $B_\mathfrak{g}(\infty)$. 
We summarize some properties of $B(\infty)$ with respect to the involution $\ast$. 
Kashiwara \cite{Kas93} shows that the crystal $B(\infty)$ is invariant under $\ast$, that is, $B(\infty)^\ast = B(\infty)$. 
It yields another crystal structure of $B(\infty)$ (cf. \cite[Section 8.3]{Kas95}) endowed with the maps ${\rm wt}, \widetilde{e}_i^\ast, \widetilde{f}_i^\ast, \varepsilon_i^\ast, \varphi_i^\ast$ on $B(\infty)$, where
\begin{eqnarray*}
    \widetilde{e}_i^\ast(b) &:=& (\widetilde{e}_i (b^\ast))^\ast, \\
    \widetilde{f}_i^\ast(b) &:=& (\widetilde{f}_i (b^\ast))^\ast, \\
    \varepsilon_i^\ast(b) &:=& \varepsilon_i(b^\ast), \\
    \varphi_i^\ast(b) &:=& \varphi_i(b^\ast)
\end{eqnarray*}
for $i \in I$ and $b \in B(\infty)$. 
We denote by ${\bf 1} \in B(\infty)$ the unique highest weight element of $B(\infty)$ (with weight $0$).

On the other hand, for $i \in I$, there exists the unique embedding of crystals
\[ \Psi_i : B(\infty) \hookrightarrow B(\infty) \otimes {\bf B}_i \]
sending ${\bf 1}$ to ${\bf 1} \otimes (0)_i$ (see \cite[Theorem 2.2.1]{Kas93}). 
In particular, we have $\Psi_i(b) = \widetilde{e}_i^{\ast\max} (b) \otimes (-\varepsilon_i^\ast(b))_i$, 
where $\widetilde{e}_i^{\ast\max}(b) = (\widetilde{e}_i^\ast)^m(b)$ with $m = \varepsilon_i^\ast(b)$, 
and the image of $\Psi_i$ is
\[ \{\, b \otimes (-m)_i \,|\, \varepsilon_i^\ast(b) = 0, m \in \mathbb{Z}_+ \}. \]

\subsection{The polyhedral realization}
We assume that a sequence $\iota = (\dots, i_3, i_2, i_1)$ of indices satisfies the following conditions:
\[ \mbox{$i_k \neq i_{k+1}$ for all $k \geq 1$}, \qquad \mbox{\#$\{ k \geq 1 : i_k = i \} = \infty$ for $i \in I$}. \]
Then we repeatedly apply $\Psi_i$ according to $\iota$ to obtain a sequence of embeddings of crystals
\[ \dots \circ \Psi_{i_2} \circ \Psi_{i_1} : B(\infty) \hookrightarrow B(\infty) \otimes {\bf B}_{i_1} \hookrightarrow B(\infty) \otimes {\bf B}_{i_2} \otimes {\bf B}_{i_1} \hookrightarrow \cdots. \]
Consequently, we obtain an embedding of crystals
\[ \Psi_\iota : B(\infty) \hookrightarrow \mathbb{Z}^\infty \]
sending $b \in B(\infty)$ to $(\dots, b_2, b_1)$ where $b_k = \varepsilon_{i_k}(\widetilde{e}_{i_{k-1}}^{\ast\max} \dots \widetilde{e}_{i_1}^{\ast\max} b)$ 
and $b_t = 0$ for sufficiently large $t$. We call this embedding the Kashiwara embedding (associated with $\iota$). 
Note that we identify $\dots \otimes {\bf B}_{i_2} \otimes {\bf B}_{i_1}$ with $\mathbb{Z}^\infty$ via
\[ \dots \otimes (-m_2)_{i_2} \otimes (-m_1)_{i_1} \ \longleftrightarrow\ (\dots, m_2, m_1) \]
for a notational convenience.

To describe the image of $\Psi_\iota$, consider the infinite-dimensional vector space
\[ \mathbb{Q}^\infty = \{ (\dots, m_2, m_1) \,|\, m_i \in \mathbb{Q}, m_k = 0 \mbox{ for sufficiently large }k \}, \]
and the set $\Xi_\iota$ of linear functions on $\mathbb{Q}^\infty$ given in \cite[(3.4)]{NZ97}, which is determined by the sequence $\iota$. 
In addition, we impose the following positivity assumption to $\iota$;
\[ \mbox{for $k \geq 1$, if } i_j \neq i_k \mbox{ for all }j < k,\mbox{ then } \varphi_k \geq 0\mbox{ for any } \varphi = \sum \varphi_k x_k \in \Xi_\iota, \]
where ${\bf e}_k$ is the $k$-th standard basis of $\mathbb{Q}^\infty$ and $x_k$ is the $k$-th dual basis of $\{ {\bf e}_k \}$, or equivalently $x_i({\bf e}_j) = \delta_{ij}$ for $i, j \geq 1$. 
\begin{thm}[{\cite[Theorem 3.1]{NZ97}}] \label{thm:polyhedral realization}
    Under the positivity assumption on $\iota$, the image of $B(\infty)$ under $\Psi_\iota$ is
    \[ \mathcal{B}_\iota(\infty) = \{ {\bf x} \in \mathbb{Z}_+^\infty \,|\, \varphi({\bf x}) \geq 0\mbox{ for all }\varphi \in \Xi_\iota \}. \]
\end{thm}
As $\Psi_i$ produces $\varepsilon_i^\ast(b)$ for $b \in B(\infty)$, we also want to get $\varepsilon_i^\ast(b)$ from $\Psi_\iota$, 
and the result is given in \cite{Nas99}. 
In fact, we need an additional assumption (strict positivity assumption) on $\iota$ (see \cite[Section 4.3]{Nas99}), but it is always satisfied in our paper.
\begin{thm}[{\cite[Theorem 4.2]{Nas99}}] \label{thm:epsilon}
    For $b \in \mathcal{B}(\infty)$ and $i \in I$, we have
    \[ \varepsilon_i^\ast(b) = \max\{ -\varphi(b) \,|\, \varphi \in \Xi_\iota^{(i)} \}, \]
    where $\Xi_\iota^{(i)}$ is the set of linear functions on $\mathbb{Z}^\infty$ given in \cite[(4.21)]{Nas99}.
\end{thm}
When there is no confusion on $\iota$, we use the expression like $\mathcal{B}(\infty)$ and $\Xi$.

In this paper, we always assume 
\begin{equation} \label{eqn:iota}
    \iota = \iota_0 = (\dots, n, \dots, 1, n, \dots, 1)
\end{equation}
where $I = \{ 1, \dots, n \}$ is the index set of the symmetrizable Kac-Moody algebra $\mathfrak{g}$. 
We choose the following two-parameter index; for $(s,t) \in \mathbb{N} \times I$, we put $b_{s,t} = b_k$ with $k = n(s-1)+t$. 
In other words, it means that $i_k = t$ and it is the $s$-th appearence of $t$ in $\iota$. 
Any other notations with the two-parameter index, e.g. ${\bf e}_{s,t}$ and $x_{s,t}$, are understood in this context.

Let 
\begin{eqnarray*}
    \mathcal{I}_n^A &=& \{ (s, t) \in \mathbb{N} \times I \,|\, s+t \leq n+1 \}, \\
    \mathcal{I}_n^B &=& \{ (s, t) \in \mathbb{N} \times I \,|\, 1 \leq s \leq n \}, \\
    \mathcal{I}_n^D &=& \{ (s, t) \in \mathbb{N} \times I \,|\, 1 \leq s \leq n-1 \}.
\end{eqnarray*}
Then we have
\[ \mathcal{B}(\infty) = \left\{ (b_{s,t}) \in \mathbb{Z}_+^\infty \,\left|\ \begin{minipage}{0.44\textwidth}
    $b_{s,t} = 0$ unless $(s, t) \in \mathcal{I}_n^A, \\ b_{1,k} \geq b_{2, k-1} \geq \dots \geq b_{k, 1}$ for $1 \leq k \leq n$
    \end{minipage}
\right. \right\} \]
when $\mathfrak{g}$ is of type $A_n$, which is shown in \cite[Theorem 5.1]{NZ97}, 
\[ \mathcal{B}(\infty) = \left\{ (b_{s,t}) \in \mathbb{Z}_+^\infty \,\left|\ \begin{minipage}{0.73\textwidth}
    $b_{s,t} = 0$ unless $(s, t) \in \mathcal{I}_n^D, \\ b_{1,k} \geq b_{2, k-1} \geq \dots \geq b_{k, 1}$ for $1 \leq k \leq n-2$, \\
    $b_{k,n-1}+b_{k,n} \geq b_{k+1,n-2} \geq \dots \geq b_{n-1,k}$ for $1 \leq k \leq n-2$, \\
    $b_{k,n-k} \geq b_{k, n-k+1} \geq \dots \geq b_{k, n-2} \geq b_{k, n-1}+b_{k, n}$ for $2 \leq k \leq n-1$, \\
    $b_{1, n-1} \geq b_{2, n} \geq b_{3, n-1} \geq \dots \geq b_{n-1, n}\mbox{ or } b_{n-1, n-1}$ \\
    $b_{1, n} \geq b_{2, n-1} \geq b_{3, n} \geq \dots \geq b_{n-1, n-1}\mbox{ or } b_{n-1, n}$
    \end{minipage}
\right. \right\} \]
when $\mathfrak{g}$ is of type $D_n$, which is shown in \cite[Theorem 3.12]{Ho05}, and
\[ \mathcal{B}(\infty) = \left\{ (b_{s,t}) \in \mathbb{Z}_+^\infty \,\left|\ \begin{minipage}{0.52\textwidth}
    $b_{s,t} = 0$ unless $(s, t) \in \mathcal{I}_n^B, \\ b_{1,k} \geq b_{2, k-1} \geq \dots \geq b_{k, 1}$ for $1 \leq k \leq n-1$, \\
    $b_{k,n} \geq b_{k+1,n-1} \geq \dots \geq b_{n,k}$ for $1 \leq k \leq n$, \\
    $b_{k,n-k+1} \geq b_{k, n-k+2} \geq \dots \geq b_{k, n}$ for $2 \leq k \leq n$
    \end{minipage}
\right. \right\} \]
when $\mathfrak{g}$ is of type $B_n$, which is shown in \cite[Theorem 3.5]{Ho05}.
For the description of $\mathcal{B}(\infty)$ for other Kac-Moody algebras, see \cite{Ho05,Ho13}. 
In addition, one can find the description of $\mathcal{B}(\infty)$ for different choices of $\iota$ in \cite{KaNa20,NZ97}.

Throughout this paper, we assume ${\bf e}_{s,t} = 0$ and $x_{s, t} = 0$ if $(s, t) \not\in \mathcal{I}_n^X$. 

\subsection{The PBW crystal} \label{subsec:PBW}
Assume that $\mathfrak{g}$ is a finite-dimensional simple Lie algebra. 
Let $T_i = T_{i,-1}'$ be the $\mathbb{Q}(q)$-algebra automorphism on $U_q(\mathfrak{g})$ given in \cite[Chapter 37]{Lus}. 
For the longest element $w_0$ in the Weyl group $W$ with $\ell = \ell(w_0)$, 
fix a reduced expression $\underline{w_0} = s_{i_1} \dots s_{i_\ell}$ of $w_0$. 
Then the set of positive roots is enumerated by
\[ \beta_1 = \alpha_{i_1}, \quad \beta_{k+1} = s_{i_1}\dots s_{i_k}(\alpha_{i_{k+1}}) \mbox{ for }k \geq 1. \]
For $1 \leq k \leq \ell$ and $c \in \mathbb{Z}_+$, let
\[ F_{\beta_k}^{(c)} = T_{i_1} T_{i_2} \dots T_{i_{k-1}}(f_{i_k}^{(c)}), \]
where $f_i^{(c)}$ is the $c$-th divided power (cf. \cite[Section 1.4]{Lus}), and,
for ${\bf c} = (c_1, \dots, c_\ell) \in \mathbb{Z}_+^\ell$, let
\[ F({\bf c}) = F_{\beta_\ell}^{(c_\ell)} \dots F_{\beta_1}^{(c_1)}. \]
Then the set $\{ F({\bf c}) \,|\, c_i \in \mathbb{Z}_+ \}$ is a $\mathbb{Q}(q)$-basis of $U_q^-$, which is called the PBW basis of $U_q^-(\mathfrak{g})$. 
We remark that we choose different convention of $F({\bf c})$ from the one in \cite{Saito94} for the notational coherence (cf. Section \ref{sec:categorical crystal}). 

Moreover, it is proved in \cite{Saito94} that $\{ F({\bf c}) \pmod{q L(\infty)} \,|\, {\bf c} \in \mathbb{Z}_+^\ell \} = B(\infty)$, 
where $L(\infty)$ is the crystal lattice of $U_q^-(\mathfrak{g})$ (cf. \cite{Kas91}). 
In other words, for $b \in B(\infty)$, there is unique ${\bf c} \in \mathbb{Z}_+^\ell$ such that $b \equiv F({\bf c}) \pmod{q L(\infty)}$, and 
we call such ${\bf c}$ the PBW datum (or Lusztig datum) of $b$. 
The PBW crystal $\mathbb{B}_{\underline{w_0}}(\infty)$ (associated with $\underline{w_0}$) is the set $\mathbb{Z}_+^\ell$ with the crystal structure 
determined by the identification between ${\bf c} \in \mathbb{Z}_+^\ell$ and $F({\bf c}) \pmod{qL(\infty)} \in B(\infty)$. 

Since two crystals $\mathcal{B}(\infty)$ and $\mathbb{B}_{\underline{w_0}}(\infty)$ are isomorphic, 
it is natural to ask whether we can construct the explicit bijection between them. 
In principle, the crystal isomorphism 
\begin{equation} \label{eqn:PBW-polyhed}
    \phi : \mathbb{B}_{\underline{w_0}}(\infty) \longrightarrow \mathcal{B}(\infty)
\end{equation}
can be constructed by calculating $\varepsilon_i^\ast(b)$ and $\widetilde{e}_i^{\ast\max}(b)$ for $b \in \mathbb{B}_{\underline{w_0}}(\infty)$. 
For some choices of a reduced expression $\underline{w_0}$, the crystal structure of $\mathbb{B}_{\underline{w_0}}(\infty)$ is given in some references such as \cite{Kwon18, SST18}.
In particular, we focus on the isomorphism for the following reduced expression $\underline{w_0}$. 
\begin{equation} \label{eqn:reduced}
    \begin{cases}
        \underline{w_0} = (s_1)(s_2s_1) (s_3s_2s_1) \dots (s_{n-1} \dots s_1) & \mbox{for type $A_n$} \\
        \underline{w_0} = (s_1 \dots s_n)^n & \mbox{for type $B_n$} \\
        \underline{w_0} = (s_1 \dots s_n)^{n-1} & \mbox{for type $D_n$}
    \end{cases}
\end{equation}
When we consider the reduced expression as the above $\underline{w_0}$, we use notations $\mathbb{B}(\infty)$ instead of $\mathbb{B}_{\underline{w_0}}(\infty)$. 

In particular, when $\mathfrak{g}$ is of type $A_n$, we can explicitly describe the bijection $\phi$. 
Recall that a reduced expression of $w_0$ induces a convex order on the positive roots. 
The corresponding convex order to $\underline{w_0} = (s_1)(s_2s_1) (s_3s_2s_1) \dots (s_{n-1} \dots s_1)$ is the following.
\begin{eqnarray*}
    && \beta_1 = \alpha_1 < \alpha_1+\alpha_2 < \alpha_2 < \alpha_1 + \alpha_2 + \alpha_3 < \alpha_2 + \alpha_3 < \alpha_3 \\
    && \hspace{1cm} < \dots < \alpha_1 + \dots + \alpha_{n-1} < \alpha_2 + \dots + \alpha_{n-1} < \dots < \alpha_{n-1} = \beta_\ell
\end{eqnarray*}
When $\beta_k = \alpha_i + \alpha_{i+1} + \dots + \alpha_j$ for some $1 \leq i \leq j \leq n-1$, we write $c_k = c_{[i,j]}$. 
Then the isomorphism $\phi$ sends $(c_1, \dots, c_\ell) = (c_{[1,1]}, \dots, c_{[n-1,n-1]}) \in \mathbb{B}(\infty)$ to 
$(b_{n-1,1}, \dots, b_{1,2}, b_{1,1}) \in \mathcal{B}(\infty)$, where
% \begin{eqnarray*}
%     && b_{11} = c_{11}, \quad b_{12} = c_{12}+c_{22}, \quad\dots, \quad b_{1,n-1} = c_{1,n-1} + c_{2,n-1} + \dots + c_{n-1,n-1} \\
%     && b_{21} = c_{12}, \quad b_{22} = c_{22}, \quad \dots, \quad b_{2,n-2} = c_{2, n-1} + \dots + c_{n-1, n-1} \\
%     && \cdots \\
%     && b_{n-1, 1} = c_{1,n-1}
% \end{eqnarray*}
\[ b_{i,j} = \sum_{t=1}^j c_{[t, i+j-1]}. \]
We remark that the inverse map is given by the equations 
\[ c_{[i,j]} = b_{j-i+1,i} - b_{j-i+2, i-1} \]
assuming $b_{k,0} = 0$ for $k \in \mathbb{Z}$. 

\subsection{The extended crystal}
We recall the notion of the extended crystal (see \cite{KP22}). For a generalized Kac-Moody algebra $\mathfrak{g}$, let
\[ \widehat{B}_\mathfrak{g}(\infty) = \left\{ {\bf b} = (b^{(k)})_{k \in \mathbb{Z}} \in \prod_{k \in \mathbb{Z}} B_\mathfrak{g}(\infty) \,:\, b^{(k)} = {\bf 1} \mbox{ for all but finitely many }k \right\} \]
be the extended crystal, which is an infinite copy of $B_\mathfrak{g}(\infty)$ as a set. 
To record the position, we draw the underline at the $0$-th component of ${\bf b} \in \widehat{B}_\mathfrak{g}(\infty)$, 
e.g. ${\bf b} = (\dots, b^{(2)}, b^{(1)}, \underline{b^{(0)}}, b^{(-1)}, \dots)$.
\begin{defn}[\cite{KP22}] \label{def:extended crystal}
    Let $\widehat{I} := I \times \mathbb{Z}$. 
    The extended crystal is the set $\widehat{B}_\mathfrak{g}(\infty)$ with the following maps 
    $(\widehat{B}_\mathfrak{g}(\infty), \widehat{\rm wt}, \widetilde{E}_{(i,k)}, \widetilde{F}_{(i,k)}, \widehat{\varepsilon}_{(i,k)})$. 
    For $(i, k) \in \widehat{I}$ and ${\bf b} = (b^{(k)}) \in \widehat{B}_\mathfrak{g}(\infty)$,
    \begin{itemize}
        \item $\displaystyle \widehat{\rm wt}({\bf b}) = \sum_{t \in \mathbb{Z}} (-1)^t \,{\rm wt}(b^{(t)})$,
        \item $\widehat{\varepsilon}_{(i,k)}({\bf b}) = \varepsilon_i(b^{(k)}) - \varepsilon_{i+1}^\ast(b^{(k+1)})$,
        \item (the extended crystal operators)
        \[ \widetilde{F}_{(i,k)}({\bf b}) = \begin{cases}
            (\dots, b^{(k+2)}, b^{(k+1)}, \widetilde{f}_i(b^{(k)}), b^{(k-1)}, \dots) & \mbox{if } \widehat{\varepsilon}_{(i,k)}({\bf b}) \geq 0, \\
            (\dots, b^{(k+2)}, \widetilde{e}_i^\ast(b^{(k+1)}), b^{(k)}, b^{(k-1)}, \dots) & \mbox{if } \widehat{\varepsilon}_{(i,k)}({\bf b}) < 0,
        \end{cases} \]
        \[ \widetilde{E}_{(i,k)}({\bf b}) = \begin{cases}
            (\dots, b^{(k+2)}, b^{(k+1)}, \widetilde{e}_i(b^{(k)}), b_{(k-1)}, \dots) & \mbox{if } \widehat{\varepsilon}_{(i,k)}({\bf b}) > 0, \\
            (\dots, b^{(k+2)}, \widetilde{f}_i^\ast(b^{(k+1)}), b^{(k)}, b_{(k-1)}, \dots) & \mbox{if } \widehat{\varepsilon}_{(i,k)}({\bf b}) \leq 0.
        \end{cases} \]
    \end{itemize}
\end{defn}
The extended crystal is proved to have similar properties as that of (usual) crystals. For those properties, see \cite[Section 4]{KP22}. 
If there is no confusion, we may omit the subscript $\mathfrak{g}$. 

\section{A combinatorial description of the bicrystal $\mathcal{B}(\infty)$} \label{sec:bicrystal}
\subsection{Settings}
For $(s, t) \in \mathcal{I}_n^X$, set
\[ \partb_{s,t} = x_{s,t} + \sum_{k = t+1}^n \langle h_t, \alpha_k \rangle \,x_{s,k} + \sum_{k = 1}^{t-1} \langle h_t, \alpha_k \rangle\, x_{s+1, k} + x_{s+1,t} \]
and
\[ \partb_{s,t}^\ast = x_{s-1,t} + \sum_{k = t+1}^n \langle h_t, \alpha_k \rangle \,x_{s-1,k} + \sum_{k = 1}^{t-1} \langle h_t, \alpha_k \rangle\, x_{s, k} + x_{s,t}. \]
% It is clear that $\beta_{s,t}^\ast = \beta_{s-1, t}$ for $(s, t) \in \mathcal{I}_n^X$ with $(s-1, t) \in \mathcal{I}_n^X$. 
For example, we have
\begin{eqnarray*}
    \mbox{(type $A_n$)}\quad \partb_{s,t} &=& x_{s,t} - x_{s,t+1} - x_{s+1,t-1} + x_{s+1, t}, \\
    \mbox{(type $D_n$)}\quad \partb_{s,t} &=& \begin{cases}
        x_{s,t} - x_{s, t+1} - x_{s+1, t-1} + x_{s+1,t} & \mbox{if } 1 \leq t \leq n-3, \\
        x_{s,n-2} - x_{s,n-1} - x_{s,n} - x_{s+1,n-3} + x_{s+1,n-2} & \mbox{if } t = n-2, \\
        x_{s,n-1} - x_{s+1, n-2} + x_{s+1, n-1} & \mbox{if } t = n-1, \\
        x_{s,n} - x_{s+1, n-2} + x_{s+1, n} & \mbox{if } t = n,
    \end{cases} \\
    \mbox{(type $B_n$)}\quad \partb_{s,t} &=& \begin{cases}
        x_{s,t} - x_{s, t+1} - x_{s+1, t-1} + x_{s+1,t} & \mbox{if } 1 \leq t \leq n-1, \\
        x_{s,n} - 2x_{s+1, n-1} + x_{s+1, n} & \mbox{if } t = n.
    \end{cases}
\end{eqnarray*}

Now, we define tableaux (type-dependent) $T_i$ and $T_i^\ast$ filled with some $\partb_{s,t}$ and $\partb_{s,t}^\ast$ in the following ways,
which will be used in Section \ref{subsec:bicrystal description}. 
Let $\eta_k = (k)$ be the horizontal shape partition of size $k \in \mathbb{Z}_+$, 
and $\delta_n = (n, n-1, \dots, 2, 1)$ be the stair-shape partition of length $n \in \mathbb{N}$.

\begin{defn} \label{def:partition set}\hfill

    (1) Assume $\mathfrak{g}$ is of type $A_n$. 
    For an integer $1 \leq i \leq n$, the tableau $T_i$ is of shape $\eta_{n+1-i}$ such that
    \[ T_i(1, t) = \partb_{t,i} \quad \mbox{for } 1 \leq t \leq n+1-i. \]
    On the other hand, for $1 \leq i \leq n$, the tableau $T_i^\ast$ is of shape $\eta_i$ such that 
    \[ T_i^\ast(1, t) = \partb_{t, i+1-t}^\ast \quad \mbox{for }1 \leq t \leq i. \]
    
    (2) Assume $\mathfrak{g}$ is of type $D_n$. 
    For an integer $1 \leq i \leq n$, the tableau $T_i$ is of shape $\eta_{n-1}$ such that 
    \[ T_i(1, t) = \partb_{t,i}\quad \mbox{for } 1 \leq t \leq n-1. \]
    On the other hand, for $1 \leq i \leq n-2$, the tableau $T_i^\ast$ is of shape $\eta_i$ such that 
    \[ T_i^\ast(1, t) = \partb_{t, i+1-t}^\ast\quad \mbox{for } 1 \leq t \leq i. \]
    The tableau $T_{n-1}^\ast$ is of shape $\delta_{n-1}$ such that 
    \[ T_{n-1}^\ast(s, t) = \begin{cases}
        \partb^\ast_{s+t-1, n-t} & \mbox{if } t \geq 2, \\
        \partb^\ast_{s, n-1} & \mbox{if $t=1$ and $s$ is odd}, \\
        \partb^\ast_{s, n} & \mbox{if $t=1$ and $s$ is even},
    \end{cases} \]
    and the tableau $T_n^\ast$ is of shape $\delta_{n-1}$ such that 
    \[ T_n^\ast(s, t) = \begin{cases}
        \partb^\ast_{s+t-1, n-t} & \mbox{if } t \geq 2, \\
        \partb^\ast_{s, n} & \mbox{if $t=1$ and $s$ is odd}, \\
        \partb^\ast_{s, n-1} & \mbox{if $t=1$ and $s$ is even}.
    \end{cases} \]
    Note that $T_n^\ast$ is obtained by swapping $x_{s,n-1}$ and $x_{s,n}$ from $T_{n-1}^\ast$ (and vice versa).

    (3) Assume $\mathfrak{g}$ is of type $B_n$. 
    For an integer $1 \leq i \leq n$, the tableaux $T_i$ is of shape $\eta_{n}$ such that 
    \[ T_i(1, t) = \partb_{t,i}\quad \mbox{for }1 \leq t \leq n. \]
    On the other hand, for $1 \leq i \leq n-1$, the tableau $T_i^\ast$ is of shape $\eta_i$ such that 
    \[ T_i^\ast(1, t) = \partb_{t, i+1-t}^\ast\quad \mbox{for } 1 \leq t \leq i. \]
    The tableau $T_n^\ast$ is of shape $\delta_n$ such that 
    \[ T_n^\ast(s, t) = \begin{cases}
        2\partb^\ast_{s+t-1, n-t} & \mbox{if } t \geq 2, \\
        \partb^\ast_{s, n} & \mbox{if } t = 1.
    \end{cases} \]
\end{defn}

\begin{ex}
    (1) When $\mathfrak{g}$ is of type $A_3$, the tableaux $T_i$ for $1 \leq i \leq 3$ are given as follows.
    \[ \ytableausetup{boxsize = 2em}
        T_1\quad \raisebox{-0.6em}{$\begin{ytableau}
            \partb_{1,1} & \partb_{2,1} & \partb_{3,1}
        \end{ytableau}$} \hspace{2cm}
        T_2\quad \raisebox{-0.6em}{$\begin{ytableau}
            \partb_{1,2} & \partb_{2,2}
        \end{ytableau}$} \hspace{2cm}
        T_3\quad \raisebox{-0.6em}{$\begin{ytableau}
            \partb_{1,3}
        \end{ytableau}$}
    \]
    On the other hand, the tableaux $T_i^\ast$ for $1 \leq i \leq 3$ are given as follows.
    \[ \ytableausetup{boxsize = 2em}
        T_1^\ast\quad \raisebox{-0.6em}{$\begin{ytableau}
            \partb_{1,1}^\ast
        \end{ytableau}$} \hspace{2cm}
        T_2^\ast\quad \raisebox{-0.6em}{$\begin{ytableau}
            \partb_{1,2}^\ast & \partb_{2,1}^\ast
        \end{ytableau}$} \hspace{2cm}
        T_3^\ast\quad \raisebox{-0.6em}{$\begin{ytableau}
            \partb_{1,3}^\ast & \partb_{2,2}^\ast & \partb_{3,1}^\ast
        \end{ytableau}$}
    \]
    
    (2) When $\mathfrak{g}$ is of type $D_4$, the tableaux $T_i$ for $1 \leq i \leq 4$ are given as follows.
    \[ \ytableausetup{boxsize = 2em}
        T_i\quad \raisebox{-0.6em}{$\begin{ytableau}
            \partb_{1,i} & \partb_{2,i} & \partb_{3,i}
        \end{ytableau}$}
    \]
    On the other hand, the tableaux $T_i^\ast$ for $1 \leq i \leq 4$ are given as follows.
    \[
        T_1^\ast\quad \raisebox{-0.6em}{$\begin{ytableau}
            \partb^\ast_{1,1}
        \end{ytableau}$} \hspace{1.0cm}
        T_2^\ast\quad \raisebox{-0.6em}{$\begin{ytableau}
            \partb_{1,2}^\ast & \partb_{2,1}^\ast
        \end{ytableau}$} \hspace{1.0cm}
        T_3^\ast\quad \raisebox{-1em}
        {$\begin{ytableau}
            \partb_{1,3}^\ast & \partb_{2,2}^\ast & \partb_{3,1}^\ast \\ \partb_{2,4}^\ast & \partb_{3,2}^\ast \\ \partb_{3,3}^\ast
        \end{ytableau}$} \hspace{1.0cm}
        T_4^\ast\quad \raisebox{-1em}
        {$\begin{ytableau}
            \partb_{1,4}^\ast & \partb_{2,2}^\ast & \partb_{3,1}^\ast \\ \partb_{2,3}^\ast & \partb_{3,2}^\ast \\ \partb_{3,4}^\ast
        \end{ytableau}$}
    \]

    (3) When $\mathfrak{g}$ is of type $B_3$, the tableaux $T_i$ for $1 \leq i \leq 3$ are given as follows.
    \[ \ytableausetup{boxsize = 2em}
        T_i\quad \raisebox{-0.6em}{$\begin{ytableau}
            \partb_{1,i} & \partb_{2,i} & \partb_{3,i}
        \end{ytableau}$}
    \]
    On the other hand, the tableaux $T_i^\ast$ for $1 \leq i \leq 3$ are given as follows.
    \[
        T_1^\ast\quad \raisebox{-0.6em}{$\begin{ytableau}
            \partb^\ast_{1,1}
        \end{ytableau}$} \hspace{1.5cm}
        T_2^\ast\quad \raisebox{-0.6em}{$\begin{ytableau}
            \partb_{1,2}^\ast & \partb_{2,1}^\ast
        \end{ytableau}$} \hspace{1.5cm}
        T_3^\ast \quad \raisebox{-1em}
        {\ytableausetup{boxsize = 2.5em}$\begin{ytableau}
            \partb_{1,3}^\ast & 2\partb_{2,2}^\ast & 2\partb_{3,1}^\ast \\ \partb_{2,3}^\ast & 2\partb_{3,2}^\ast \\ \partb_{3,3}^\ast
        \end{ytableau}$}
    \]
\end{ex}

Let $\Pi_i$ (resp. $\Pi_i^\ast$) be the set of non-empty partitions (resp. non-empty strict partitions) contained in the shape of $T_i$ (resp. $T_i^\ast$). 
Suppose that a tableau $T$ of shape $\tau$ is filled with $\partb_{a,b}$ and $\partb_{a,b}^\ast$ (with allowing $2\partb_{a,b}^\ast$ for type $B_n$). 
For $(s,t) \in \tau$, we denote by $\mathcal{I}_T(s,t) \in \mathcal{I}_n^X$ the two-parameter index occupying $T(s,t)$. 
In other words, $\mathcal{I}_T(s,t) = (a,b) \in \mathcal{I}_n^X$ if and only if $T(s,t) = \partb_{a,b}$ or $\partb_{a,b}^\ast$ (or $2\partb_{a,b}^\ast$ for type $B_n$).

\subsection{A description on the bicrystal structure of $B(\infty)$} \label{subsec:bicrystal description}
For $i \in I$ and $\lambda \in \Pi_i$ (cf. Definition \ref{def:partition set}) and $b \in \mathcal{B}(\infty)$, define
\begin{eqnarray*}
    \sumb_\lambda &=& \sum_{s = |\lambda|}^{|\eta|} \partb_{s,i} \left( = \sum_{s = |\lambda|}^{|\eta|} T_i(1,s) \right) \,, \\
    \sigma_i(b) &=& \max\{ \sumb_\lambda(b) : \lambda \in \Pi_i \}, \\
    m_i(b) &=& \min\{ \lambda \in \Pi_i : \sumb_\lambda(b) = \sigma_i(b) \}, \\
    M_i(b) &=& \max\{ \lambda \in \Pi_i : \sumb_\lambda(b) = \sigma_i(b) \},
\end{eqnarray*}
where $\eta$ is the shape of $T_i$, 
and put
\[ {\bf v}(s,t) = {\bf e}_{s,t} - {\bf e}_{s-1,t} \]
for $(s,t) \in \mathcal{I}_n^X$. 
\begin{defn} \label{def:crystal usual}
    Suppose a finite-dimensional Lie algebra $\mathfrak{g}$ is of type $ABD$. 
    The maps associated with the crystal $\mathcal{B}(\infty)$ are the followings. For $b = (b_{s,t}) \in \mathcal{B}(\infty)$,
    \begin{eqnarray*}
        {\rm wt}(b) &=& -\sum_{(s,t) \in \mathcal{I}_n^X} b_{s,t}\,\alpha_t, \\
        \varepsilon_i(b) &=& \sigma_i(b), \\
        \varphi_i(b) &=& \varepsilon_i(b) + \langle h_i, {\rm wt}(b) \rangle, \\
        \widetilde{e}_i(b) &=& \begin{cases} \displaystyle b - \sum_{(s,t) \in M_i} {\bf v}(\mathcal{I}_{T_i}(s,t)) & \mbox{if } \varepsilon_i(b) > 0, \\
            0 & \mbox{otherwise}, \end{cases} \\
        \widetilde{f}_i(b) &=& b + \sum_{(s,t) \in m_i} {\bf v}(\mathcal{I}_{T_i}(s,t)),
    \end{eqnarray*}
    where $m_i = m_i(b)$ and $M_i = M_i(b)$. 
\end{defn}
Note that it is the same as the maps given in \cite[Section 2.4]{NZ97}. 
In particular, we have $\displaystyle \sum_{(s,t) \in \eta_k} {\bf v}(\mathcal{I}_{T_i}(s,t)) = {\bf e}_{k,i}$ for $1 \leq k \leq i$.

\begin{rem}
    By construction, we have infinitely many trivial components $b_{s,t}$ (for $(s, t) \not \in \mathcal{I}_n^X$), so we will ignore them.
    Indeed, we can explain the reason why this argument holds using the braid-type isomorphism (cf. \cite{Nas99braid}). 
    Thus, we consider $b \in \mathcal{B}(\infty)$ as an element of $\mathbb{Z}_+^N$ (not $\mathbb{Z}_+^\infty$), 
    where $N$ is the number of positive roots for underlying Lie algebra $\mathfrak{g}$.
\end{rem}

On the other hand, we define other maps on $\mathcal{B}(\infty)$. 
First, we define linear functions
\[ \sumb_\lambda^\ast = \sum_{(s,t) \in \lambda} T_i^\ast(s,t) \]
for $i \in I$ and $\lambda \in \Pi_i^\ast$. By the case-by-case calculation, we have
\[ \{ \sumb_\lambda^\ast : \lambda \in \Pi_i^\ast \} = -\Xi^{(i)} \,(= \{ -\varphi : \varphi \in \Xi^{(i)} \}); \]
we leave it as an exercise. Recall the sequence we use is the sequence $\iota_0$ given in \eqref{eqn:iota}.
\begin{ex} \label{ex:-Xi}
    (1) When $\mathfrak{g}$ is of type $A_3$, we have the following correspondence between $-\Xi^{(3)}$ and $\Pi^\ast_{3}$.
    \[ \renewcommand{\arraystretch}{1.5}
    \begin{array}[!htp]{rcccccc}
        \Pi_3^\ast : && (1) & \subseteq & (2) & \subseteq & (3) \\
        -\Xi^{(3)} : && x_{1,3}-x_{1,2} & & x_{2,2}-x_{2,1} & & x_{3,1}
    \end{array} \]

    (2) When $\mathfrak{g}$ is of type $D_4$, we have the following correspondence between $-\Xi^{(3)}$ and $\Pi^\ast_{3}$.
    \[
        -\Xi^{(3)} = \left\{ \raisebox{-7.5em}{$\begin{tikzpicture}
            \node at (0, 0) {$x_{1,3}-x_{1,2}$};
            \node at (0, 1) {$x_{2,2}-x_{1,4}-x_{2,1}$};
            \node at (-1.5, 2) {$x_{2,4}-x_{2,1}$};
            \node at (1.5, 2) {$x_{3,1}-x_{1,4}$};
            \node at (0, 3) {$x_{2,4}+x_{3,1}-x_{2,2}$};
            \node at (0, 4) {$x_{3,2}-x_{2,3}$};
            \node at (0, 5) {$x_{3,3}$};

            \draw (0, 0.3) -- (0, 0.7);
            \draw (-0.15, 1.3) -- (-1.5, 1.7);
            \draw (0.15, 1.3) -- (1.5, 1.7);
            \draw (-1.5, 2.3) -- (-0.15, 2.7);
            \draw (1.5, 2.3) -- (0.15, 2.7);
            \draw (0, 3.3) -- (0, 3.7);
            \draw (0, 4.3) -- (0, 4.7);
        \end{tikzpicture}$} \right\} \qquad 
        \Pi_3^\ast = \left\{ \raisebox{-7.5em}{$\begin{tikzpicture}\ytableausetup{smalltableaux}
            \node at (0, 0) {$(1)$};
            \node at (0, 1) {$(2)$};
            \node at (-1, 2) {$(2,1)$};
            \node at (1, 2) {$(3)$};
            \node at (0, 3) {$(3,1)$};
            \node at (0, 4) {$(3,2)$};
            \node at (0, 5) {$(3,2,1)$};

            \draw (0, 0.3) -- (0, 0.7);
            \draw (-0.15, 1.3) -- (-1, 1.7);
            \draw (0.15, 1.3) -- (1, 1.7);
            \draw (-1, 2.3) -- (-0.15, 2.7);
            \draw (1, 2.3) -- (0.15, 2.7);
            \draw (0, 3.3) -- (0, 3.7);
            \draw (0, 4.3) -- (0, 4.7);
        \end{tikzpicture}$} \right\}
        % \Pi_3^\ast = \left\{ \raisebox{-11em}{$\begin{tikzpicture}\ytableausetup{smalltableaux}
        %     \node at (0, 0.5) {$\ydiagram{1}$};
        %     \node at (0, 1.5) {$\ydiagram{2}$};
        %     \node at (-1.5, 3) {$\ydiagram{2,1+1}$};
        %     \node at (1.5, 3) {$\ydiagram{3}$};
        %     \node at (0, 4.5) {$\ydiagram{3,1+1}$};
        %     \node at (0, 6) {$\ydiagram{3,1+2}$};
        %     \node at (0, 7.5) {$\ydiagram{3,1+2,2+1}$};
        %
        %     \draw (0, 0.8) -- (0, 1.2);
        %     \draw (-0.15, 1.8) -- (-1.5, 2.5);
        %     \draw (0.15, 1.8) -- (1.5, 2.5);
        %     \draw (-1.5, 3.5) -- (-0.15, 4.0);
        %     \draw (1.5, 3.5) -- (0.15, 4.0);
        %     \draw (0, 5.0) -- (0, 5.5);
        %     \draw (0, 6.5) -- (0, 7.0);
        % \end{tikzpicture}$} \right\}
    \]

    (3) When $\mathfrak{g}$ is of type $B_3$, we have the following correspondence between $-\Xi^{(3)}$ and $\Pi^\ast_{3}$.
    \[
        -\Xi^{(3)} = \left\{ \raisebox{-7.5em}{$\begin{tikzpicture}
            \node at (0, 0) {$x_{1,3}-2x_{1,2}$};
            \node at (0, 1) {$2x_{2,2}-x_{1,3}-2x_{2,1}$};
            \node at (-1.5, 2) {$x_{2,3}-2x_{2,1}$};
            \node at (1.5, 2) {$2x_{3,1}-x_{1,3}$};
            \node at (0, 3) {$x_{2,3}+2x_{3,1}-2x_{2,2}$};
            \node at (0, 4) {$2x_{3,2}-x_{2,3}$};
            \node at (0, 5) {$x_{3,3}$};

            \draw (0, 0.3) -- (0, 0.7);
            \draw (-0.15, 1.3) -- (-1.5, 1.7);
            \draw (0.15, 1.3) -- (1.5, 1.7);
            \draw (-1.5, 2.3) -- (-0.15, 2.7);
            \draw (1.5, 2.3) -- (0.15, 2.7);
            \draw (0, 3.3) -- (0, 3.7);
            \draw (0, 4.3) -- (0, 4.7);
        \end{tikzpicture}$} \right\} \qquad 
        \Pi_3^\ast = \left\{ \raisebox{-7.5em}{$\begin{tikzpicture}\ytableausetup{smalltableaux}
            \node at (0, 0) {$(1)$};
            \node at (0, 1) {$(2)$};
            \node at (-1, 2) {$(2,1)$};
            \node at (1, 2) {$(3)$};
            \node at (0, 3) {$(3,1)$};
            \node at (0, 4) {$(3,2)$};
            \node at (0, 5) {$(3,2,1)$};

            \draw (0, 0.3) -- (0, 0.7);
            \draw (-0.15, 1.3) -- (-1, 1.7);
            \draw (0.15, 1.3) -- (1, 1.7);
            \draw (-1, 2.3) -- (-0.15, 2.7);
            \draw (1, 2.3) -- (0.15, 2.7);
            \draw (0, 3.3) -- (0, 3.7);
            \draw (0, 4.3) -- (0, 4.7);
        \end{tikzpicture}$} \right\}
        % \Pi_3^\ast = \left\{ \raisebox{-11em}{$\begin{tikzpicture}\ytableausetup{smalltableaux}
        %     \node at (0, 0.5) {$\ydiagram{1}$};
        %     \node at (0, 1.5) {$\ydiagram{2}$};
        %     \node at (-1.5, 3) {$\ydiagram{2,1+1}$};
        %     \node at (1.5, 3) {$\ydiagram{3}$};
        %     \node at (0, 4.5) {$\ydiagram{3,1+1}$};
        %     \node at (0, 6) {$\ydiagram{3,1+2}$};
        %     \node at (0, 7.5) {$\ydiagram{3,1+2,2+1}$};
        %     \draw (0, 0.8) -- (0, 1.2);
        %     \draw (-0.15, 1.8) -- (-1.5, 2.5);
        %     \draw (0.15, 1.8) -- (1.5, 2.5);
        %     \draw (-1.5, 3.5) -- (-0.15, 4.0);
        %     \draw (1.5, 3.5) -- (0.15, 4.0);
        %     \draw (0, 5.0) -- (0, 5.5);
        %     \draw (0, 6.5) -- (0, 7.0);
        % \end{tikzpicture}$} \right\}
    \]
\end{ex}

\begin{rem}

    (1) Whereas the partial order on $\Pi_i^\ast$ is obvious, that on $-\Xi^{(i)}$ is not. 
    The partial order on $-\Xi^{(i)}$ is defined using operators $S_{s,t}$ on $(\mathbb{Q}^\infty)^\ast$ (cf. \cite[(3.3)]{NZ97}). 

    (2) For a linear function $f \in -\Xi^{(i)}$, the partition $\lambda$ such that $\sumb_\lambda^\ast = f$ is determined by coefficients of $f$. 
    To be precise, suppose $A$ is the set of two-parameter indices $(s,t)$ such that the coefficient of $x_{s,t}$ in $f$ is positive. 
    Then $\lambda$ is the partition whose outer corners are filled with $x_{s,t}$'s for $(s,t) \in A$.
\end{rem}

For $i \in I$ and $b \in \mathcal{B}(\infty)$, define
\begin{eqnarray*}
    \sigma_i^\ast(b) &=& \max\{ \sumb_\lambda^\ast(b) : \lambda \in \Pi^\ast_{i} \}, \\
    m_i^\ast(b) &=& \min\{ \lambda \in \Pi_i^\ast : \sumb_\lambda^\ast(b) = \sigma_i^\ast(b) \},\\
    M_i^\ast(b) &=& \max\{ \lambda \in \Pi_i^\ast : \sumb_\lambda^\ast(b) = \sigma_i^\ast(b) \}.
\end{eqnarray*}
By Theorem \ref{thm:epsilon}, we have $\varepsilon_i^\ast(b) = \sigma_i^\ast(b)$ for $b \in \mathcal{B}(\infty)$. 
The set $\Pi_i^\ast$ may not be totally ordered (cf. Example \ref{ex:-Xi}), and it may cause $m_i^\ast(b)$ or $M_i^\ast(b)$ not to be well-defined. 
But the following lemma guarantees the well-definedness of $m_i^\ast(b)$ and $M_i^\ast(b)$.
\begin{lem}
    For $b \in \mathcal{B}(\infty)$, suppose that $\sigma_i^\ast(b) = \sumb_\lambda^\ast(b) = \sumb_\mu^\ast(b)$ for some $\lambda, \mu \in \Pi_{i}^\ast$. 
    Then we have $\sigma_i^\ast(b) = \sumb_{\lambda\cup\mu}^\ast(b) = \sumb_{\lambda \cap \mu}^\ast(b)$.
\end{lem}
\begin{proof}
    By definition (or the inclusion-exclusion principle), we easily check that
    \[ \sumb_{\lambda\cap\mu}^\ast(b) + \sumb_{\lambda\cup\mu}^\ast(b) = \sumb_\lambda^\ast(b) + \sumb_\mu^\ast(b). \]
    Since $\sigma_i(b)$ is the maximal value among $\sumb_\tau^\ast(b)$ for $\tau \in \Pi_i^\ast$, 
    the left-hand side of the equation is less than or equal to its right-hand side.
    Thus, the fact that the equality holds implies $\sumb_{\lambda\cap\mu}^\ast(b) = \sumb_{\lambda\cup\mu}^\ast(b) = \sigma_i^\ast(b)$. 
\end{proof}
As a corollary, we set $m_i^\ast(b)$ and $M_i^\ast(b)$ to be the intersection and union of partitions $\tau \in \Pi_i^\ast$ such that $\sumb_\tau^\ast(b) = \sigma_i^\ast(b)$, respectively.

\begin{defn} \label{def:crystal star}
    Suppose a finite-dimensional Lie algebra $\mathfrak{g}$ is of type $ABD$. 
    Define the (crystal) maps on $\mathcal{B}(\infty)$ as follows. For $b = (b_{s,t}) \in \mathcal{B}(\infty)$,
    \begin{eqnarray*}
        {\rm wt}^\ast(b) &=& {\rm wt}(b), \\
        \varepsilon_i^\ast(b) &=& \sigma_i^\ast(b), \\
        \varphi_i^\ast(b) &=& \varepsilon_i^\ast(b) + \langle h_i, {\rm wt}(b) \rangle, \\
        \widetilde{e}_i^\ast(b) &=& \begin{cases}
            \displaystyle b - \sum_{(s, t) \in M_i^\ast} {\bf v}(\mathcal{I}_{T_i^\ast}(s,t)) & \mbox{if } \varepsilon_i^\ast(b) >0, \\
            0 & \mbox{otherwise}, \end{cases} \\
        \widetilde{f}_i^\ast(b) &=& b + \sum_{(s, t) \in m_i^\ast} {\bf v}(\mathcal{I}_{T_i^\ast}(s, t)), \\
    \end{eqnarray*}
    where $m_i^\ast = m_i^\ast(b)$ and $M_i^\ast = M_i^\ast(b)$. 
\end{defn}
We refer Definition \ref{def:crystal star} to the additional crystal structure $({\rm wt}^\ast, \widetilde{e}_i^\ast, \widetilde{f}_i^\ast, \varepsilon_i^\ast, \varphi_i^\ast)$ (which will be explained later).

\begin{ex} \label{ex:A3 bicrystal}
    When $\mathfrak{g}$ is of type $A_3$, take an element $b = (2,4,0,5,1,3) \in \mathcal{B}(\infty)$, that is, 
    \[ (b_{3,1},\, b_{2,2},\, b_{2,1},\, b_{1,3},\, b_{1,2},\, b_{1,1}) = (2,4,0,5,1,3). \]
    Then we directly check $\partb_{1,1}(b) = 2$, $\partb_{2,1}(b) = -2$, and $\partb_{3,1}(b) = 2$, 
    which implies that $\sumb_{(1)}(b) = 2, \sumb_{(2)}(b) = 0, \sumb_{(3)}(b) = 2$, and $m_1 = (1)$ and $M_1 = (3)$. Hence, we know
    \begin{eqnarray*}
        \widetilde{e}_1(b) &=& b - \sum_{(s,t) \in (3)} {\bf v}(\mathcal{I}_{T_1}(s,t)) \\
            &=& (1,4,0,5,1,3) \ (=b-{\bf e}_{3,1}), \\
        \widetilde{f}_1(b) &=& b + \sum_{(s,t) \in (1)} {\bf v}(\mathcal{I}_{T_1}(s,t)) \\
            &=& (2,4,0,5,1,4) \ (=b+{\bf e}_{1,1}).
    \end{eqnarray*}
    Note that we obtain
    \begin{eqnarray*}
        && \widetilde{e}_2(b) = (2,3,0,5,1,3), \quad \widetilde{f}_2(b) = (2,4,0,5,2,3), \\
        && \widetilde{e}_3(b) = (2,4,0,4,1,3), \quad \widetilde{f}_3(b) = (2,4,0,6,1,3).
    \end{eqnarray*}
    % \begin{itemize}
    %     \item $\beta_{11} = 1, \beta_{21} = 1, \beta_{31} = 0$ 
    %     $\Rightarrow$ $B_{11} = 2, B_{21} = 1, B_{31} = 0$ : $m_1(b) = M_1(b) = 1$ \\
    %     $\therefore$ $\widetilde{e}_1(b) = (0,0,1,1,2,{\color{red}1})$, $\widetilde{f}_1(b) = (0,0,1,1,2,{\color{red}3})$
    %     \item $\beta_{12} = 0, \beta_{22} = 0$ 
    %     $\Rightarrow$ $B_{12} = 0, B_{22} = 0$ : $m_2(b) = 1$, $M_2(b) = 2$ \\
    %     $\therefore$ $\widetilde{e}_2(b) = 0$, $\widetilde{f}_2(b) = (0,0,1,1,{\color{red}3},2)$
    %     \item $\beta_{13} = 1$ $\Rightarrow$ $B_{13} = 1$ : $m_3(b) = M_3(b) = 1$ \\
    %     $\therefore$ $\widetilde{e}_3(b) = (0,0,1,{\color{red}0},2,2)$, $\widetilde{f}_3(b) = (0,0,1,{\color{red}2},2,2)$
    % \end{itemize}
    
    On the other hand, we have $\partb_{1,3}^\ast(b) = 4$, $\partb_{2,2}^\ast(b) = 0$, and $\partb_{3,1}^\ast(b) = -2$, 
    which implies that $\sumb_{(1)}^\ast(b) = 4, \sumb_{(2)}^\ast(b) = 4, \sumb_{(3)}^\ast(b) = 2$, and $m_3^\ast = (1)$ and $M_3^\ast = (2)$. Hence, we know
    \begin{eqnarray*}\
        \widetilde{e}_3^\ast(b) &=& b - \sum_{(s,t) \in (2)} {\bf v}(\mathcal{I}_{T_3^\ast}(s,t)) \\
            &=& (2,3,0,4,2,3) \ (=b - ({\bf e}_{1,3}) - ({\bf e}_{2,2} - {\bf e}_{1,2})), \\
        \widetilde{f}_3^\ast(b) &=& b + \sum_{(s,t) \in (1)} {\bf v}(\mathcal{I}_{T_3^\ast}(s,t)) \\
            &=& (2,4,0,6,1,3) \ (=b + {\bf e}_{1,3}).
    \end{eqnarray*}
    Note that we obtain
    \begin{eqnarray*}
        && \widetilde{e}_1^\ast(b) = (2,4,0,5,1,2), \quad \widetilde{f}_1^\ast(b) = (2,4,0,5,1,4), \\
        && \widetilde{e}_2^\ast(b) = {\bf 0}, \quad \widetilde{f}_2^\ast(b) = (2,4,1,5,2,2).
    \end{eqnarray*}
    % \begin{itemize}
    %     \item $\beta_{11}^\ast = 2$ $\Rightarrow$ $B_{11}^\ast = 2$ : $m_1^\ast(b) = M_1^\ast(b) = 1$ \\
    %     $\therefore$ $\widetilde{e}_1^\ast(b) = (0,0,1,1,2,{\color{red}1})$, $\widetilde{f}_1^\ast(b) = (0,0,1,1,2,{\color{red}3})$
    %     \item $\beta_{12}^\ast = 0, \beta_{21}^\ast = 1$ 
    %     $\Rightarrow$ $B_{12}^\ast = 0, B_{21}^\ast = 1$ : $m_2^\ast(b) = M_2^\ast(b) = 2$ \\
    %     $\therefore$ $\widetilde{e}_2(b) = (0,0,{\color{blue}0},1,{\color{red}1},{\color{blue}3})$, $\widetilde{f}_2(b) = (0,0,{\color{blue}2},1,{\color{red}3},{\color{blue}1})$
    %     \item $\beta_{13}^\ast = -1, \beta_{22}^\ast = 0, \beta_{31}^\ast = 1$ 
    %     $\Rightarrow$ $B_{13}^\ast = -1, B_{22}^\ast = -1, B_{31}^\ast = 0$ : $m_3(b) = M_3(b) = 3$ \\
    %     $\therefore$ $\widetilde{e}_3^\ast(b) = 0$, $\widetilde{f}_3(b) = ({\color{green}1},{\color{blue}1},{\color{green}0},{\color{red}2},{\color{blue}1},2)$
    % \end{itemize}
\end{ex}

\begin{ex} \label{ex:D4 bicrystal}
    When $\mathfrak{g}$ is of type $D_4$, take an element $b = (0,0,0,2,0,1,3,0,2,1,0,0) \in \mathcal{B}(\infty)$, that is, 
    \[ (b_{3,4}, b_{3,3}, b_{3,2}, b_{3,1}, b_{2,4}, b_{2,3}, b_{2,2}, b_{2,1}, b_{1,4}, b_{1,3}, b_{1,2}, b_{1,1}) = (0,0,0,2,0,1,3,0,2,1,0,0). \]
    Then we directly check $\partb_{1,1}(b) = 0$, $\partb_{2,1}(b) = -1$, and $\partb_{3,1}(b) = 2$, 
    which implies that $\sumb_{(1)}(b) = 1, \sumb_{(2)}(b) = 1, \sumb_{(3)}(b) = 2$, and $m_1 = (3)$ and $M_1 = (3)$. Hence, we know
    \begin{eqnarray*}
        \widetilde{e}_1(b) &=& b - \sum_{(s,t) \in (3)} {\bf v}(\mathcal{I}_{T_1}(s,t)) \\
            &=& (0,0,0,1,0,1,3,0,2,1,0,0) \ (=b-{\bf e}_{3,1}), \\
        \widetilde{f}_1(b) &=& b + \sum_{(s,t) \in (3)} {\bf v}(\mathcal{I}_{T_1}(s,t)) \\
            &=& (0,0,0,3,0,1,3,0,2,1,0,0) \ (=b+{\bf e}_{3,1}).
    \end{eqnarray*}
    Note that we obtain
    \begin{eqnarray*}
        && \widetilde{e}_2(b) = {\bf 0}, \quad \widetilde{f}_2(b) = (0,0,0,2,0,1,3,0,2,1,1,0), \\
        && \widetilde{e}_3(b) = (0,0,0,2,0,0,3,0,2,1,0,0), \quad \widetilde{f}_3(b) = (0,0,0,2,0,2,3,0,2,1,0,0), \\
        && \widetilde{e}_4(b) = {\bf 0}, \quad \widetilde{f}_4(b) = (0,0,0,2,1,1,3,0,2,1,0,0).
    \end{eqnarray*}
    % \begin{itemize}
    %     \item $\beta_{11} = 1, \beta_{21} = 1, \beta_{31} = 0$ 
    %     $\Rightarrow$ $B_{11} = 2, B_{21} = 1, B_{31} = 0$ : $m_1(b) = M_1(b) = 1$ \\
    %     $\therefore$ $\widetilde{e}_1(b) = (0,0,1,1,2,{\color{red}1})$, $\widetilde{f}_1(b) = (0,0,1,1,2,{\color{red}3})$
    %     \item $\beta_{12} = 0, \beta_{22} = 0$ 
    %     $\Rightarrow$ $B_{12} = 0, B_{22} = 0$ : $m_2(b) = 1$, $M_2(b) = 2$ \\
    %     $\therefore$ $\widetilde{e}_2(b) = 0$, $\widetilde{f}_2(b) = (0,0,1,1,{\color{red}3},2)$
    %     \item $\beta_{13} = 1$ $\Rightarrow$ $B_{13} = 1$ : $m_3(b) = M_3(b) = 1$ \\
    %     $\therefore$ $\widetilde{e}_3(b) = (0,0,1,{\color{red}0},2,2)$, $\widetilde{f}_3(b) = (0,0,1,{\color{red}2},2,2)$
    % \end{itemize}
    
    On the other hand, we have
    \[ \partb_{1,3}^\ast(b) = 1, \, \partb_{2,2}^\ast(b) = 0, \, \partb_{3,1}^\ast(b) = -1, \,\partb_{2,4}^\ast(b) = -1,\, \partb_{3,2}^\ast(b) = 0,\, \partb_{3,3}^\ast(b) = 1, \]
    which implies that
    \begin{eqnarray*}
        && \sumb_{(1)}^\ast(b) = 1, \quad \sumb_{(2)}^\ast(b) = 1, \quad \sumb_{(3)}^\ast(b) = 0, \\
        && \sumb_{(2,1)}^\ast(b) = 0, \quad \sumb_{(3,1)}^\ast(b) = -1, \quad \sumb_{(3,2)}^\ast(b) = -1, \quad \sumb_{(3,2,1)}^\ast(b) = 0,
    \end{eqnarray*}
    and $m_3^\ast = (1)$ and $M_3^\ast = (2)$. Hence, we know
    \begin{eqnarray*}
        \widetilde{e}_3^\ast(b) &=& b - \sum_{(s,t) \in (2)} {\bf v}(\mathcal{I}_{T_3^\ast}(s,t)) \\
            &=& (0,0,0,2,0,1,2,0,2,0,1,0) (=b - ({\bf e}_{1,3}) - ({\bf e}_{2,2} - {\bf e}_{1,2})), \\
        \widetilde{f}_3^\ast(b) &=& b + \sum_{(s,t) \in (1)} {\bf v}(\mathcal{I}_{T_3^\ast}(s,t)) \\
            &=& (0,0,0,2,0,1,3,0,2,2,0,0) (=b + {\bf e}_{1,3}).
    \end{eqnarray*}
    Note that we obtain
    \begin{eqnarray*}
        && \widetilde{e}_1^\ast(b) = {\bf 0}, \quad \widetilde{f}_1^\ast(b) = (0,0,0,2,0,1,3,0,2,1,0,1), \\
        && \widetilde{e}_2^\ast(b) = {\bf 0}, \quad \widetilde{f}_2^\ast(b) = (0,0,0,2,0,1,3,0,2,1,1,0), \\
        && \widetilde{e}_4^\ast(b) = (0,0,0,2,0,1,2,0,1,1,1,0), \quad \widetilde{f}_4^\ast(b) = (0,0,0,2,0,1,3,0,3,1,0,0).
    \end{eqnarray*}
    % \begin{itemize}
    %     \item $\beta_{11}^\ast = 2$ $\Rightarrow$ $B_{11}^\ast = 2$ : $m_1^\ast(b) = M_1^\ast(b) = 1$ \\
    %     $\therefore$ $\widetilde{e}_1^\ast(b) = (0,0,1,1,2,{\color{red}1})$, $\widetilde{f}_1^\ast(b) = (0,0,1,1,2,{\color{red}3})$
    %     \item $\beta_{12}^\ast = 0, \beta_{21}^\ast = 1$ 
    %     $\Rightarrow$ $B_{12}^\ast = 0, B_{21}^\ast = 1$ : $m_2^\ast(b) = M_2^\ast(b) = 2$ \\
    %     $\therefore$ $\widetilde{e}_2(b) = (0,0,{\color{blue}0},1,{\color{red}1},{\color{blue}3})$, $\widetilde{f}_2(b) = (0,0,{\color{blue}2},1,{\color{red}3},{\color{blue}1})$
    %     \item $\beta_{13}^\ast = -1, \beta_{22}^\ast = 0, \beta_{31}^\ast = 1$ 
    %     $\Rightarrow$ $B_{13}^\ast = -1, B_{22}^\ast = -1, B_{31}^\ast = 0$ : $m_3(b) = M_3(b) = 3$ \\
    %     $\therefore$ $\widetilde{e}_3^\ast(b) = 0$, $\widetilde{f}_3(b) = ({\color{green}1},{\color{blue}1},{\color{green}0},{\color{red}2},{\color{blue}1},2)$
    % \end{itemize}
\end{ex}

Now, we are ready to state one of main results in this paper. 
\begin{thm} \label{thm:bicrystal}
    Assume that $\mathfrak{g}$ is of finite $ABD$ type. 
    The set $\mathcal{B}(\infty)$ with two collections of maps appearing in Definition \ref{def:crystal usual} and Definition \ref{def:crystal star} 
    forms a bicrystal and is isomorphic to the Kashiwara's bicrystal $B(\infty)$ as bicrystals.
\end{thm}

The following proposition is the key idea to prove Theorem \ref{thm:bicrystal}, which characterizes $B(\infty)$. 
Note that it has many equivalent statements, which are essentially based on \cite[Proposition 3.2.3]{KS97}, and we use the one in \cite{CT15}. 
We remark (cf. \cite[Definition 2.4]{CT15}) that a highest weight crystal is a crystal $B$ with the highest weight element $b_0$ such that
\begin{itemize}
    \item The highest weight element $b_0$ can be reached from any $b \in B$ by applying a sequence of $\widetilde{e}_i$ for various $i \in I$.
    \item For all $b \in B$ and $i \in I$, we have $\varepsilon_i(b) = \max\{ n \geq 0 : \widetilde{e}_i^n(b) \neq 0 \}$.
\end{itemize}
\begin{prop}[{\cite[Proposition 2.6]{CT15}}] \label{prop:bicrystal}
    Suppose that a set $B$ has two highest weight crystal structures with respect to $({\rm wt}, \varepsilon_i, \varphi_i, \widetilde{e}_i, \widetilde{f}_i)$ and $({\rm wt}, \varepsilon_i^\ast, \varphi_i^\ast, \widetilde{e}_i^\ast, \widetilde{f}_i^\ast)$. 
    Additionally, assume that $B$ satisfies the following conditions. For $i \neq j \in I$ and $b \in B$,
    \begin{enumerate}
        \item $\widetilde{f}_i(b) \neq 0$, $\widetilde{f}_i^\ast(b) \neq 0$.
        \item $B$ has the unique highest weight vector $b_0$ with weight $0$, that is, $\widetilde{e}_i(b_0) = \widetilde{e}_i^\ast(b_0) = 0$ for all $i \in I$.
        \item $\widetilde{f}_i\widetilde{f}_j^\ast(b) = \widetilde{f}_j^\ast\widetilde{f}_i(b)$.
        \item $\varepsilon_i(b) + \varepsilon_i^\ast(b) + \langle h_i, {\rm wt}(b) \rangle \geq 0$.
        \item if $\varepsilon_i(b) + \varepsilon_i^\ast(b) + \langle h_i, {\rm wt}(b) \rangle = 0$, then $\widetilde{f}_i(b) = \widetilde{f}_i^\ast(b)$.
        \item if $\varepsilon_i(b) + \varepsilon_i^\ast(b) + \langle h_i, {\rm wt}(b) \rangle \geq 1$, then $\varepsilon_i^\ast(\widetilde{f}_i(b)) = \varepsilon_i^\ast(b)$ and $\varepsilon_i(\widetilde{f}_i^\ast(b)) = \varepsilon_i(b)$.
        \item if $\varepsilon_i(b) + \varepsilon_i^\ast(b) + \langle h_i, {\rm wt}(b) \rangle \geq 2$, then $\widetilde{f}_i^\ast \widetilde{f}_i(b) = \widetilde{f}_i \widetilde{f}_i^\ast(b)$.
    \end{enumerate}
    Then $B$ is isomorphic to Kashiwara's crystal $B(\infty)$ as bicrystals.
\end{prop}
Thus, we complete the proof of Theorem \ref{thm:bicrystal} by checking all above conditions. 
The proof will be given in Section \ref{sec:bicrystal proof}.

\subsection{A sliding diamond rule}
To easily compute values in Definition \ref{def:crystal usual} and Definition \ref{def:crystal star}, 
we introduce a combinatorial rule for $\mathcal{B}(\infty)$, which we call a sliding diamond rule. 
Briefly speaking, we will record each component of $\mathcal{B}(\infty)$ under some rules and read some values from it. 
For this, we define the following diamonds (a parallelogram or its slanted one).
    \begin{itemize}
        \item $\Diamond_{s,t}$ : the diamond containing $b_{s,t}, b_{s+1, t}$, $b_{s,k}$ for $k > t$ with $\langle h_k, \alpha_t \rangle \neq 0$,
        and $b_{s+1, k}$ for $k < t$ with $\langle h_k, \alpha_t \rangle \neq 0$
        \item $\Diamond_{s,t}^\ast$ : the diamond containing $b_{s,t}, b_{s-1,t}$, $b_{s-1,k}$ for $k > t$ with $\langle h_k, \alpha_t \rangle \neq 0$, 
        and $b_{s, k}$ for $k < t$ with $\langle h_k, \alpha_t \rangle \neq 0$
    \end{itemize}
In this case, we write $(u, v) \in \Diamond_{s,t}$ or $(u, v) \in \Diamond_{s,t}^\ast$ if $b_{u, v}$ is contained in the diamond.
\begin{defn}[The sliding diamond rule] \label{def:SD}
    Take $b = (b_{s,t}) \in \mathcal{B}(\infty)$. 
    \begin{enumerate}
        \item Put $b_{s,t}$ in a plane $\mathbb{Z}^2$ as follows;
        \begin{enumerate}
            \item put $b_{s,t}$ at the point $(2s+t-3, t)$ for $(s, t) \in \mathcal{I}_n^A$,
            \item put $b_{s,t}$ at the point $(2s+t-3, t)$ for $(s, t) \in \mathcal{I}_n^B$,
            \item put $b_{s,t}$ at the point $(2s+t-3, t)$ for $(s, t) \in \mathcal{I}_n^D$ with $t \neq n$, 
            $b_{s,n}$ at the point $(2s+n-4, n-2)$ for $(s, n) \in \mathcal{I}_n^D$.
        \end{enumerate}
        \item Consider $b_{u, v}$ contained in $\Diamond_{s,t}$ and $\Diamond_{s,t}^\ast$, and assign the coefficient $\epsilon_{u,v}$ to each $b_{u,v}$ as follows;
        for any $u$, 
        \[ \epsilon_{u,v} = \begin{cases}
            1 & \mbox{if }v = t, \\
            \langle h_t, \alpha_v \rangle & \mbox{otherwise.}
        \end{cases} \]
        \item Obtain the $\epsilon$-weighted sum of $b_{u,v}$ contained in $\Diamond_{s,t}$ or $\Diamond_{s,t}^\ast$, 
        which results in
        \[ \sum_{(u, v) \in \Diamond_{s,t}} \epsilon_{u,v}b_{u,v} = \partb_{s,t}(b), \quad \sum_{(u, v) \in \Diamond_{s,t}^\ast} \epsilon_{u,v}b_{u,v} = \partb_{s,t}^\ast(b). \]
    \end{enumerate}
\end{defn}
Note that we might place $b_{u,v}$ in other positions by shifting, rotating, or flipping.
\begin{ex}[The configuration of $b_{s,t}$] \label{ex:configuration} \hfill

    (1) The left triangle is the configuration of $b_{s,t}$ for type $A_3$ given in Definition \ref{def:SD}.
    
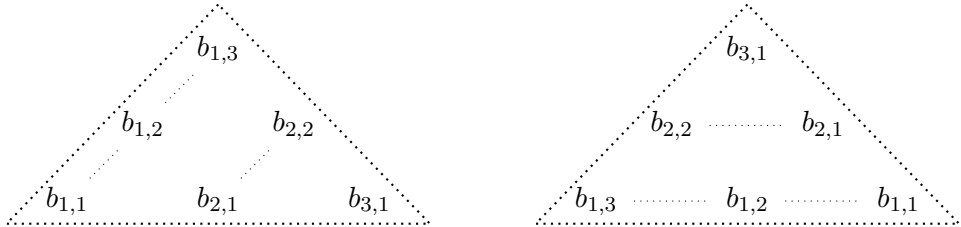
    \[ \begin{tikzpicture}
        \node at (0,1) {\large$b_{1,1}$};
        \node at (2,1) {\large$b_{2,1}$};
        \node at (4,1) {\large$b_{3,1}$};
        \node at (1,2) {\large$b_{1,2}$};
        \node at (3,2) {\large$b_{2,2}$};
        \node at (2,3) {\large$b_{1,3}$};
        \draw[dotted, thick] (2, 3.6) -- (-0.8, 0.7) -- (4.8, 0.7) -- (2, 3.6);
        \draw[dotted] (0.3, 1.3) -- (0.7, 1.7);
        \draw[dotted] (1.3, 2.3) -- (1.7, 2.7);
        \draw[dotted] (2.3, 1.3) -- (2.7, 1.7);

        \node[shift = {(7,0)}] at (0,1) {\large$b_{1,3}$};
        \node[shift = {(7,0)}] at (2,1) {\large$b_{1,2}$};
        \node[shift = {(7,0)}] at (4,1) {\large$b_{1,1}$};
        \node[shift = {(7,0)}] at (1,2) {\large$b_{2,2}$};
        \node[shift = {(7,0)}] at (3,2) {\large$b_{2,1}$};
        \node[shift = {(7,0)}] at (2,3) {\large$b_{3,1}$};
        \draw[dotted, shift = {(7,0)}, thick] (2, 3.6) -- (-0.8, 0.7) -- (4.8, 0.7) -- (2, 3.6);
        \draw[dotted] (7.5, 1) -- (8.5, 1);
        \draw[dotted] (9.5, 1) -- (10.5, 1);
        \draw[dotted] (8.5, 2) -- (9.5, 2);
    \end{tikzpicture} \]
    Note that the right triangle is obtained by rotating the left one. This configuration is useful in Section \ref{subsec:extended diamond}.
    
    (2) The left diagram is the configuration of $b_{s,t}$ for type $D_4$. 
    The right diagram is the configuration of $b_{s,t}$ for type $B_3$.
    
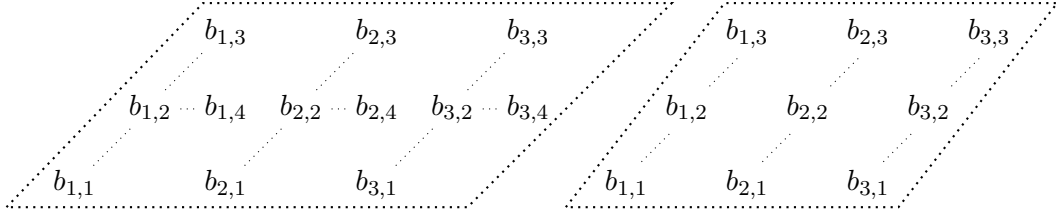
\[ \begin{tikzpicture}
        \node at (0, 1) {\large$b_{1,1}$};
        \node at (2, 1) {\large$b_{2,1}$};
        \node at (4, 1) {\large$b_{3,1}$};
        \node at (1, 2) {\large$b_{1,2}$};
        \node at (3, 2) {\large$b_{2,2}$};
        \node at (5, 2) {\large$b_{3,2}$};
        \node at (2, 3) {\large$b_{1,3}$};
        \node at (4, 3) {\large$b_{2,3}$};
        \node at (6, 3) {\large$b_{3,3}$};
        \node at (2, 2) {\large$b_{1,4}$};
        \node at (4, 2) {\large$b_{2,4}$};
        \node at (6, 2) {\large$b_{3,4}$};

        \draw[thick, dotted] (-0.9, 0.7) -- (1.7, 3.4) -- (7.9, 3.4) -- (5.2, 0.7) -- (-0.9, 0.7);
        \draw[dotted] (0.25, 1.25) -- (0.75, 1.75);
        \draw[dotted] (1.2, 2.25) -- (1.7, 2.75);
        \draw[dotted] (1.4, 2) -- (1.6, 2);
        \draw[dotted] (2.25, 1.25) -- (2.75, 1.75);
        \draw[dotted] (3.2, 2.25) -- (3.7, 2.75);
        \draw[dotted] (3.4, 2) -- (3.6, 2);
        \draw[dotted] (4.25, 1.25) -- (4.75, 1.75);
        \draw[dotted] (5.2, 2.25) -- (5.7, 2.75);
        \draw[dotted] (5.4, 2) -- (5.6, 2);

        \node[shift = {(7.3,0)}] at (0, 1) {\large$b_{1,1}$};
        \node[shift = {(7.3,0)}] at (1.6, 1) {\large$b_{2,1}$};
        \node[shift = {(7.3,0)}] at (3.2, 1) {\large$b_{3,1}$};
        \node[shift = {(7.3,0)}] at (0.8, 2) {\large$b_{1,2}$};
        \node[shift = {(7.3,0)}] at (2.4, 2) {\large$b_{2,2}$};
        \node[shift = {(7.3,0)}] at (4.0, 2) {\large$b_{3,2}$};
        \node[shift = {(7.3,0)}] at (1.6, 3) {\large$b_{1,3}$};
        \node[shift = {(7.3,0)}] at (3.2, 3) {\large$b_{2,3}$};
        \node[shift = {(7.3,0)}] at (4.8, 3) {\large$b_{3,3}$};

        \draw[dotted, thick, shift = {(7.3, 0)}] (-0.8, 0.7) -- (1.3, 3.4) -- (5.7, 3.4) -- (3.6, 0.7) -- (-0.8, 0.7);
        \draw[dotted, shift = {(7.3, 0)}] (0.2, 1.3) -- (0.6, 1.7);
        \draw[dotted, shift = {(7.3, 0)}] (1.1, 2.3) -- (1.5, 2.7);
        \draw[dotted, shift = {(7.3, 0)}] (1.8, 1.3) -- (2.2, 1.7);
        \draw[dotted, shift = {(7.3, 0)}] (2.7, 2.3) -- (3.1, 2.7);
        \draw[dotted, shift = {(7.3, 0)}] (3.4, 1.3) -- (3.8, 1.7);
        \draw[dotted, shift = {(7.3, 0)}] (4.3, 2.3) -- (4.7, 2.7);
    \end{tikzpicture} \]
\end{ex}

As we can calculate $\partb_{s,t}(b)$ and $\partb^\ast_{s,t}(b)$ (and $\sumb_\lambda(b)$ and $\sumb_\mu^\ast(b)$) from the sliding diamond rule, 
we can obtain the maps as given in Definition \ref{def:crystal usual} and Definition \ref{def:crystal star}. 
\begin{ex} \label{ex:A3 diamond}
    Let $b \in \mathcal{B}(\infty)$ be the element given in Example \ref{ex:A3 bicrystal} 
    with the following configuration of $b_{s,t}$ (see Example \ref{ex:configuration}~(1)).
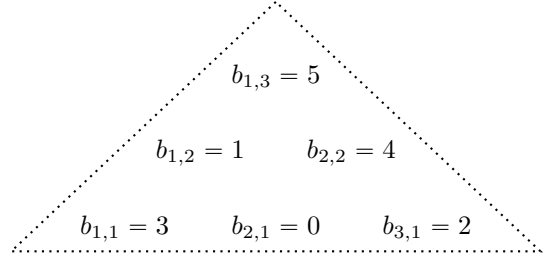
    \[ \begin{tikzpicture}
        \node at (0,1) {$b_{1,1} = 3$};
        \node at (2,1) {$b_{2,1} = 0$};
        \node at (4,1) {$b_{3,1} = 2$};
        \node at (1,2) {$b_{1,2} = 1$};
        \node at (3,2) {$b_{2,2} = 4$};
        \node at (2,3) {$b_{1,3} = 5$};
        \draw[dotted, thick] (2, 4) -- (-1.5, 0.7) -- (5.5, 0.7) -- (2, 4);
    \end{tikzpicture} \]
    
    To find values $\sumb_\lambda(b)$ for $\lambda \in \Pi_1$, consider diamonds $\Diamond_{i, 1}$ for $1 \leq i \leq 3$.
    Note that we record $\epsilon_{u,v}b_{u,v}$ instead of $b_{u,v}$ in each diamond.
    
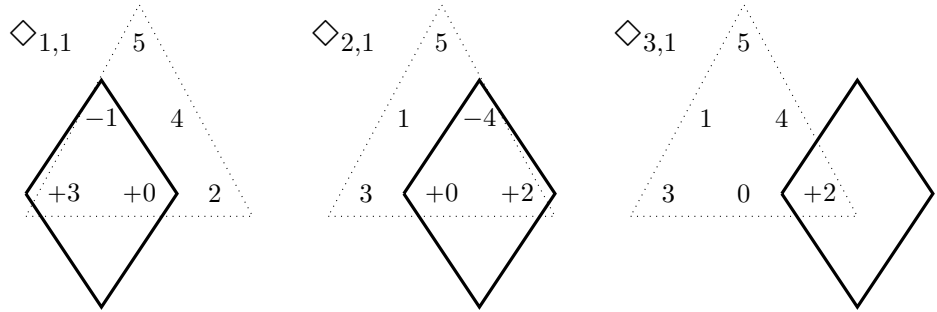
\[ \begin{tikzpicture}
        \node at (2, 0) {$2$};
        \node at (1, 0) {$+0$};
        \node at (0, 0) {$+3$};
        \node at (1.5, 1) {$4$};
        \node at (0.5, 1) {$-1$};
        \node at (1, 2) {$5$};
    
        \draw[dotted] (-0.5, -0.3) -- (2.5, -0.3) -- (1, 2.5) -- (-0.5, -0.3);
        \draw[very thick] (0-0.5, 0) -- (0.5, 1.0+0.5) -- (1.0+0.5, 0) -- (0.5, -1-0.5) -- (0-0.5, 0);
        \node at (-0.3, 2) {\LARGE$\Diamond_{1,1}$};

        \node[shift = {(4, 0)}] at (2, 0) {$+2$};
        \node[shift = {(4, 0)}] at (1, 0) {$+0$};
        \node[shift = {(4, 0)}] at (0, 0) {$3$};
        \node[shift = {(4, 0)}] at (1.5, 1) {$-4$};
        \node[shift = {(4, 0)}] at (0.5, 1) {$1$};
        \node[shift = {(4, 0)}] at (1, 2) {$5$};
    
        \draw[dotted, shift = {(4,0)}] (-0.5, -0.3) -- (2.5, -0.3) -- (1, 2.5) -- (-0.5, -0.3);
        \draw[very thick, shift = {(5, 0)}] (0-0.5, 0) -- (0.5, 1.0+0.5) -- (1.0+0.5, 0) -- (0.5, -1-0.5) -- (0-0.5, 0);
        \node at (4-0.3, 2) {\LARGE$\Diamond_{2,1}$};

        \node[shift = {(8, 0)}] at (2, 0) {$+2$};
        \node[shift = {(8, 0)}] at (1, 0) {$0$};
        \node[shift = {(8, 0)}] at (0, 0) {$3$};
        \node[shift = {(8, 0)}] at (1.5, 1) {$4$};
        \node[shift = {(8, 0)}] at (0.5, 1) {$1$};
        \node[shift = {(8, 0)}] at (1, 2) {$5$};
    
        \draw[dotted, shift = {(8,0)}] (-0.5, -0.3) -- (2.5, -0.3) -- (1, 2.5) -- (-0.5, -0.3);
        \draw[very thick, shift = {(10, 0)}] (0-0.5, 0) -- (0.5, 1.0+0.5) -- (1.0+0.5, 0) -- (0.5, -1-0.5) -- (0-0.5, 0);
        \node at (8-0.3, 2) {\LARGE$\Diamond_{3,1}$};
    \end{tikzpicture} \]
    Then we have
    \[ \partb_{1,1}(b) = 3+0-1 = 2,\ \partb_{2,1}(b) = 0+2-4 = -2,\ \partb_{3,1}(b) = 2, \]
    which implies $\varepsilon_1(b) = 2, m_1(b) = (1), M_1(b) = (3)$. 
    This coincides with Example \ref{ex:A3 bicrystal}.

    On the other hand, to find values $\sumb_\mu^\ast(b)$ for $\mu \in \Pi_3^\ast$, consider diamonds $\Diamond_{i, 4-i}^\ast$ for $1 \leq i \leq 3$ 
    with the recording of $\epsilon_{u,v}b_{u,v}$ in each diamond. 
    
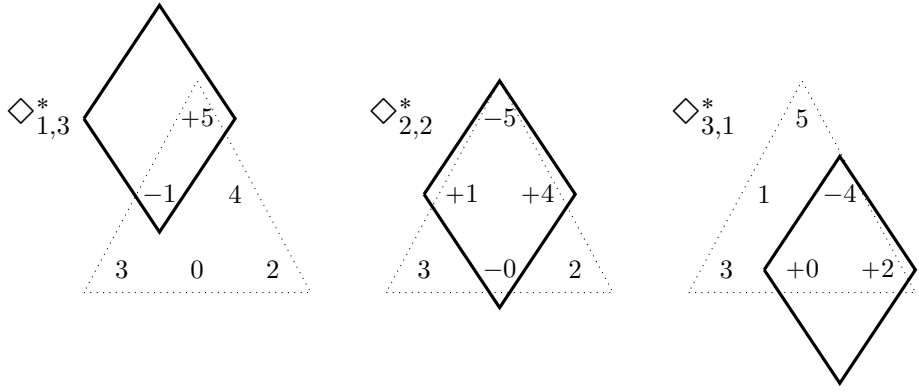
\[ \begin{tikzpicture}
        \node at (2, 0) {$2$};
        \node at (1, 0) {$0$};
        \node at (0, 0) {$3$};
        \node at (1.5, 1) {$4$};
        \node at (0.5, 1) {$-1$};
        \node at (1, 2) {$+5$};
    
        \draw[dotted] (-0.5, -0.3) -- (2.5, -0.3) -- (1, 2.5) -- (-0.5, -0.3);
        \draw[very thick, shift = {(0,2)}] (0-0.5, 0) -- (0.5, 1.0+0.5) -- (1.0+0.5, 0) -- (0.5, -1-0.5) -- (0-0.5, 0);
        \node at (-1.1, 2) {\LARGE$\Diamond_{1,3}^\ast$};

        \node[shift = {(4, 0)}] at (2, 0) {$2$};
        \node[shift = {(4, 0)}] at (1, 0) {$-0$};
        \node[shift = {(4, 0)}] at (0, 0) {$3$};
        \node[shift = {(4, 0)}] at (1.5, 1) {$+4$};
        \node[shift = {(4, 0)}] at (0.5, 1) {$+1$};
        \node[shift = {(4, 0)}] at (1, 2) {$-5$};
    
        \draw[dotted, shift = {(4,0)}] (-0.5, -0.3) -- (2.5, -0.3) -- (1, 2.5) -- (-0.5, -0.3);
        \draw[very thick, shift = {(4.5,1)}] (0-0.5, 0) -- (0.5, 1.0+0.5) -- (1.0+0.5, 0) -- (0.5, -1-0.5) -- (0-0.5, 0);
        \node at (4-0.3, 2) {\LARGE$\Diamond_{2,2}^\ast$};

        \node[shift = {(8, 0)}] at (2, 0) {$+2$};
        \node[shift = {(8, 0)}] at (1, 0) {$+0$};
        \node[shift = {(8, 0)}] at (0, 0) {$3$};
        \node[shift = {(8, 0)}] at (1.5, 1) {$-4$};
        \node[shift = {(8, 0)}] at (0.5, 1) {$1$};
        \node[shift = {(8, 0)}] at (1, 2) {$5$};
    
        \draw[dotted, shift = {(8,0)}] (-0.5, -0.3) -- (2.5, -0.3) -- (1, 2.5) -- (-0.5, -0.3);
        \draw[very thick, shift = {(9, 0)}] (0-0.5, 0) -- (0.5, 1.0+0.5) -- (1.0+0.5, 0) -- (0.5, -1-0.5) -- (0-0.5, 0);
        \node at (8-0.3, 2) {\LARGE$\Diamond_{3,1}^\ast$};
    \end{tikzpicture} \]
    Then we have
    \[ \partb_{1,3}^\ast(b) = 5-1 = 4,\ \partb_{2,2}^\ast(b) = 1+4-5-0 = 0,\ \partb_{3,1}^\ast(b) = +2+0-4 = -2, \]
    which implies $\varepsilon_3^\ast(b) = 4, m_3^\ast(b) = (1), M_3^\ast(b) = (2)$. 
    This coincides with Example \ref{ex:A3 bicrystal}.
\end{ex}

\begin{ex}
    Let $b \in \mathcal{B}(\infty)$ be the element appearing in Example \ref{ex:D4 bicrystal} with 
    the following configuration of $b_{s,t}$ (see Example \ref{ex:configuration}~(2)).
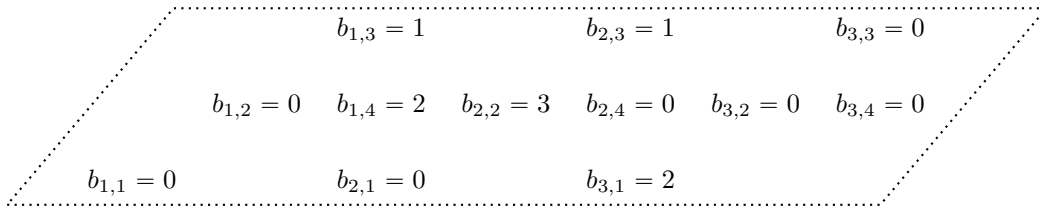
    \[ \begin{tikzpicture}[xscale=1.1]
        \node at (0,1) {$b_{1,1} = 0$};
        \node at (3,1) {$b_{2,1} = 0$};
        \node at (6,1) {$b_{3,1} = 2$};
        \node at (1.5,2) {$b_{1,2} = 0$};
        \node at (4.5,2) {$b_{2,2} = 3$};
        \node at (7.5,2) {$b_{3,2} = 0$};
        \node at (3,3) {$b_{1,3} = 1$};
        \node at (6,3) {$b_{2,3} = 1$};
        \node at (9,3) {$b_{3,3} = 0$};
        \node at (3,2) {$b_{1,4} = 2$};
        \node at (6,2) {$b_{2,4} = 0$};
        \node at (9,2) {$b_{3,4} = 0$};
        \draw[dotted, thick] (-1.5, 0.7) -- (9, 0.7) -- (11, 3.3) -- (0.5, 3.3) -- (-1.5, 0.7);
    \end{tikzpicture} \]
    
    To find values $\sumb_\lambda(b)$ for $\lambda \in \Pi_1$, consider diamonds $\Diamond_{i, 1}$ for $1 \leq i \leq 3$
    with the recording of $\epsilon_{u,v}b_{u,v}$ in each diamond.
    
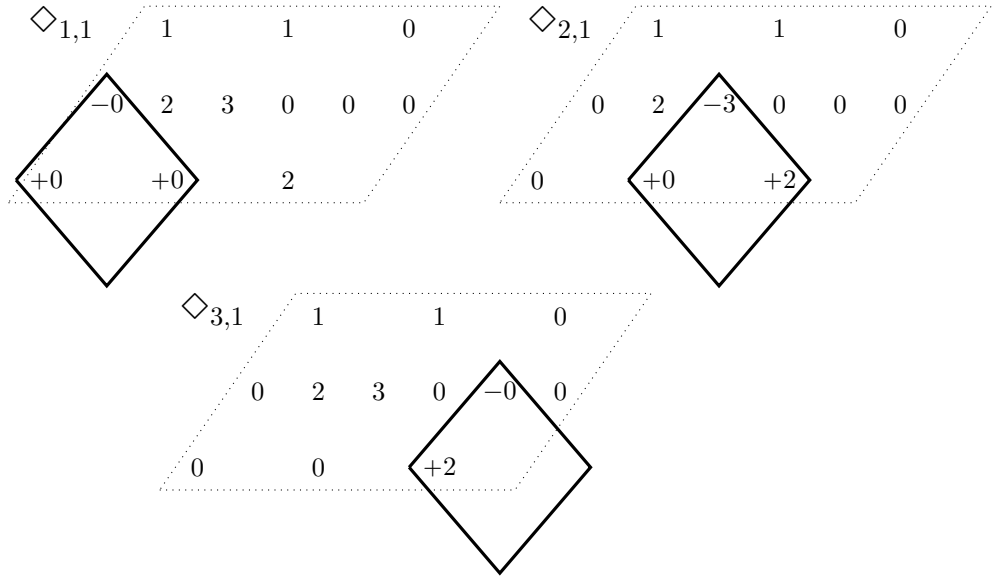
\[ \begin{tikzpicture}
        \node at (0, 0) {$+0$};
        \node at (1.6, 0) {$+0$};
        \node at (3.2, 0) {$2$};
        \node at (0.8, 1) {$-0$};
        \node at (2.4, 1) {$3$};
        \node at (4.0, 1) {$0$};
        \node at (1.6, 2) {$1$};
        \node at (3.2, 2) {$1$};
        \node at (4.8, 2) {$0$};
        \node at (1.6, 1) {$2$};
        \node at (3.2, 1) {$0$};
        \node at (4.8, 1) {$0$};
    
        \draw[dotted] (-0.5, -0.3) -- (4.2, -0.3) -- (6.0, 2.3) -- (1.3, 2.3) -- (-0.5, -0.3);
        \draw[very thick] (0-0.4, 0) -- (0.8, 1.0+0.4) -- (1.6+0.4, 0) -- (0.8, -1.0-0.4) -- (0-0.4, 0);
        \node at (0+0.2, 2) {\LARGE$\Diamond_{1,1}$};

        \node[shift = {(6.5,0)}] at (0, 0) {$0$};
        \node[shift = {(6.5,0)}] at (1.6, 0) {$+0$};
        \node[shift = {(6.5,0)}] at (3.2, 0) {$+2$};
        \node[shift = {(6.5,0)}] at (0.8, 1) {$0$};
        \node[shift = {(6.5,0)}] at (2.4, 1) {$-3$};
        \node[shift = {(6.5,0)}] at (4.0, 1) {$0$};
        \node[shift = {(6.5,0)}] at (1.6, 2) {$1$};
        \node[shift = {(6.5,0)}] at (3.2, 2) {$1$};
        \node[shift = {(6.5,0)}] at (4.8, 2) {$0$};
        \node[shift = {(6.5,0)}] at (1.6, 1) {$2$};
        \node[shift = {(6.5,0)}] at (3.2, 1) {$0$};
        \node[shift = {(6.5,0)}] at (4.8, 1) {$0$};
    
        \draw[dotted, shift = {(6.5,0)}] (-0.5, -0.3) -- (4.2, -0.3) -- (6.0, 2.3) -- (1.3, 2.3) -- (-0.5, -0.3);
        \draw[very thick, shift = {(6.5+1.6,0)}] (0-0.4, 0) -- (0.8, 1.0+0.4) -- (1.6+0.4, 0) -- (0.8, -1.0-0.4) -- (0-0.4, 0);
        \node at (6.5+0.3, 2) {\LARGE$\Diamond_{2,1}$};
    
        \node[shift = {(2, -3.8)}] at (0, 0) {$0$};
        \node[shift = {(2, -3.8)}] at (1.6, 0) {$0$};
        \node[shift = {(2, -3.8)}] at (3.2, 0) {$+2$};
        \node[shift = {(2, -3.8)}] at (0.8, 1) {$0$};
        \node[shift = {(2, -3.8)}] at (2.4, 1) {$3$};
        \node[shift = {(2, -3.8)}] at (4.0, 1) {$-0$};
        \node[shift = {(2, -3.8)}] at (1.6, 2) {$1$};
        \node[shift = {(2, -3.8)}] at (3.2, 2) {$1$};
        \node[shift = {(2, -3.8)}] at (4.8, 2) {$0$};
        \node[shift = {(2, -3.8)}] at (1.6, 1) {$2$};
        \node[shift = {(2, -3.8)}] at (3.2, 1) {$0$};
        \node[shift = {(2, -3.8)}] at (4.8, 1) {$0$};
    
        \draw[dotted, shift = {(2, -3.8)}] (-0.5, -0.3) -- (4.2, -0.3) -- (6.0, 2.3) -- (1.3, 2.3) -- (-0.5, -0.3);
        \draw[very thick, shift = {(3.2+2, 0-3.8)}] (0-0.4, 0) -- (0.8, 1.0+0.4) -- (1.6+0.4, 0) -- (0.8, -1.0-0.4) -- (0-0.4, 0);
        \node[shift = {(2, -3.8)}] at (0+0.2, 2) {\LARGE$\Diamond_{3,1}$};
    \end{tikzpicture} \]
    Then we have
    \[ \partb_{1,1}(b) = 0-0 = 0,\ \partb_{2,1}(b) = 0+2-3 = -1,\ \partb_{3,1}(b) = 2, \]
    which implies $\varepsilon_1(b) = 2, m_1(b) = M_1(b) = (3)$. 
    This coincides with Example \ref{ex:D4 bicrystal}.

    On the other hand, to find values $\sumb_\mu^\ast(b)$ for $\mu \in \Pi_3^\ast$, consider following six diamonds $\Diamond_{s, t}^\ast$ 
    with the recording of $\epsilon_{u,v}b_{u,v}$ in each diamond.
    
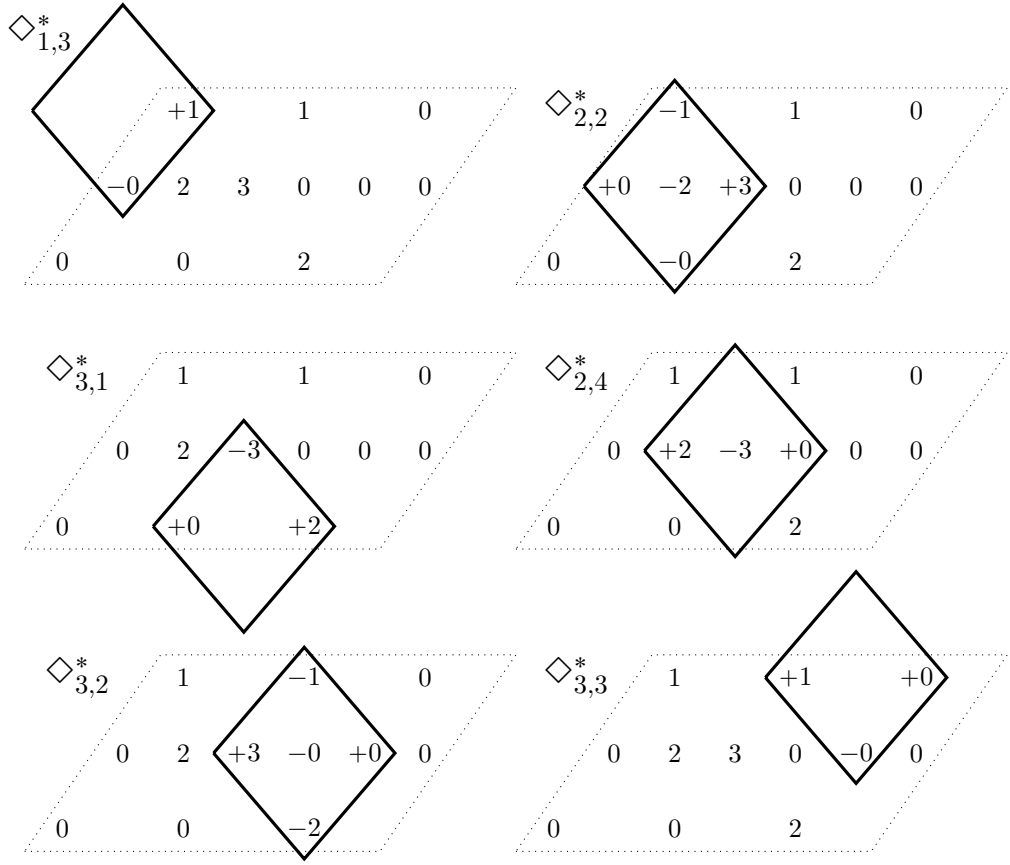
\[ \begin{tikzpicture}
        \node at (0, 0) {$0$};
        \node at (1.6, 0) {$0$};
        \node at (3.2, 0) {$2$};
        \node at (0.8, 1) {$-0$};
        \node at (2.4, 1) {$3$};
        \node at (4.0, 1) {$0$};
        \node at (1.6, 2) {$+1$};
        \node at (3.2, 2) {$1$};
        \node at (4.8, 2) {$0$};
        \node at (1.6, 1) {$2$};
        \node at (3.2, 1) {$0$};
        \node at (4.8, 1) {$0$};
    
        \draw[dotted] (-0.5, -0.3) -- (4.2, -0.3) -- (6.0, 2.3) -- (1.3, 2.3) -- (-0.5, -0.3);
        \draw[very thick, shift = {(0,2)}] (0-0.4, 0) -- (0.8, 1.0+0.4) -- (1.6+0.4, 0) -- (0.8, -1.0-0.4) -- (0-0.4, 0);
        \node at (-0.3, 3) {\LARGE$\Diamond_{1,3}^\ast$};

        \node[shift = {(6.5,0)}] at (0, 0) {$0$};
        \node[shift = {(6.5,0)}] at (1.6, 0) {$-0$};
        \node[shift = {(6.5,0)}] at (3.2, 0) {$2$};
        \node[shift = {(6.5,0)}] at (0.8, 1) {$+0$};
        \node[shift = {(6.5,0)}] at (2.4, 1) {$+3$};
        \node[shift = {(6.5,0)}] at (4.0, 1) {$0$};
        \node[shift = {(6.5,0)}] at (1.6, 2) {$-1$};
        \node[shift = {(6.5,0)}] at (3.2, 2) {$1$};
        \node[shift = {(6.5,0)}] at (4.8, 2) {$0$};
        \node[shift = {(6.5,0)}] at (1.6, 1) {$-2$};
        \node[shift = {(6.5,0)}] at (3.2, 1) {$0$};
        \node[shift = {(6.5,0)}] at (4.8, 1) {$0$};
    
        \draw[dotted, shift = {(6.5,0)}] (-0.5, -0.3) -- (4.2, -0.3) -- (6.0, 2.3) -- (1.3, 2.3) -- (-0.5, -0.3);
        \draw[very thick, shift = {(6.5+0.8,1)}] (0-0.4, 0) -- (0.8, 1.0+0.4) -- (1.6+0.4, 0) -- (0.8, -1.0-0.4) -- (0-0.4, 0);
        \node at (6.5+0.3, 2) {\LARGE$\Diamond_{2,2}^\ast$};
    
        \node[shift = {(0, -3.5)}] at (0, 0) {$0$};
        \node[shift = {(0, -3.5)}] at (1.6, 0) {$+0$};
        \node[shift = {(0, -3.5)}] at (3.2, 0) {$+2$};
        \node[shift = {(0, -3.5)}] at (0.8, 1) {$0$};
        \node[shift = {(0, -3.5)}] at (2.4, 1) {$-3$};
        \node[shift = {(0, -3.5)}] at (4.0, 1) {$0$};
        \node[shift = {(0, -3.5)}] at (1.6, 2) {$1$};
        \node[shift = {(0, -3.5)}] at (3.2, 2) {$1$};
        \node[shift = {(0, -3.5)}] at (4.8, 2) {$0$};
        \node[shift = {(0, -3.5)}] at (1.6, 1) {$2$};
        \node[shift = {(0, -3.5)}] at (3.2, 1) {$0$};
        \node[shift = {(0, -3.5)}] at (4.8, 1) {$0$};
    
        \draw[dotted, shift = {(0, -3.5)}] (-0.5, -0.3) -- (4.2, -0.3) -- (6.0, 2.3) -- (1.3, 2.3) -- (-0.5, -0.3);
        \draw[very thick, shift = {(1.6+0,0-3.5)}] (0-0.4, 0) -- (0.8, 1.0+0.4) -- (1.6+0.4, 0) -- (0.8, -1.0-0.4) -- (0-0.4, 0);
        \node[shift = {(0, -3.5)}] at (0+0.2, 2) {\LARGE$\Diamond_{3,1}^\ast$};

        \node[shift = {(6.5,-3.5)}] at (0, 0) {$0$};
        \node[shift = {(6.5,-3.5)}] at (1.6, 0) {$0$};
        \node[shift = {(6.5,-3.5)}] at (3.2, 0) {$2$};
        \node[shift = {(6.5,-3.5)}] at (0.8, 1) {$0$};
        \node[shift = {(6.5,-3.5)}] at (2.4, 1) {$-3$};
        \node[shift = {(6.5,-3.5)}] at (4.0, 1) {$0$};
        \node[shift = {(6.5,-3.5)}] at (1.6, 2) {$1$};
        \node[shift = {(6.5,-3.5)}] at (3.2, 2) {$1$};
        \node[shift = {(6.5,-3.5)}] at (4.8, 2) {$0$};
        \node[shift = {(6.5,-3.5)}] at (1.6, 1) {$+2$};
        \node[shift = {(6.5,-3.5)}] at (3.2, 1) {$+0$};
        \node[shift = {(6.5,-3.5)}] at (4.8, 1) {$0$};
    
        \draw[dotted, shift = {(6.5,-3.5)}] (-0.5, -0.3) -- (4.2, -0.3) -- (6.0, 2.3) -- (1.3, 2.3) -- (-0.5, -0.3);
        \draw[very thick, shift = {(6.5+1.6,1-3.5)}] (0-0.4, 0) -- (0.8, 1.0+0.4) -- (1.6+0.4, 0) -- (0.8, -1.0-0.4) -- (0-0.4, 0);
        \node at (6.5+0.3, 2-3.5) {\LARGE$\Diamond_{2,4}^\ast$};

        \node[shift = {(0, -7.5)}] at (0, 0) {$0$};
        \node[shift = {(0, -7.5)}] at (1.6, 0) {$0$};
        \node[shift = {(0, -7.5)}] at (3.2, 0) {$-2$};
        \node[shift = {(0, -7.5)}] at (0.8, 1) {$0$};
        \node[shift = {(0, -7.5)}] at (2.4, 1) {$+3$};
        \node[shift = {(0, -7.5)}] at (4.0, 1) {$+0$};
        \node[shift = {(0, -7.5)}] at (1.6, 2) {$1$};
        \node[shift = {(0, -7.5)}] at (3.2, 2) {$-1$};
        \node[shift = {(0, -7.5)}] at (4.8, 2) {$0$};
        \node[shift = {(0, -7.5)}] at (1.6, 1) {$2$};
        \node[shift = {(0, -7.5)}] at (3.2, 1) {$-0$};
        \node[shift = {(0, -7.5)}] at (4.8, 1) {$0$};
    
        \draw[dotted, shift = {(0, -7.5)}] (-0.5, -0.3) -- (4.2, -0.3) -- (6.0, 2.3) -- (1.3, 2.3) -- (-0.5, -0.3);
        \draw[very thick, shift={(2.4+0, 1-7.5)}] (0-0.4, 0) -- (0.8, 1.0+0.4) -- (1.6+0.4, 0) -- (0.8, -1.0-0.4) -- (0-0.4, 0);
        \node[shift = {(0, -7.5)}] at (0+0.2, 2) {\LARGE$\Diamond_{3,2}^\ast$};

        \node[shift = {(6.5,-7.5)}] at (0, 0) {$0$};
        \node[shift = {(6.5,-7.5)}] at (1.6, 0) {$0$};
        \node[shift = {(6.5,-7.5)}] at (3.2, 0) {$2$};
        \node[shift = {(6.5,-7.5)}] at (0.8, 1) {$0$};
        \node[shift = {(6.5,-7.5)}] at (2.4, 1) {$3$};
        \node[shift = {(6.5,-7.5)}] at (4.0, 1) {$-0$};
        \node[shift = {(6.5,-7.5)}] at (1.6, 2) {$1$};
        \node[shift = {(6.5,-7.5)}] at (3.2, 2) {$+1$};
        \node[shift = {(6.5,-7.5)}] at (4.8, 2) {$+0$};
        \node[shift = {(6.5,-7.5)}] at (1.6, 1) {$2$};
        \node[shift = {(6.5,-7.5)}] at (3.2, 1) {$0$};
        \node[shift = {(6.5,-7.5)}] at (4.8, 1) {$0$};
    
        \draw[dotted, shift = {(6.5,0-7.5)}] (-0.5, -0.3) -- (4.2, -0.3) -- (6.0, 2.3) -- (1.3, 2.3) -- (-0.5, -0.3);
        \draw[very thick, shift = {(6.5+3.2,2-7.5)}] (0-0.4, 0) -- (0.8, 1.0+0.4) -- (1.6+0.4, 0) -- (0.8, -1.0-0.4) -- (0-0.4, 0);
        \node at (6.5+0.3, 2-7.5) {\LARGE$\Diamond_{3,3}^\ast$};
    \end{tikzpicture} \]
    Then we have
    \begin{eqnarray*}
        && \partb_{1,3}^\ast(b) = 1-0 = 1,\ \partb_{2,2}^\ast(b) = 0+3-1-2-0 = 0,\ \partb_{3,1}^\ast(b) = +2+0-3 = -1, \\
        && \partb_{2,4}^\ast(b) = 2+0-3 = -1, \ \partb_{3,2}^\ast(b) = 3+0-1-0-2 = 0,\ \partb_{3,3}^\ast(b) = 1+0-0 = 1,
    \end{eqnarray*}
    which implies $\varepsilon_3^\ast(b) = 1, m_3^\ast(b) = (1), M_3^\ast(b) = (2)$. 
    This coincides with Example \ref{ex:D4 bicrystal}.
\end{ex}

\section{A combinatorial description of the extended crystal} \label{sec:extended}
\subsection{A combinatorial description of the extended crystal}
In this subsection, we assume $\mathfrak{g}$ is always of finite $ABD$ type. Let
\[ \widehat{\mathcal{B}}(\infty) (= \widehat{\mathcal{B}}_\mathfrak{g}(\infty)) \,=\, \left\{\, \left. {\bf b} = (b^{(k)})_{k \in \mathbb{Z}} \in \prod_{k \in \mathbb{Z}} \mathcal{B}_\mathfrak{g}(\infty)  \,\right|\, 
    b^{(k)} = {\bf 1} \mbox{ for all but finitely many } k \right\}. \]
We use the exponent $(t)$ to indicate the position; for instance, $b^{(k)} = (b_{s,t}^{(k)})$ is the $k$-th component of ${\bf b} \in \widehat{\mathcal{B}}(\infty)$. 
We can find the combinatorial description of $\widehat{\mathcal{B}}(\infty)$ by extending that given in Section \ref{subsec:bicrystal description}. 
Let $(\widehat{\Pi}_i, \leq)$ be a partially ordered set
\[ \widehat{\Pi}_i = \{ \lambda \,|\, \lambda \in \Pi_i \} \cup \{ \mu^\ast \,|\, \mu \in \Pi^\ast_i \}, \]
where $\mu^\ast$ is a formal symbol, with the following ordering.
\begin{itemize}
    \item For any $\lambda \in \Pi_i$ and $\mu \in \Pi_i^\ast$, $\lambda \leq \mu^\ast$ always holds.
    \item For any $\lambda, \mu \in \Pi_i$, $\lambda \leq \mu$ in $\widehat{\Pi}_i$ holds if and only if $\lambda \subseteq \mu$.
    \item For any $\lambda, \mu \in \Pi_i^\ast$, $\lambda^\ast \leq \mu^\ast$ in $\widehat{\Pi}_i$ holds if and only if $\lambda \supseteq \mu$.
\end{itemize}
We call $\gamma \in \widehat{\Pi}_i$ unstarrred (resp. starred) if $\gamma = \lambda$ (resp. $\gamma = \mu^\ast$) for some $\lambda \in \Pi_i$ (resp. $\mu \in \Pi_i^\ast$).

For $\gamma \in \widehat{\Pi}_i$ and ${\bf b} = (b^{(k)}) \in \widehat{\mathcal{B}}(\infty)$ and $k \in \mathbb{Z}$, set 
\[ \widehat{\sumb}_\gamma^{(k)}({\bf b}) = \begin{cases}
    \displaystyle \sumb_\lambda(b^{(k)}) & \mbox{if $\gamma = \lambda$ is unstarred}, \\
    \displaystyle \sumb_\mu^\ast(b^{(k+1)}) & \mbox{if $\gamma = \mu^\ast$ is starrred}.
\end{cases} \]
In addition, for $(i, k) \in \widehat{I}$ and ${\bf b} = (b^{(k)}) \in \widehat{\mathcal{B}}(\infty)$, set
\begin{eqnarray*}
    \mathcal{E}_{(i,k)}({\bf b}) &=& \max\{ \widehat{\sumb}_\gamma^{(k)} ({\bf b}) : \gamma \in \widehat{\Pi}_i \}, \\
    \widehat{m}_{(i,k)}({\bf b}) &=& \min\{ \gamma \in \widehat{\Pi}_i : \widehat{\sumb}_\gamma^{(k)}({\bf b}) = \mathcal{E}_{(i,k)}({\bf b}) \},\\
    \widehat{M}_{(i,k)}({\bf b}) &=& \max\{ \gamma \in \widehat{\Pi}_i : \widehat{\sumb}_\gamma^{(k)}({\bf b}) = \mathcal{E}_{(i,k)}({\bf b}) \}.
\end{eqnarray*}
Note that
\begin{equation} \label{eqn:max value}
    \max\{ \widehat{\sumb}_\lambda^{(k)}({\bf b}) : \lambda \in \Pi_i \} = \varepsilon_i(b^{(k)}) \quad{and}\quad
    \max\{ \widehat{\sumb}_{\mu^\ast}^{(k)}({\bf b}) : \mu \in \Pi_i^\ast \} = \varepsilon_i^\ast(b^{(k+1)}).
\end{equation}

\begin{defn} \label{def:combi extended crystal}
Define maps on $\widehat{\mathcal{B}}(\infty)$. 
For ${\bf b} = (b^{(k)}) \in \widehat{\mathcal{B}}(\infty)$ and $(i, k) \in \widehat{I}$, 
\begin{eqnarray*}
    \widehat{\rm wt}({\bf b}) &=& \sum_{t \in \mathbb{Z}} (-1)^t \,{\rm wt}(b^{(t)}), \\
    \widehat{\varepsilon}_{(i,k)}({\bf b}) &=& \max\{ \widehat{\sumb}_\lambda^{(k)}({\bf b}) : \lambda \in \Pi_i \} 
    - \max\{ \widehat{\sumb}_{\mu^\ast}^{(k)}({\bf b}) : \mu \in \Pi_i^\ast \},\\
    \widetilde{E}_{(i,k)}({\bf b}) &=& \begin{cases}
        \displaystyle {\bf b} - \sum_{(s,t) \in \lambda} {\bf v}(\mathcal{I}_{T_i}(s,t))^{(k)} & \mbox{if } \widehat{M}_{(i,k)}({\bf b}) = \lambda \mbox{ is unstarred}, \\
        \displaystyle {\bf b} - \sum_{(s,t) \in \mu} {\bf v}(\mathcal{I}_{T_i^\ast}(s,t))^{(k+1)} & \mbox{if } \widehat{M}_{(i,k)}({\bf b}) = \mu^\ast \mbox{ is starred},
    \end{cases} \\
    \widetilde{F}_{(i,k)}({\bf b}) &=& \begin{cases}
        \displaystyle {\bf b} + \sum_{(s,t) \in \lambda} {\bf v}(\mathcal{I}_{T_i}(s,t))^{(k)} & \mbox{if } \widehat{m}_{(i,k)}({\bf b}) = \lambda \mbox{ is unstarred}, \\
        \displaystyle {\bf b} + \sum_{(s,t) \in \mu} {\bf v}(\mathcal{I}_{T_i^\ast}(s,t))^{(k+1)} & \mbox{if } \widehat{m}_{(i,k)}({\bf b}) = \mu^\ast \mbox{ is starred}.
    \end{cases}
\end{eqnarray*}
\end{defn}
\begin{ex}\label{ex:D4 extended crystal}
    We assume $\mathfrak{g}$ is of type $D_4$, and
    let ${\bf b} = (b^{(k)}) \in \widehat{\mathcal{B}}(\infty)$, where
    \[ b^{(0)} = (0,0,0,2,0,1,3,0,2,1,0,0), \qquad b^{(1)} = (0,0,0,2,0,0,2,1,2,1,1,0), \]
    and $b^{(k)} = {\bf 1}$ for all $k \in \mathbb{Z}$ except $0$ and $1$. 
    Then we have
    \[ \widehat{\sumb}_{(1)}^{(0)}({\bf b}) = 1, \quad\widehat{\sumb}_{(2)}^{(0)}({\bf b}) = 1, \quad\widehat{\sumb}_{(3)}^{(0)}({\bf b}) = 2, \quad\widehat{\sumb}_{(1)^\ast}^{(0)}({\bf b}) = 0 \]
    (cf. Example \ref{ex:D4 bicrystal}) and so
    \[ \widehat{M}_{(1,0)}({\bf b}) = (3) \in \Pi_1 \quad \mbox{and} \quad \widehat{m}_{(1,0)}({\bf b}) = (3) \in \Pi_1. \]
    Thus, $\widetilde{E}_{(1,0)}({\bf b})$ is obtained by applying $\widetilde{e}_1$ to $b_0$, 
    and $\widetilde{F}_{(1,0)}({\bf b})$ is obtained by applying $\widetilde{f}_1$ to $b_0$, that is,
    \begin{eqnarray*}
        {\bf b} &=& (\dots, {\bf 1}, b_1, b_0, {\bf 1}, \dots) \\
            &=& (\dots, {\bf 1}, (0,0,0,2,0,0,2,1,2,1,1,0),\, (0,0,0,2,0,1,3,0,2,1,0,0),\, {\bf 1}, \dots), \\
        \widetilde{E}_{(1,0)}({\bf b}) &=& (\dots, {\bf 1}, b_1, \widetilde{e}_1(b_0), {\bf 1}, \dots) \\
            &=& (\dots, {\bf 1}, b_1,\, (0,0,0,1,0,1,3,0,2,1,0,0), \,{\bf 1}, \dots), \\
        \widetilde{F}_{(1,0)}({\bf b}) &=& (\dots, {\bf 1}, b_1, \widetilde{f}_1(b_0), {\bf 1}, \dots) \\
            &=& (\dots, {\bf 1}, b_1,\, (0,0,0,3,0,1,3,0,2,1,0,0),\, {\bf 1}, \dots).
    \end{eqnarray*}
\end{ex}

If we identify $B(\infty)$ with $\mathcal{B}(\infty)$, we consider $\widehat{B}(\infty)$ and $\widehat{\mathcal{B}}(\infty)$ as the same sets.
\begin{thm} \label{thm:extended crystal iso}
    The maps $\widehat{\rm wt}, \widehat{\varepsilon}_{(i,k)}, \widetilde{E}_{(i,k)}, \widetilde{F}_{(i,k)}$ given in Definition \ref{def:extended crystal} and Definition \ref{def:combi extended crystal} coincide respectively.
\end{thm}
\begin{proof}
    When we consider $\widehat{B}(\infty)$ and $\widehat{\mathcal{B}}(\infty)$ as the same sets, 
    it is immediate that the maps $\widehat{\rm wt}$ coincide. 
    In addition, the maps $\widehat{\varepsilon}_{(i,k)}$ coincide by \eqref{eqn:max value}. 
    It remains to show that the extended crystal operators respectively coincide.

    For ${\bf b} \in \widehat{B}(\infty)$, suppose $\widehat{\varepsilon}_{(i,k)}({\bf b}) > 0$, or equivalently $\varepsilon_i(b^{(k)}) > \varepsilon_i^\ast(b^{(k+1)})$. 
    By \eqref{eqn:max value}, $\max\{ \widehat{\sumb}_\lambda^{(k)}({\bf b}) : \lambda \in \Pi_i \} > \max\{ \widehat{\sumb}_{\mu^\ast}^{(k)}({\bf b}) : \mu \in \Pi_i^\ast \}$ and 
    it is equivalent to that $\widehat{M}_{(i,k)} ({\bf b})$ is unstarred for ${\bf b} \in \widehat{\mathcal{B}}(\infty)$. 
    In this case, the operator $\widetilde{E}_{(i,k)}({\bf b})$ is obtained by applying $\widetilde{e}_i$ to $b^{(k)}$. 
    Similarly, we can show that $\varepsilon_{(i,k)}({\bf b}) \leq 0$ if and only if $\widehat{M}_{(i,k)}({\bf b})$ is starred, 
    which implies extended crystal operators $\widetilde{E}_{(i,k)}$ on $\widehat{B}(\infty)$ and $\widehat{\mathcal{B}}(\infty)$ coincide. 
    We can apply the similar argument to $\widetilde{F}_{(i,k)}$, which completes the proof.
\end{proof}

\subsection{An extended sliding diamond rule for type $A_n$} \label{subsec:extended diamond}
Since $\widehat{\mathcal{B}}(\infty)$ can be considered as an infinite copy of $\mathcal{B}(\infty)$, 
we construct a combinatorial model for $\widehat{\mathcal{B}}(\infty)$ by infinitely combining that for $\mathcal{B}(\infty)$ (cf. Example \ref{ex:configuration}~(1)).
In particular, we explicitly give a combinatorial rule for the extended crystal operators on $\widehat{\mathcal{B}}(\infty)$ for type $A_n$, 
which is called an extended sliding diamond rule. It is a natural extension of the sliding diamond rule (Definition \ref{def:SD}).
\begin{defn}[The extended sliding diamond rule] \label{def:extended SD}
    Take ${\bf b} = (b^{(k)}) \in \widehat{\mathcal{B}}(\infty)$ with $b^{(k)} = (b^{(k)}_{s,t}) \in \mathcal{B}(\infty)$ for $k \in \mathbb{Z}$. 
    \begin{enumerate}
        \item Put $b_{s,t}^{(k)}$ in a plane lying at the point
        \[ \begin{cases}
            (-s-2t+3-kn, s) & \mbox{when $k$ is even,} \\
            (-s-2t+3-kn, n-s+2) & \mbox{when $k$ is odd}
        \end{cases} \]
        for $(s, t) \in \mathcal{I}_n^A$.
        In this case, we denote by $\Diamond^{(k)}_{s,t}$ (resp. $\Diamond^{\ast(k)}_{s,t}$) the diamond $\Diamond_{s,t}$ (resp. $\Diamond_{s,t}^\ast$) in the $k$-th component (cf. Definition \ref{def:SD}). 

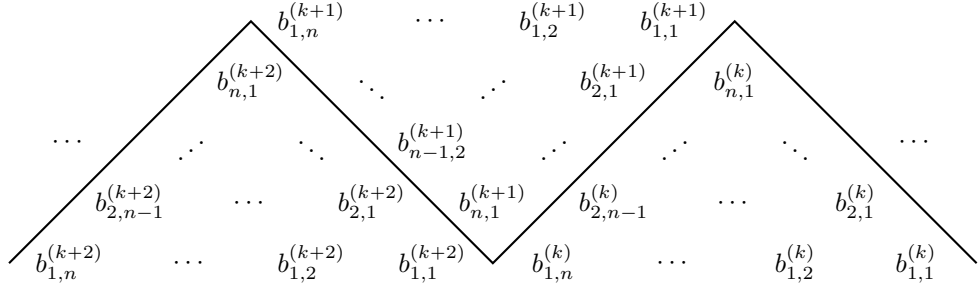
        \[ \begin{tikzpicture}[xscale=0.8, yscale=0.8]
        \node at (14, 0) {$b_{1,1}^{(k)}$};
        \node at (12, 0) {$b_{1,2}^{(k)}$};
        \node at (10, 0) {$\cdots$};
        \node at (8, 0) {$b_{1,n}^{(k)}$};
        \node at (13, 1) {$b_{2,1}^{(k)}$};
        \node at (11, 1) {$\cdots$};
        \node at (9, 1) {$b_{2,n-1}^{(k)}$};
        \node at (12, 2) {$\ddots$};
        \node at (10, 2) {$\iddots$};
        \node at (11, 3) {$b_{n,1}^{(k)}$};
    
        \node at (10, 4) {$b_{1,1}^{(k+1)}$};
        \node at (9, 3) {$b_{2,1}^{(k+1)}$};
        \node at (8, 2) {$\iddots$};
        \node at (7, 1) {$b_{n,1}^{(k+1)}$};
        \node at (8, 4) {$b_{1,2}^{(k+1)}$};
        \node at (7, 3) {$\iddots$};
        \node at (6, 2) {$b_{n-1, 2}^{(k+1)}$};
        \node at (6, 4) {$\cdots$};
        \node at (5, 3) {$\ddots$};
        \node at (4, 4) {$b_{1,n}^{(k+1)}$};
    
        \node at (6, 0) {$b_{1,1}^{(k+2)}$};
        \node at (4, 0) {$b_{1,2}^{(k+2)}$};
        \node at (2, 0) {$\cdots$};
        \node at (0, 0) {$b_{1,n}^{(k+2)}$};
        \node at (5, 1) {$b_{2,1}^{(k+2)}$};
        \node at (3, 1) {$\cdots$};
        \node at (1, 1) {$b_{2,n-1}^{(k+2)}$};
        \node at (4, 2) {$\ddots$};
        \node at (2, 2) {$\iddots$};
        \node at (3, 3) {$b_{n,1}^{(k+2)}$};
    
        \draw[thick] (-1, 0) -- (3, 4) -- (7, 0) -- (11, 4) -- (15, 0);
        \node at (0, 2) {$\cdots$};
        \node at (14, 2) {$\cdots$};
    \end{tikzpicture} \]

        \item For $b_{u, v}^{(k)}$ contained in $\Diamond_{s,t}^{(k)}$ and $\Diamond_{s,t}^{\ast(k)}$, assign the coefficient $\epsilon_{u,v}^{(k)}$ to each $b_{u,v}^{(k)}$ as follows;
        for any $u,v$ and $k$, $\epsilon_{u,v}^{(k)} = \epsilon_{u,v}$ (cf. Definition \ref{def:SD}), that is,
        \[ \epsilon_{u,v}^{(k)} = \begin{cases}
            1 & \mbox{if }v = t, \\
            \langle h_t, \alpha_v \rangle & \mbox{otherwise.}
        \end{cases} \]
        \item Obtain the $\epsilon$-weighted sum of $b_{u,v}^{(k)}$ contained in $\Diamond_{s,t}^{(k)}$ or $\Diamond_{s,t}^{\ast(k)}$, 
        which results in
        \[ \sum_{(u, v) \in \Diamond_{s,t}^{(k)}} \epsilon_{u,v}^{(k)}b_{u,v}^{(k)} = \partb_{s,t}^{(k)}(b), \quad 
        \sum_{(u, v) \in \Diamond_{s,t}^{\ast(k)}} \epsilon_{u,v}^{(k)}b_{u,v}^{(k)} = \partb_{s,t}^{\ast(k)}(b). \]
    \end{enumerate}
\end{defn}

We can obtain $\widehat{\sumb}_\gamma^{(k)}({\bf b})$ for $\gamma \in \widehat{\Pi}_i$ from the extended sliding diamond rule, 
and then we can calculate the maps as given in Definition \ref{def:combi extended crystal}.
In particular, when $\mathfrak{g}$ is of type $A$, $\widehat{\Pi}_i$ is totally ordered 
and so we can simply relabel $\widehat{\Pi}_i$ as integers; identify the unstarred $\lambda$ with $|\lambda|$ 
and the starred $\mu^\ast$ with $n+2-|\mu|$. Then the extended crystal operators can be rewritten as follows. 
\begin{eqnarray*}
    \widetilde{E}_{(i,k)}({\bf b}) &=& \begin{cases}
        (\dots, b^{(k+1)}, \widetilde{e}_i(b^{(k)}), \dots) & \mbox{if } 1 \leq \widehat{M}_{(i,k)}({\bf b}) \leq n+1-i, \\
        (\dots, \widetilde{f}_i^\ast(b^{(k+1)}), b^{(k)}, \dots) & \mbox{if } n+1-i < \widehat{M}_{(i,k)}({\bf b}) \leq n+1
    \end{cases} \\
    \widetilde{F}_{(i,k)}({\bf b}) &=& \begin{cases}
        (\dots, b^{(k+1)}, \widetilde{f}_i(b^{(k)}), \dots) & \mbox{if } 1 \leq \widehat{m}_{(i,k)}({\bf b}) \leq n+1-i, \\
        (\dots, \widetilde{e}_i^\ast(b^{(k+1)}), b^{(k)}, \dots) & \mbox{if } n+1-i < \widehat{m}_{(i,k)}({\bf b}) \leq n+1
    \end{cases}
\end{eqnarray*}

\begin{ex}\label{ex:A3 extended crystal}
    We assume $\mathfrak{g}$ is of type $A_3$, and 
    let ${\bf b} = (b^{(k)}) \in \widehat{\mathcal{B}}(\infty)$, where $b^{(0)} = (2,4,0,5,1,3)$ and $b^{(1)} = (0,2,1,3,1,2)$ and $b^{(k)} = {\bf 1}$ for all $k \in \mathbb{Z}$ except $0$ and $1$. 
    
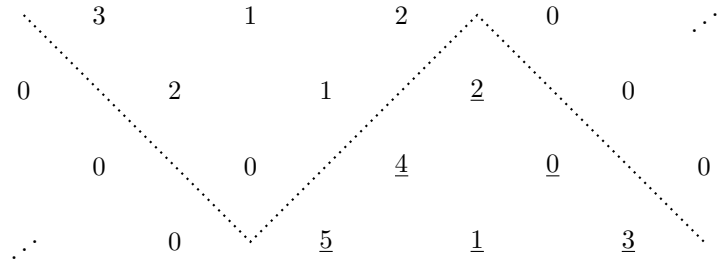
\[ \begin{tikzpicture}
        \node at (7, 3) {$0$};
        \node at (8, 2) {$0$};
        \node at (9, 1) {$0$};
        \node at (8, 0) {$\underline{3}$};
        \node at (6, 0) {$\underline{1}$};
        \node at (4, 0) {$\underline{5}$};
        \node at (7, 1) {$\underline{0}$};
        \node at (5, 1) {$\underline{4}$};
        \node at (6, 2) {$\underline{2}$};
        \node at (5, 3) {$2$};
        \node at (4, 2) {$1$};
        \node at (3, 1) {$0$};
        \node at (3, 3) {$1$};
        \node at (2, 2) {$2$};
        \node at (1, 3) {$3$};
        \node at (2, 0) {$0$};
        \node at (1, 1) {$0$};
        \node at (0, 2) {$0$};

        \draw[thick, dotted] (0, 3) -- (3, 0) -- (6, 3) -- (9, 0);
        \node at (0, 0) {$\iddots$};
        \node at (9, 3) {$\iddots$};
    \end{tikzpicture} \]
    Note that unwritten integers in the diagram are all zero.
    By calculating $\epsilon$-weighted sums for diamonds $\Diamond_{1,1}^{(0)}, \Diamond_{2,1}^{(0)}, \Diamond_{3,1}^{(0)}$, and $\Diamond_{1,1}^{\ast(1)}$, 
    we directly obtain $\partb_{1,1}(b^{(0)}) = 2, \quad\partb_{2,1}(b^{(0)}) = -2, \quad\partb_{3,1}(b^{(0)}) = 2, \quad\partb_{1,1}^\ast(b^{(1)}) = 2$ (cf. Example \ref{ex:A3 diamond}).
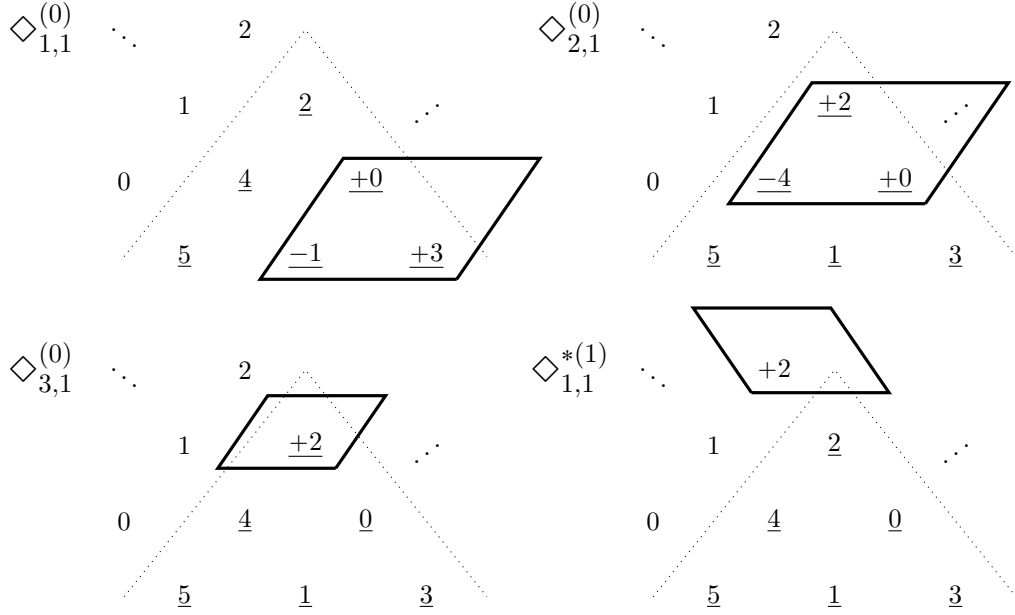
    \[ \begin{tikzpicture}
        \node at (5*0.8, 0) {$\underline{+3}$};
        \node at (3*0.8, 0) {$\underline{-1}$};
        \node at (1*0.8, 0) {$\underline{5}$};
        \node at (4*0.8, 1) {$\underline{+0}$};
        \node at (2*0.8, 1) {$\underline{4}$};
        \node at (3*0.8, 2) {$\underline{2}$};
        \node at (2*0.8, 3) {$2$};
        \node at (1*0.8, 2) {$1$};
        \node at (0, 1) {$0$};
        \node at (0, 3) {$\ddots$};
        \node at (5*0.8, 2) {$\iddots$};

        \draw[dotted] (0, 0) -- (3*0.8, 3) -- (6*0.8, 0);
        \draw[very thick] (4.4, -0.3) --(5.5, 1.3) -- (2.9, 1.3) -- (1.8, -0.3) -- (4.4, -0.3);
        \node at (-1.1, 3) {\LARGE$\Diamond_{1,1}^{(0)}$};

        \node[shift={(7, 0)}] at (5*0.8, 0) {$\underline{3}$};
        \node[shift={(7, 0)}] at (3*0.8, 0) {$\underline{1}$};
        \node[shift={(7, 0)}] at (1*0.8, 0) {$\underline{5}$};
        \node[shift={(7, 0)}] at (4*0.8, 1) {$\underline{+0}$};
        \node[shift={(7, 0)}] at (2*0.8, 1) {$\underline{-4}$};
        \node[shift={(7, 0)}] at (3*0.8, 2) {$\underline{+2}$};
        \node[shift={(7, 0)}] at (2*0.8, 3) {$2$};
        \node[shift={(7, 0)}] at (1*0.8, 2) {$1$};
        \node[shift={(7, 0)}] at (0, 1) {$0$};
        \node[shift={(7, 0)}] at (0, 3) {$\ddots$};
        \node[shift={(7, 0)}] at (5*0.8, 2) {$\iddots$};

        \draw[dotted, shift={(7, 0)}] (0, 0) -- (3*0.8, 3) -- (6*0.8, 0);
        \draw[very thick, shift={(7-0.8, 0+1)}] (4.4, -0.3) --(5.5, 1.3) -- (2.9, 1.3) -- (1.8, -0.3) -- (4.4, -0.3);
        \node[shift={(7, 0)}] at (-1.1, 3) {\LARGE$\Diamond_{2,1}^{(0)}$};
    % \end{tikzpicture} \]
    % \[ \begin{tikzpicture}
        \node[shift={(0, -4.5)}] at (5*0.8, 0) {$\underline{3}$};
        \node[shift={(0, -4.5)}] at (3*0.8, 0) {$\underline{1}$};
        \node[shift={(0, -4.5)}] at (1*0.8, 0) {$\underline{5}$};
        \node[shift={(0, -4.5)}] at (4*0.8, 1) {$\underline{0}$};
        \node[shift={(0, -4.5)}] at (2*0.8, 1) {$\underline{4}$};
        \node[shift={(0, -4.5)}] at (3*0.8, 2) {$\underline{+2}$};
        \node[shift={(0, -4.5)}] at (2*0.8, 3) {$2$};
        \node[shift={(0, -4.5)}] at (1*0.8, 2) {$1$};
        \node[shift={(0, -4.5)}] at (0, 1) {$0$};
        \node[shift={(0, -4.5)}] at (0, 3) {$\ddots$};
        \node[shift={(0, -4.5)}] at (5*0.8, 2) {$\iddots$};

        \draw[dotted, shift={(0, -4.5)}] (0, 0) -- (3*0.8, 3) -- (6*0.8, 0);
        \draw[very thick, shift={(0-1.6, -4.5+2)}] (4.4, -0.3) --(4.4+0.66, -0.3+0.96) -- (4.4-1.56+0.66, -0.3+0.96) -- (4.4-1.56, -0.3) -- (4.4, -0.3);
        \node[shift={(0, -4.5)}] at (-1.1, 3) {\LARGE$\Diamond_{3,1}^{(0)}$};

        \node[shift={(7, -4.5)}] at (5*0.8, 0) {$\underline{3}$};
        \node[shift={(7, -4.5)}] at (3*0.8, 0) {$\underline{1}$};
        \node[shift={(7, -4.5)}] at (1*0.8, 0) {$\underline{5}$};
        \node[shift={(7, -4.5)}] at (4*0.8, 1) {$\underline{0}$};
        \node[shift={(7, -4.5)}] at (2*0.8, 1) {$\underline{4}$};
        \node[shift={(7, -4.5)}] at (3*0.8, 2) {$\underline{2}$};
        \node[shift={(7, -4.5)}] at (2*0.8, 3) {$+2$};
        \node[shift={(7, -4.5)}] at (1*0.8, 2) {$1$};
        \node[shift={(7, -4.5)}] at (0, 1) {$0$};
        \node[shift={(7, -4.5)}] at (0, 3) {$\ddots$};
        \node[shift={(7, -4.5)}] at (5*0.8, 2) {$\iddots$};

        \draw[dotted, shift={(7, -4.5)}] (0, 0) -- (3*0.8, 3) -- (6*0.8, 0);
        \draw[very thick, shift={(7, -4.5)}] (1.3, 2.7) -- (1.3+1.82, 2.7) -- (1.3-0.77+1.82, 2.7+1.12) -- (1.3-0.77, 2.7+1.12) -- (1.3, 2.7);
        \node[shift={(7, -4.5)}] at (-1.1, 3) {\LARGE$\Diamond_{1,1}^{\ast(1)}$};
    \end{tikzpicture} \]
    Then we have
    \[ \widehat{\sumb}_{(1)}^{(0)}({\bf b}) = 2, \quad\widehat{\sumb}_{(2)}^{(0)}({\bf b}) = 0, \quad\widehat{\sumb}_{(3)}^{(0)}({\bf b}) = 2, \quad\widehat{\sumb}_{(1)^\ast}^{(0)}({\bf b}) = 2 \]
    and so
    \[ \widehat{M}_{(1,0)}({\bf b}) = (1)^\ast\,\in \Pi_1^\ast \quad \mbox{and} \quad \widehat{m}_{(1,0)}({\bf b}) = (1) \in \Pi_1. \]
    Thus, $\widetilde{E}_{(1,0)}({\bf b})$ is obtained by applying $\widetilde{f}_1^\ast$ to $b_1$, 
    and $\widetilde{F}_{(1,0)}({\bf b})$ is obtained by applying $\widetilde{f}_1$ to $b_0$, that is,
    \begin{eqnarray*}
        {\bf b} &=& (\dots, {\bf 1}, b_1, b_0, {\bf 1}, \dots) \\
            &=& (\dots, {\bf 1}, \,(0,2,1,3,1,2),\, (2,4,0,5,1,3),\, {\bf 1}, \dots), \\
        \widetilde{E}_{(1,0)}({\bf b}) &=& (\dots, {\bf 1}, \widetilde{f}_1^\ast(b_1), b_0, {\bf 1}, \dots) \\
            &=& (\dots, {\bf 1}, \,(0,2,1,3,1,{\bf 3}),\, (2,4,0,5,1,3), \,{\bf 1}, \dots), \\
        \widetilde{F}_{(1,0)}({\bf b}) &=& (\dots, {\bf 1}, b_1, \widetilde{f}_1(b_0), {\bf 1}, \dots) \\
            &=& (\dots, {\bf 1}, \,(0,2,1,3,1,2),\, (2,4,0,5,1,{\bf 4}),\, {\bf 1}, \dots).
    \end{eqnarray*}
\end{ex}

\section{An application to representation theory: categorical crystal} \label{sec:categorical crystal}
\subsection{The categorical crystal} \label{subsec:categorical crystal review}
We briefly review the categorical crystal introduced in \cite{KP22} and we follow notations in \cite{KP22}. 
Let $\mathcal{C}_\mathfrak{g}$ be the category of finite-dimensional integrable modules over a quantum affine algebra $U_q'(\mathfrak{g})$ associated with a symmetrizable Kac-Moody algebra $\mathfrak{g}$, and 
$\mathcal{C}_\mathfrak{g}^0$ be a Hernandez-Leclerc category (cf. \cite{HL10}), that is, the smallest full subcategory of $\mathcal{C}_\mathfrak{g}$ such that 
\begin{itemize}
    \item $\mathcal{C}_\mathfrak{g}^0$ contains $V(\varpi_i)_x$ for all $(i, x) \in \sigma_0(\mathfrak{g})$ (see \cite[Section 2.2]{KP22});
    \item $\mathcal{C}_\mathfrak{g}^0$ is stable by taking subquotients, extensions, and tensor products.
\end{itemize}
The canonical choice of $\sigma_0(\mathfrak{g})$ is given in \cite[Section 2.3]{KKOP22}, and we will follow it. 
For a module $M$ in $\mathcal{C}_\mathfrak{g}^0$, we denote by $\mathscr{D}M$ and $\mathscr{D}^{-1}M$ the right and left dual of $M$, respectively (see \cite[Section 2.3]{KKOP24} also). 
By considering $\mathscr{D}$ as a functor, we use notations like $\mathscr{D}^k M$ for $k \in \mathbb{Z}$. 

When (a Lie algebra of affine type) $\mathfrak{g}$ is given, the type of the simple Lie algebra $\mathfrak{g}_{\rm fin}$ is given in \cite[Table 1]{KP22}. 
In particular, if $\mathfrak{g}$ is of type $A_n^{(1)}$ and $D_n^{(1)}$, the type of $\mathfrak{g}_{\rm fin}$ is $A_n$ and $D_n$, respectively. 
We denote by $I_{\rm fin}$ the index set corresponding to $\mathfrak{g}_{\rm fin}$, 
and by $R$ the symmetric quiver Hecke algebra associated with the Cartan datum of $\mathfrak{g}_{\rm fin}$ (see \cite{KKK18}). 

For a strong duality datum $\mathcal{D} = \{ {\sf L}_i \}_{i \in I_{\rm fin}}$ associated with $\mathfrak{g}_{\rm fin}$, 
authors of \cite{KKOP24} construct an exact functor
\[ \mathcal{F}_\mathcal{D} : \mbox{$R$-gmod} \longrightarrow \mathcal{C}_\mathfrak{g}^0 \]
which sends simple $R$-modules to simple modules in $\mathcal{C}_\mathfrak{g}^0$ (cf. \cite{KKK18}). 
In particular, we have $\mathcal{F}_\mathcal{D}(L(i)) = {\sf L}_i$, where $L(i)$ is the one-dimensional simple $R(\alpha_i)$-modules (see \cite{KKK18}). 
We call the functor $\mathcal{F}_\mathcal{D}$ the quantum affine Schur-Weyl duality functor, or shortly duality functor. 
In addition, we define the category $\mathcal{C}_\mathcal{D}$ to be the smallest full subcategory of $\mathcal{C}_\mathfrak{g}^0$ such that
\begin{itemize}
    \item $\mathcal{C}_\mathcal{D}$ contains $\mathcal{F}_\mathcal{D}(L)$ for any simple $R$-modules $L$;
    \item $\mathcal{C}_\mathcal{D}$ is stable by taking subquotients, extensions, and tensor products.
\end{itemize}
It is proved in \cite{LV11} that the set of isomorphism classes of simple modules in $R$-gmod has a crystal structure over $\mathfrak{g}_{\rm fin}$ 
and it is isomorphic to $B_{\mathfrak{g}_{\rm fin}}(\infty)$. 
We denote by $B_\mathcal{D}$ the set of isomorphism classes of simple modules in $\mathcal{C}_\mathcal{D}$ 
and then we have a bijection
\begin{equation} \label{eqn:category crystal iso}
    \mathcal{L}_\mathcal{D} : B_{\mathfrak{g}_{\rm fin}}(\infty) \ \longrightarrow \ B_\mathcal{D},
\end{equation}
which is obtained from $\mathcal{F}_\mathcal{D}$ with considering isomorphism classes of $R$-gmod as elements of $B_{\mathfrak{g}_{\rm fin}}(\infty)$. 
This holds since $\mathcal{F}_\mathcal{D}$ preserves simple modules. 
We note that if $\mathcal{D} = \{ {\sf L}_i \}$ is a strong duality datum, then $\mathscr{D}^k\mathcal{D} = \{ \mathscr{D}^k {\sf L}_i \}$ is also a strong duality datum, 
and that $\mathcal{L}_{\mathscr{D}^k \mathcal{D}}(b) = \mathscr{D}^k \mathcal{L}_\mathcal{D}(b)$ 
for any $k \in \mathbb{Z}$ and $b \in B_{\mathfrak{g}_{\rm fin}}(\infty)$ (cf. \cite[Lemma 3.3]{KP22}). 

A strong duality datum $\mathcal{D} = \{ {\sf L}_i \}_{i \in I_{\rm fin}}$ is said to be complete if for any simple module $M$ in $\mathcal{C}_\mathfrak{g}^0$, 
there are simple modules $M_k \in \mathcal{C}_\mathcal{D}$ such that
\begin{itemize}
    \item $M_k$ is isomorphic to the trivial module except finitely many $k$;
    \item $M$ is isomorphic to ${\rm hd}(\dots \otimes \mathscr{D}^2 M_2 \otimes \mathscr{D} M_1 \otimes M_0 \otimes \mathscr{D}^{-1} M_{-1} \otimes \cdots)$,
\end{itemize}
where ${\rm hd}(M)$ is the head of the given module $M$.

For ${\bf b} = (b_k) \in \widehat{B}_{\mathfrak{g}_{\rm fin}}(\infty)$, let
\[ \mathscr{L}_\mathcal{D}({\bf b}) = {\rm hd}(\dots \otimes \mathscr{D}^2 L_2 \otimes \mathscr{D} L_1 \otimes L_0 \otimes \mathscr{D}^{-1} L_{-1} \otimes \cdots) \]
where $L_k = \mathcal{L}_\mathcal{D}(b_k)$ for all $k \in \mathbb{Z}$ and $\mathcal{L}_\mathcal{D}$ is the bijection \eqref{eqn:category crystal iso}. 
Note that $\mathscr{D}^k L_k = \mathcal{L}_{\mathscr{D}^k \mathcal{D}}(b_k)$ for $k \in \mathbb{Z}$. 
Then it is shown in \cite[Proposition 5.5]{KP22} that $\mathscr{L}_\mathcal{D}({\bf b})$ is simple. 
In addition, we define a map $\Phi_\mathcal{D} : \widehat{B}_{\mathfrak{g}_{\rm fin}}(\infty) \ \longrightarrow \ \mathscr{B}(\mathfrak{g})$ by
\begin{equation} \label{eqn:extended category crystal iso}
    \Phi_\mathcal{D}({\bf b}) = \mathscr{L}_\mathcal{D}({\bf b}),
\end{equation}
where $\mathscr{B}(\mathfrak{g})$ is the set of the isomorphism classes of simple modules in $\mathcal{C}_\mathfrak{g}^0$, 
and then the map $\Phi$ is a well-defined bijection. 

For $M \in \mathscr{B}(\mathfrak{g})$ and $(i, k) \in \widehat{I}_{\rm fin}$, we define operators
\[ \widetilde{E}_{(i,k)}(M) := {\rm hd}(M \otimes \mathscr{D}^{k+1}{\sf L}_i), \qquad \widetilde{F}_{(i,k)}(M) := {\rm hd}(\mathscr{D}^{k}{\sf L}_i \otimes M). \]
When we consider $\mathscr{B}(\mathfrak{g})$ with $\widetilde{E}_{(i,k)}$ and $\widetilde{F}_{(i,k)}$, we write $\mathscr{B}_\mathcal{D}(\mathfrak{g})$ instead of $\mathscr{B}(\mathfrak{g})$ 
since the action of $\widetilde{E}_{(i,k)}$ and $\widetilde{F}_{(i,k)}$ depends on the choice of $\mathcal{D}$. 
\begin{thm}[{\cite[Theorem 5.10]{KP22}}]
    The map $\Phi_\mathcal{D} : \widehat{B}_{\mathfrak{g}_{\rm fin}}(\infty) \to \mathscr{B}_\mathcal{D}(\mathfrak{g})$ \eqref{eqn:extended category crystal iso} 
    induces an $\widehat{I}_{\rm fin}$-colored graph isomorphism, in other words, we have
    \[ \Phi_\mathcal{D}(\widetilde{E}_{(i,k)}({\bf b})) = \widetilde{E}_{(i,k)}\Phi_\mathcal{D}({\bf b}) \quad \mbox{and} \quad \Phi_\mathcal{D}(\widetilde{E}_{(i,k)}({\bf b})) = \widetilde{E}_{(i,k)}\Phi_\mathcal{D}({\bf b})\] 
    for $(i, k) \in \widehat{I}_{\rm fin}$ and ${\bf b} \in \widehat{B}_{\mathfrak{g}_{\rm fin}}(\infty)$.
\end{thm}

On the other hand, we revisit the category $\mathcal{C}_\mathfrak{g}^0$. 
Recall that we associate a convex order on the set of positive roots with a reduced expression $\underline{w_0} = s_{i_1} \dots s_{i_\ell}$ by
\[ \beta_k = s_{i_1} \dots s_{i_{k-1}} (\alpha_{i_k}) \]
for $k \geq 1$ (cf. Section \ref{subsec:PBW}). Let $\{ {\sf V}_k \}_{1 \leq k \leq \ell}$ be the cuspidal $R_{\mathfrak{g}_{\rm fin}}$-modules (see \cite{Mc15}) associated with $\underline{w_0}$ and
${\sf S}_k := \mathcal{F}_\mathcal{D}({\sf V}_k)$ for $1 \leq k \leq \ell$. 
We remark that ${\sf V}_k$ corresponds to the dual PBW vector corresponding to $\beta_k$ under the isomorphism \cite[(2.2)]{Mc15}. 
Now, we extend a set $\{ i_k \}_{1 \leq k \leq \ell}$ to $\{ i_k \}_{k \in \mathbb{Z} }$ by $i_{k+\ell} = (i_k)^\ast$ for all $k \in \mathbb{Z}$, 
where $i^\ast$ is the unique element of $I_{\rm fin}$ such that $\alpha_{i^\ast} = -w_0\alpha_i$ for $i \in I_{\rm fin}$. 
In addition, we extend a collection $\{ {\sf S}_k \}_{1 \leq k \leq \ell }$ of modules to $\{ {\sf S}_k \}_{k \in \mathbb{Z} }$ by 
${\sf S}_{k+\ell} = \mathscr{D}({\sf S}_k)$ for all $k \in \mathbb{Z}$. 
The modules ${\sf S}_k$ are called the affine cuspidal modules corresponding to $\mathcal{D}$ and $\underline{w_0}$ (cf. \cite[Section 3]{KP22}). 

For a given complete duality datum $\mathcal{D}$ and a fixed reduced expression $\underline{w_0}$ for $\mathfrak{g}_{\rm fin}$, 
we define a standard module ${\sf P}_{\mathcal{D}, \underline{w_0}}({\bf a})$ by 
\[ {\sf P}_{\mathcal{D}, \underline{w_0}}({\bf a}) = \dots \otimes {\sf S}_2^{\otimes a_2} \otimes {\sf S}_1^{\otimes a_1} \otimes {\sf S}_0^{\otimes a_0} \otimes {\sf S}_{-1}^{\otimes a_{-1}} \otimes \dots \]
for ${\bf a} = (a_k)_{k \in \mathbb{Z}} \in \mathbb{Z}_+^{\oplus \mathbb{Z}}$.
It is proved in \cite[Theorem 6.10]{KKOP24} that 
the head of ${\sf P}_{\mathcal{D}, \underline{w_0}}({\bf a})$ is simple, and 
that any simple module $M$ in $\mathcal{C}_\mathfrak{g}^0$ (or an element $M \in \mathscr{B}_\mathcal{D}(\mathfrak{g})$) 
is written as the head of ${\sf P}_{\mathcal{D}, \underline{w_0}}({\bf a})$ for some unique ${\bf a} \in \mathbb{Z}_+^{\oplus \mathbb{Z}}$. 
Hence, when $\mathcal{D}$ and $\underline{w_0}$ are fixed, we describe $M \in \mathscr{B}_\mathcal{D}(\mathfrak{g})$ as the unique sequence of nonnegative integers, say ${\bf a}(M)$. 
By the construction of (affine) cuspidal modules, the sequence ${\bf a}(M)$ consists of PBW data of all components of ${\bf b} = (b_k) \in \widehat{B}_{\mathfrak{g}_{\rm fin}}(\infty)$. 
In other words, if $\Phi_\mathcal{D}({\bf b}) = M$ and ${\bf a}(M) = (a_k)_{k \in \mathbb{Z}}$, 
then $(a_{\ell + k\ell}, \dots, a_{1+k\ell})$ coincides with the PBW datum of $b_k$ for $k \in \mathbb{Z}$ (associated with $\underline{w_0}$).

\subsection{Examples} \label{subsec:categorical crystal example}
We exemplify the action of extended crystal operators on $\mathscr{B}_\mathcal{D}(\mathfrak{g})$ when $\mathfrak{g}$ is of type $A_n^{(1)}$ and $D_n^{(1)}$. 
We choose the reduced expression as \eqref{eqn:reduced} for all examples of this section.

Before introducing examples, we briefly explain diagrams given in examples. 
They illustrate PBW datum of a given simple module $M$. 
The point at $(i, p) \in \sigma_0(\mathfrak{g})$ designates the fundamental representation $V(\varpi_i)_{(-q)^p}$ 
which corresponds to an affine cuspidal module ${\sf S}_k$ for some $k$ in our examples. 
The integer at $(i, p)$ means the integer $a_k$ which is the (tensor) exponent of the corresponding affine cuspidal modules ${\sf S}_k$. 
Note that we underline numbers at the $0$-th position and empty positions are assumed to be filled with $0$.
\begin{rem}
    With the help of the $q$-character theory of quantum affine algebras (cf. \cite{FM01, FR99}), 
    we can parametrize any simple module in $\mathcal{C}_\mathfrak{g}$ using dominant monomials (cf. \cite{Naj03}). 
    Indeed, it is proved in \cite{CP91,CP95} that any simple module in $\mathcal{C}_\mathfrak{g}$ is isomorphic to 
    a quotient of a tensor product of fundamental modules $V(\varpi_i)_a$ ($i \in I, a \in \mathbb{C}^\times$). 
    Meanwhile, the $q$-character of a simple module $M$ in $\mathcal{C}_\mathfrak{g}$ contains a unique highest weight dominant monomial, 
    so a given simple module can be characterized by the dominant monomial. 
    Note that the dominant monomial corresponding to the fundamental module $V(\varpi_i)_a$ is $Y_{i, a}$. 
    In the following examples, we additionally declare the dominant monomial corresponding to given simple modules. 
    Due to the choice of $\sigma_0(\mathfrak{g})$ and $\mathcal{D}$, dominant monomials appearing in the examples 
    are monomials in variables $\{ Y_{i, (-q)^p} \,|\, i \in I, p \in \mathbb{Z} \}$, and we simply denote $Y_{i,p} = Y_{i, (-q)^p}$.
\end{rem}
\begin{ex}[Type $A_3$]
    Let
    \[ \mathcal{D} = \{ L_1 = V(\varpi_1)_0, \ L_2 = V(\varpi_1)_{(-q)^2}, \ L_3 = V(\varpi_1)_{(-q)^4} \} \]
    be a complete duality datum and consider a simple module $M = \Phi_\mathcal{D}({\bf b})$ where 
    \[ {\bf b} = (\dots, {\bf 1}, (1,2,0,0,1,2), \underline{(1,2,2,1,0,3)}, {\bf 1}, \dots). \]
    Note that the dominant monomial corresponding to $M$ is $Y_{3,8}Y_{2,7}^2Y_{2,5}Y_{3,4}^2Y_{1,4}Y_{2,3}^2Y_{3,2}^2Y_{1,2}Y_{1,0}^3$.
    \[ \begin{tikzpicture}
        \foreach \x in {-1,0,1,2,...,8}
            {\node at (\x, 0) {\x};}
        \foreach \y in {1, 2, 3}
            {\node at (-2, -\y) {\y};}
        \draw (-2.5, -0.5) -- (8.5, -0.5);
        \draw (-1.5, 0.5) -- (-1.5, -3.5);
        \draw (-2.5, 0.5) -- (-1.5, -0.5);
        \node[above right] at (-2, 0) {$p$};
        \node[below left] at (-2, 0) {$i$};

        \draw[thick] (-1, -1) -- (1, -3);
        \draw[thick] (3, -3) -- (5, -1);
        \draw[thick] (7, -1) -- (9, -3);

        \node at (0, -1) {$\underline{3}$};
        \node at (1, -2) {$\underline{0}$};
        \node at (2, -1) {$\underline{1}$};
        \node at (2, -3) {$\underline{2}$};
        \node at (3, -2) {$\underline{2}$};
        \node at (4, -1) {$\underline{1}$};
        
        \node at (4, -3) {$2$};
        \node at (5, -2) {$1$};
        \node at (6, -3) {$0$};
        \node at (6, -1) {$0$};
        \node at (7, -2) {$2$};
        \node at (8, -3) {$1$};
    \end{tikzpicture} \]
    We directly check that $\widetilde{F}_{(1,0)}({\bf b}) = (\dots,{\bf 1}, (1,2,0,0,1,2), \underline{(1,2,2,1,0,4)}, {\bf 1}, \dots)$. 
    In addition, the image of ${\bf b}$ and $\widetilde{F}_{(1,0)}({\bf b})$ under the bijection $\widehat{\phi}$, 
    which is the bijection obtained by applying the bijection $\phi$ \eqref{eqn:PBW-polyhed} to each component, are
    \begin{eqnarray*}
        \widehat{\phi}({\bf b}) &=& (\dots, {\bf 1}, \,(0,2,1,3,1,2),\, \underline{(2,4,0,5,1,3)},\, {\bf 1}, \dots), \\
        \widehat{\phi}(\widetilde{F}_{(1,0)}({\bf b})) &=& (\dots, {\bf 1}, \,(0,2,1,3,1,2),\, \underline{(2,4,0,5,1,4)},\, {\bf 1}, \dots),
    \end{eqnarray*}
    respectively. It is immediate from Example \ref{ex:A3 extended crystal} that $\widetilde{F}_{(1,0)}\widehat{\phi}({\bf b}) = \widehat{\phi}(\widetilde{F}_{(1,0)}({\bf b}))$.
    We remark that the dominant monomial corresponding to $\widetilde{F}_{(1,0)}(M)$ is $Y_{3,8}Y_{2,7}^2Y_{2,5}Y_{3,4}^2Y_{1,4}Y_{2,3}^2Y_{3,2}^2Y_{1,2}Y_{1,0}^4$.
\end{ex}

\begin{ex}[Type $D_4$]
    Let
    \[ \mathcal{D} = \{ L_1 = V(\varpi_1)_0, \ L_2 = V(\varpi_1)_{(-q)^2}, \ L_3 = V(\varpi_3)_{(-q)^6}, \ L_4 = V(\varpi_4)_{(-q)^6} \} \]
    be a complete duality datum and consider a simple module $M = \Phi_\mathcal{D}({\bf b})$ where 
    \[ {\bf b} = {\bf a}(M) = (\dots, {\bf 1}, (0,0,0,1,0,0,0,0,1,0,1,0), \underline{(0,1,0,0,0,0,1,0,1,0,0,0)}, {\bf 1}, \dots). \]
    Note that the dominant monomial corresponding to $M$ is $Y_{10,1}Y_{8,4}Y_{7,2}Y_{6,3}Y_{3,2}Y_{2,4}$.
    \[ \begin{tikzpicture}[xscale=0.9]
        \foreach \x in {0,1,2,...,12}
            {\node at (\x, 0) {\x};}
        \foreach \y in {1, 2, 3, 4}
            {\node at (-1, -\y) {\y};}
        \draw (-1.5, -0.5) -- (12.5, -0.5);
        \draw (-0.5, 0.5) -- (-0.5, -4.5);
        \draw (-1.5, 0.5) -- (-0.5, -0.5);
        \node[above right] at (-1, 0) {$p$};
        \node[below left] at (-1, 0) {$i$};

        \draw[thick] (0, -1.5) -- (1, -2.5) -- (1, -4.5);
        \draw[thick] (5, -1) -- (7, -2.5) -- (7, -4.5);
        \draw[thick] (11, -1) -- (13, -2.5) -- (13, -4.5);

        \node at (0, -1) {$\underline{0}$};
        \node at (1, -2) {$\underline{0}$};
        \node at (2, -3) {$\underline{0}$};
        \node at (2, -4) {$\underline{1}$};
        \node at (2, -1) {$\underline{0}$};
        \node at (3, -2) {$\underline{1}$};
        \node at (4, -3) {$\underline{0}$};
        \node at (4, -4) {$\underline{0}$};
        \node at (4, -1) {$\underline{0}$};
        \node at (5, -2) {$\underline{0}$};
        \node at (6, -3) {$\underline{1}$};
        \node at (6, -4) {$\underline{0}$};
        
        \node at (6, -1) {$0$};
        \node at (7, -2) {$1$};
        \node at (8, -3) {$0$};
        \node at (8, -4) {$1$};
        \node at (8, -1) {$0$};
        \node at (9, -2) {$0$};
        \node at (10, -3) {$0$};
        \node at (10, -4) {$0$};
        \node at (10, -1) {$1$};
        \node at (11, -2) {$0$};
        \node at (12, -3) {$0$};
        \node at (12, -4) {$0$};
    \end{tikzpicture} \]
    We directly check that
    \[ \widetilde{F}_{(1,0)}({\bf b}) = (\dots, {\bf 1}, (0,0,0,1,0,0,0,0,1,0,1,0), \underline{(0,1,0,0,0,0,1,0,2,0,0,0)}, {\bf 1}, \dots). \] 
    In addition, the image of ${\bf b}$ and $\widetilde{F}_{(1,0)}({\bf b})$ under the bijection $\widehat{\phi}$ are
    \begin{eqnarray*}
        \widehat{\phi}({\bf b}) &=& (\dots, {\bf 1}, (0,0,0,2,0,0,2,1,2,1,1,0), \underline{(0,0,0,2,0,1,3,0,2,1,0,0)}, {\bf 1}, \dots), \\
        \widehat{\phi}(\widetilde{F}_{(1,0)}{\bf b}) &=& (\dots, {\bf 1}, (0,0,0,2,0,0,2,1,2,1,1,0), \underline{(0,0,0,3,0,1,3,0,2,1,0,0)}, {\bf 1}, \dots),
    \end{eqnarray*}
    respectively. It is immediate from Example \ref{ex:D4 extended crystal} that $\widetilde{F}_{(1,0)}\widehat{\phi}({\bf b}) = \widehat{\phi}(\widetilde{F}_{(1,0)}({\bf b}))$. 
    We remark that the dominant monomial corresponding to $\widetilde{F}_{(1,0)}(M)$ is $Y_{10,1}Y_{8,4}Y_{7,2}Y_{6,3}Y_{3,2}Y_{2,4}^2$.
\end{ex}

\section{Proof of Theorem \ref{thm:bicrystal}} \label{sec:bicrystal proof}
The goal of this section is to check the conditions appearing in Proposition \ref{prop:bicrystal}, which implies Theorem \ref{thm:bicrystal}. 
The conditions are following.
\begin{enumerate}
    \item[(0)] $(\mathcal{B}(\infty), {\rm wt}, \widetilde{e}_i^\ast, \widetilde{f}_i^\ast, \varepsilon_i^\ast, \varphi_i^\ast)$ is a highest weight $U_q(\mathfrak{g})$-crystal.
    \item $\widetilde{f}_i(b) \neq {\bf 0}$, $\widetilde{f}_i^\ast(b) \neq {\bf 0}$
    \item $B$ has the unique highest weight vector $b_0$ with weight $0$, that is, $\widetilde{e}_i(b_0) = \widetilde{e}_i^\ast(b_0) = 0$ for all $i \in I$.
    \item $\widetilde{f}_i\widetilde{f}_j^\ast(b) = \widetilde{f}_j^\ast\widetilde{f}_i(b)$
    \item $\varepsilon_i(b) + \varepsilon_i^\ast(b) + \langle h_i, {\rm wt}(b) \rangle \geq 0$
    \item If $\varepsilon_i(b) + \varepsilon_i^\ast(b) + \langle h_i, {\rm wt}(b) \rangle = 0$, then $\widetilde{f}_i(b) = \widetilde{f}_i^\ast(b)$.
    \item If $\varepsilon_i(b) + \varepsilon_i^\ast(b) + \langle h_i, {\rm wt}(b) \rangle \geq 1$, then $\varepsilon_i^\ast(\widetilde{f}_i(b)) = \varepsilon_i^\ast(b)$ and $\varepsilon_i(\widetilde{f}_i^\ast(b)) = \varepsilon_i(b)$.
    \item If $\varepsilon_i(b) + \varepsilon_i^\ast(b) + \langle h_i, {\rm wt}(b) \rangle \geq 2$, then $\widetilde{f}_i^\ast \widetilde{f}_i(b) = \widetilde{f}_i \widetilde{f}_i^\ast(b)$.
\end{enumerate}
We can straightforwardly check the conditions (1) and (2) regardless of the type of $\mathfrak{g}$, 
and we focus on proving that the condition (0), (3), and (4)--(7) when $\mathfrak{g}$ is of type $A$. 
Even though key ideas to check the conditions (except (1) and (2)) are essentially same, 
details are much complicated due to the complexity of combinatorial (type-dependent) models. 
In this reason, we check the conditions when $\mathfrak{g}$ is of type $A_n$, 
and give a simple note on the proof of Theorem \ref{thm:bicrystal} for type $BD$ at the end of the section (see Appendix also). 

Meanwhile, when $\mathfrak{g}$ is of type $A_n$, $\Pi_i^\ast$ is a totally ordered set and 
so we can identify partitions in $\Pi_i$ and $\Pi_i^\ast$ with integers. 
In particular, the partition $\eta_k$ is identified with $k \in \mathbb{N}$. 
Under this identification, we use integers instead of partitions in this section. 
For example, we write expressions such as $\sumb_{(k)} = \sumb_k, \sumb_{(k)}^\ast = \sumb_k^\ast$ or $m_i(b) = k, m_i^\ast(b) = k$.

\subsection{Proof on the star crystal structure}
In this subsection, we prove the following proposition.
\begin{prop} \label{prop:polyhed star crystal}
    The pair $(\mathcal{B}(\infty), {\rm wt}, \varepsilon_i^\ast, \varphi_i^\ast, \widetilde{e}_i^\ast, \widetilde{f}_i^\ast)$ is a highest weight $U_q(\mathfrak{g})$-crystal (cf. Definition \ref{def:crystal star}).
\end{prop}

To check the well-definedness of crystal operators, we introduce the following definition.
\begin{defn}[{\cite[Definition 3.3]{KaNa20}}]
    For $j \in \{ 1, 2, \dots, n+1 \}$ and $s \in \mathbb{Z}_+$, we set 
    \[ \boxed{j}^A_s = x_{s,j} - x_{s+1, j-1} \in (\mathbb{Q}^\infty)^\ast, \]
    where $x_{m,0} = x_{m, n+1} = 0$ for $m \in \mathbb{Z}_+$.
\end{defn}
In \cite[Corollary 3.9]{KaNa20}, the set $\Xi$ is combinatorially described using the above notation, that is, 
\[ \mathcal{B}(\infty) = \{\, {\bf x} \in \mathbb{Z}_+^\infty \ |\ \boxed{j}_s^A({\bf x}) \geq 0 \mbox{ for all } s \geq 1, 1 \leq j \leq n+1 \,\}. \]

\begin{lem}[{\cite[Lemma 3.4]{KaNa20}}] \label{lem:tab eqn A}
        When $\mathfrak{g}$ is of type $A_n$, we have
        \[ \boxed{j+1}_s^A = \boxed{j}_s^A - \partb_{s, j} \quad (1 \leq j \leq n,\ s \geq 1). \]
\end{lem}

\begin{prop} \label{prop:bicrystal def A}
    For $b \in \mathcal{B}(\infty)$, the maps $\widetilde{e}_i^\ast : \mathcal{B}(\infty) \to \mathcal{B}(\infty) \cup \{ {\bf 0} \}$ and 
    $\widetilde{f}_i^\ast : \mathcal{B}(\infty) \to \mathcal{B}(\infty)$ given in Definition \ref{def:crystal star} are well-defined.
\end{prop}
\begin{proof}
    Since we can apply the similar argument to $\widetilde{e}_i^\ast$, we only show that $\widetilde{f}_i^\ast(b) \in \mathcal{B}(\infty)$ for $b = (b_{s,t}) \in \mathcal{B}(\infty)$. 
    For $b \in \mathcal{B}(\infty)$ and $i \in I$, put $m_i^\ast(b) = \eta_k$ and $b' = \widetilde{f}_i^\ast(b) = (b_{s,t}')$. 
    Then we directly check 
    % \[ b_{s,t}' = \begin{cases}
    %     b_{s,t}+1 & \mbox{if } s+t = i+1 \mbox{ with }1 \leq s \leq k, \\
    %     b_{s,t}-1 & \mbox{if } s+t = i \mbox{ with } 1 \leq s < k, \\
    %     b_{s,t} & \mbox{otherwise,}
    % \end{cases} \]
    % and
    \[ \boxed{j}_s^A (b') = \begin{cases}
        \boxed{j}_s^A(b) + 1 & \mbox{if $(s, j)=(k, i+1-k)$}, \\
        \boxed{j}_s^A(b) -1 & \mbox{if $(s, j) = (k-1, i+1-k)$}, \\
        \boxed{j}_s^A(b) & \mbox{otherwise.}
    \end{cases} \]
    When $(s, j) \neq (k-1, i+1-k)$ for all $1 < k \leq i$, we have $\boxed{j}_s^A(b') \geq \boxed{j}_s^A(b) \geq 0$.

    Suppose $(s, j) = (k-1, i+1-k)$ for some $1 < k \leq i$ and we will prove that $\boxed{j}_{s}^A(b') \geq 0$. 
    Since $m_i^\ast(b) = \eta_k \in \Pi_i^\ast$, we have $\Gamma_{(k)}^\ast(b) > \Gamma_{(k-1)}^\ast(b)$, that is, $\partb_{k,i+1-k}^\ast(b) > 0$. We have
    \begin{eqnarray*}
        \boxed{i+1-k}_{k-1}^A(b') &=& \boxed{i+1-k}_{k-1}^A(b) - 1 \\
            &=& \boxed{i+2-k}_{k-1}^A(b) + (\partb_{k, i+1-k}^\ast(b) -1)
    \end{eqnarray*}
    by Lemma \ref{lem:tab eqn A}. Since two terms are nonnegative, we have
    \[ \boxed{i+1-k}_{k-1}^A(b') \geq 0. \]
\end{proof}

By the direct observation, we can find the difference between $\sumb_{s}^\ast(b)$ and $\sumb_{s}^\ast(\widetilde{f}_i^\ast(b))$.
\begin{lem} \label{lem:sumb eqn A}
    Suppose that $\mathfrak{g}$ is of type $A_n$.
    For $i \in I$, when $m_i^\ast(b) = 1 \ (\in \Pi_i^\ast)$, we have
    \[ \sumb_s^\ast(\widetilde{f}_i^\ast(b)) = \begin{cases}
        \sumb_s^\ast(b) + 1 & \mbox{if } s=1, \\
        \sumb_s^\ast(b) & \mbox{otherwise.}
    \end{cases} \]
    When $m_i^\ast(b) = k \ (\in \Pi_i^\ast)$ for some $k \geq 2$, we have
    \[ \sumb_s^\ast(\widetilde{f}_i^\ast(b)) = \begin{cases}
        \sumb_s^\ast(b) + 2 & \mbox{if } s<k, \\
        \sumb_s^\ast(b) + 1 & \mbox{if } s=k, \\
        \sumb_s^\ast(b) & \mbox{if } s>k.
    \end{cases} \]
    In particular, we have $\sumb_{k}^\ast(\widetilde{f}_i^\ast(b)) = \sumb_{k}^\ast(b) + 1$ for $k = m_i^\ast(b)$.
\end{lem}
\begin{proof}
    Put $b' = \widetilde{f}_i^\ast(b)$ with $b = (b_{s,t})$ and $b' = (b_{s,t}')$. 
    When $m_i^\ast(b) = 1$, we can observe
    \[ \partb_{s,i+1-s}^\ast(\widetilde{f}_i^\ast(b)) = \begin{cases}
        \partb_{s,i+1-s}^\ast(b) + 1 & \mbox{if } s=1, \\
        \partb_{s,i+1-s}^\ast(b) - 1 & \mbox{if } s=2, \\
        \partb_{s,i+1-s}^\ast(b) & \mbox{otherwise}.
    \end{cases} \]
    On the other hand, suppose $m_i^\ast(b) = k$ for some $k \geq 2$ and then we have
    \[ \partb_{s,i+1-s}^\ast (\widetilde{f}_i^\ast(b)) = \begin{cases}
        \partb_{s,i+1-s}^\ast(b) + 2 & \mbox{if } s = 1, \\
        \partb_{s,i+1-s}^\ast(b) - 1 & \mbox{if $s=k$ or $s=k+1$}, \\
        \partb_{s,i+1-s}^\ast(b) & \mbox{otherwise}.
    \end{cases} \]
    The statement follows from the above observations.
\end{proof}

Now, we can show that $\mathcal{B}(\infty)$ satisfies the condition (4) in Definition \ref{def:crystal def}. 
\begin{prop} \label{prop:bicrystal inverse A}
    For $b, b' \in \mathcal{B}(\infty)$, we have $\widetilde{f}_i^\ast(b) = b'$ if and only if $\widetilde{e}_i^\ast(b') = b$.
\end{prop}
\begin{proof}
    Since the argument can be applied similarly, it is sufficient to show that if $b' = \widetilde{f}_i^\ast(b)$, 
    then $b = \widetilde{e}_i^\ast(b')$. It is equivalent to show that
    \[ \mbox{if } m_i^\ast(b) = m, \quad \mbox{then } M_i^\ast(b') = m, \]
    that is, we show $\sumb_m^\ast(b') \geq \sumb_k^\ast(b')$ for $k < m$ and 
    $\sumb_m^\ast(b') > \sumb_k^\ast(b')$ for $k > m$. 
    
    Let $m = m_i^\ast(b)$. For $k < m$, we have
    \[ \sumb_m^\ast(b') - \sumb_k^\ast(b') = \sumb_m^\ast(b) - \sumb_k^\ast(b) - 1 \]
    by Lemma \ref{lem:sumb eqn A}. Since $\sumb_m^\ast(b) > \sumb_k^\ast(b)$, we have $\sumb_m^\ast(b') \geq \sumb_k^\ast(b')$.
    On the other hand, for $k > m$, we have
    \[ \sumb_m^\ast(b') - \sumb_k^\ast(b') = \sumb_m^\ast(b) - \sumb_k^\ast(b) + 1 \]
    by Lemma \ref{lem:sumb eqn A}. Since $\sumb_m^\ast(b) \geq \sumb_k^\ast(b)$, we have $\sumb_m^\ast(b') > \sumb_k^\ast(b')$, which completes the proof.
\end{proof}

Now, we show that $\mathcal{B}(\infty)$ satisfies the axioms for (abstract) crystals.
\begin{proof}[Proof of Proposition \ref{prop:polyhed star crystal}]
    First of all, the crystal operators $\widetilde{e}_i^\ast$ and $\widetilde{f}_i^\ast$ are well-defined by Proposition \ref{prop:bicrystal def A}. 
    We easily show that all the other maps are well-defined. 

    Now, we show that $\mathcal{B}(\infty)$ satisfies the axioms given in Definition \ref{def:crystal def}. 
    The condition (1) holds by definition, and the condition (5) is vacuously true. 
    The condition (4) is proved by Proposition \ref{prop:bicrystal inverse A}.
    Since the condition (3) can be proved by the similar argument for the condition (2), we check the condition (2).

    Suppose $\widetilde{e}_i^\ast(b) \neq {\bf 0}$ for $b \in \mathcal{B}(\infty)$. 
    By the note below Definition \ref{def:crystal usual}, we have ${\rm wt}(\widetilde{e}_i^\ast b) = {\rm wt}(b) + \alpha_i$. 
    If we set $\varepsilon_i^\ast(b) = \varepsilon$, then
    \begin{eqnarray*}
        && (\widetilde{e}_i^\ast)^\varepsilon (\widetilde{e}_i^\ast b) = (\widetilde{e}_i^\ast)^{\varepsilon+1}(b) = {\bf 0}, \\
        && (\widetilde{e}_i^\ast)^{\varepsilon-1} (\widetilde{e}_i^\ast b) = (\widetilde{e}_i^\ast)^\varepsilon(b) \neq {\bf 0}.
    \end{eqnarray*}
    The first equation implies that $\varepsilon_i^\ast(\widetilde{e}_i^\ast b) < \varepsilon$, and
    the second equation implies that $\varepsilon_i^\ast(\widetilde{e}_i^\ast b) \geq \varepsilon -1$. 
    Thus, $\varepsilon_i^\ast(\widetilde{e}_i^\ast b) = \varepsilon -1$. 
    On the other hand, we have
    \begin{eqnarray*}
        \varphi_i^\ast (\widetilde{e}_i^\ast b)
        &=& \varepsilon_i^\ast (\widetilde{e}_i^\ast b) + \langle h_i, {\rm wt}(\widetilde{e}_i^\ast b) \rangle \\
        &=& \varepsilon_i^\ast (b) -1 + \langle h_i, {\rm wt}(b) + \alpha_i \rangle \\
        &=& \varepsilon_i^\ast(b) + \langle h_i, {\rm wt}(b) \rangle +1 \\
        &=& \varphi_i^\ast(b) + 1.
    \end{eqnarray*}

    It remains to show that $(\mathcal{B}(\infty), {\rm wt}, \varepsilon_i^\ast, \varphi_i^\ast, \widetilde{e}_i^\ast, \widetilde{f}_i^\ast)$ is a highest weight crystal. 
    For $b \in \mathcal{B}(\infty)$ and $i \in I$, we have $\varepsilon_i^\ast(b) \geq 0$ as a corollary of Theorem \ref{thm:epsilon}, 
    and $\widetilde{e}_i^\ast(b) = {\bf 0}$ if and only if $\varepsilon_i^\ast(b) = 0$ by definition. 
    We can show that if $\widetilde{e}_i^\ast(b) = {\bf 0}$ for all $i \in I$, or equivalently $\varepsilon_i^\ast(b) = 0$ for all $i \in I$, 
    then we have $b = {\bf 1}$ by inspecting all linear functions in $\Xi^{(i)}$. 
    It means that any $b \in \mathcal{B}(\infty)$ is connected to ${\bf 1}$ by applying a sequence of $\widetilde{e}_i^\ast$'s. 
    In addition, when we set $t = \max\{ k \geq 0 \,|\, (\widetilde{e}_i^\ast)^k(b) \neq {\bf 0} \}$, 
    we have $(\widetilde{e}_i^\ast)^t(b) \neq {\bf 0}$ and $\widetilde{e}_i^\ast((\widetilde{e}_i^\ast)^t(b)) = {\bf 0}$, 
    or equivalently $\varepsilon_i^\ast((\widetilde{e}_i^\ast)^t(b)) = \varepsilon_i^\ast(b) - t = 0$. 
    Thus, $(\mathcal{B}(\infty), {\rm wt}, \varepsilon_i^\ast, \varphi_i^\ast, \widetilde{e}_i^\ast, \widetilde{f}_i^\ast)$ is a highest weight crystal.
\end{proof}

\subsection{Proof on the bicrystal structure}
In this subsection, we show the conditions (3), (4)-(7) given in Proposition \ref{prop:bicrystal}. 
We set
\[ {\rm jump}_i(b) = \varepsilon_i(b) + \varepsilon_i^\ast(b) + \langle h_i, {\rm wt}(b) \rangle \]
for $i \in I$ and $b \in \mathcal{B}(\infty)$.
\begin{prop} \label{prop:jump zero A}
    For any $i \in I$ and $b \in \mathcal{B}(\infty)$,
    \begin{enumerate}
        \item we have ${\rm jump}_i(b) = (\varepsilon_i(b) - \sumb_{1}(b)) + (\varepsilon_i^\ast(b) - \sumb_{1}^\ast(b))$. In particular, ${\rm jump}_i(b) \geq 0$.
        \item if ${\rm jump}_i(b) = 0$ for $b \in \mathcal{B}(\infty)$, then $\widetilde{f}_i(b) = \widetilde{f}_i^\ast(b)$.
    \end{enumerate}
\end{prop}
\begin{proof}
    We can easily check that
    \[ \langle h_i, {\rm wt}(b) \rangle = - \sumb_{1}(b) - \sumb_{1}^\ast(b). \]
    Since $\varepsilon_i(b) = \max\{ \sumb_k(b) \,|\, 1 \leq k \leq n+1-i \}$ and $\varepsilon_i^\ast(b) = \max\{ \sumb_k^\ast(b) \,|\, 1 \leq k \leq i \}$, 
    we have ${\rm jump}_i(b) \geq 0$.
    In addition, since $\varepsilon_i(b) \geq \sumb_{1}(b)$ and $\varepsilon_i^\ast(b) \geq \sumb_{1}^\ast(b)$, 
    ${\rm jump}_i(b) = 0$ if and only if $m_i(b) = 1$ and $m_i^\ast(b) = 1$. 
    In this case, we know $\widetilde{f}_i(b) = \widetilde{f}_i^\ast(b) = b + {\bf e}_{1,i}$.
\end{proof}

To consider the case ${\rm jump}_i(b) \neq 0$, we observe $\partb_{s,i+1-s}^\ast(\widetilde{f}_i(b))$ and $\partb_{s,i}(\widetilde{f}_i^\ast(b))$ for some $s$ 
as similarly as Lemma \ref{lem:sumb eqn A} to obtain the following lemma on $\sumb_k^\ast(\widetilde{f}_ib)$ and $\sumb_{l}(\widetilde{f}_i^\ast b)$.
\begin{lem} \label{lem:fixed sumb A}
    Assume that $\mathfrak{g}$ is of type $A_n$ and take $b \in \mathcal{B}(\infty)$. 
    For $i \in I$, let $m_i(b) = m$ for some $1 \leq m \leq n+1-i$ and $m_n^\ast(b) = {m^\ast}$ for some $1 \leq m^\ast \leq i$. 
    The followings hold:
    \begin{eqnarray*}
        \sumb_{s}^\ast(\widetilde{f}_i(b)) &=& \begin{cases}
            \sumb_{s}^\ast(b) + 1 & \mbox{if $m=1$ and $s = 1$,} \\
            \sumb_{s}^\ast(b) & \mbox{otherwise,}
        \end{cases} \\
        \sumb_{s}(\widetilde{f}_i^\ast(b)) &=& \begin{cases}
            \sumb_{s}(b) + 1 & \mbox{if $m^\ast = 1$ and $s=1$,} \\
            \sumb_{s}(b) & \mbox{otherwise.}
        \end{cases}
    \end{eqnarray*}
\end{lem}
\begin{prop} \label{prop:jump nonzero A}
    For $i \in I$ and $b \in \mathcal{B}(\infty)$, suppose that ${\rm jump}_i(b) \neq 0$. Then we have
    \[ \varepsilon_i^\ast(\widetilde{f}_ib) = \varepsilon_i^\ast(b) \quad\mbox{and}\quad\varepsilon_i(\widetilde{f}_i^\ast(b)) = \varepsilon_i(b). \]
    In particular, if ${\rm jump}_i(b) \geq 2$, then we have $\widetilde{f}_i\widetilde{f}_i^\ast(b) = \widetilde{f}_i^\ast\widetilde{f}_i(b)$. 
\end{prop}
\begin{proof}
    Put $\widetilde{f}_i(b) = b', \widetilde{f}_i^\ast(b) = b''$, and 
    suppose $m_i(b) = m$ for some $1 \leq m \leq n+1-i$ and $m_i^\ast(b) = {m^\ast}$ for some $1 \leq m^\ast \leq i$. 
    
    When $m \neq 1$, we deduce from Lemma \ref{lem:fixed sumb A} that
    \begin{eqnarray*}
        \sumb_{m^\ast}^\ast(b') = \sumb_{m^\ast}(b) &>& \sumb_{s}(b) = \sumb_{s}^\ast(b') \quad \mbox{for } s < m^\ast, \\
        \sumb_{m^\ast}^\ast(b') = \sumb_{m^\ast}(b) &\geq& \sumb_{s}(b) = \sumb_{s}^\ast(b') \quad \mbox{for } s > m^\ast,
    \end{eqnarray*}
    and we obtain
    \begin{equation} \label{eqn:mstar f}
        \varepsilon_i^\ast(b') = \varepsilon_i^\ast(b) \quad\mbox{and}\quad m_i^\ast(b') = m_i^\ast(b).
    \end{equation}
    On the other hand, when $m^\ast \neq 1$, we similarly deduce from Lemma \ref{lem:fixed sumb A} that
    \begin{eqnarray*}
        \sumb_{m}(b'') = \sumb_{m}(b) &>& \sumb_{s}(b) = \sumb_{s}(b'') \quad \mbox{for } s < m,  \\
        \sumb_{m}(b'') = \sumb_{m}(b) &\geq& \sumb_{s}(b) = \sumb_{s}(b'') \quad \mbox{for } s > m,
    \end{eqnarray*}
    and we obtain
    \begin{equation} \label{eqn:m fstar}
        \varepsilon_i(b'') = \varepsilon_i(b) \quad\mbox{and}\quad m_i(b'') = m_i(b).
    \end{equation}

    Now, we prove the statement in the following cases.
    First, we assume $m=1$ and $m^\ast \neq 1$.
    By \eqref{eqn:m fstar}, we have $\varepsilon_i(b'') = \varepsilon_i(b), m_i(b'') = m_i(b) (= 1)$. 
    In addition, since $m = 1$ and ${\rm jump}_i(b) \geq 1$, we have $\sumb_{m^\ast}^\ast(b) \geq \sumb_{1}^\ast(b) + 1$ and 
    \[ \sumb_{m^\ast}^\ast(b') = \sumb_{m^\ast}^\ast(b) \geq \sumb_{1}^\ast(b)+1 = \sumb_{1}^\ast(b'). \]
    On the other hand, we have 
    \begin{eqnarray*}
        && \sumb_{m^\ast}^\ast(b') = \sumb_{m^\ast}^\ast(b) > \sumb_{s}^\ast(b) = \sumb_{s}^\ast(b') \quad \mbox{for } 1 < s < m^\ast, \\
        &\mbox{and}& \sumb_{m^\ast}^\ast(b') = \sumb_{m^\ast}^\ast(b) \geq \sumb_{s}^\ast(b) = \sumb_{s}^\ast(b') \quad \mbox{for } s > m^\ast.
    \end{eqnarray*}
    Hence
    \[ \varepsilon_i^\ast(b') = \sumb_{m^\ast}^\ast(b') = \varepsilon_i^\ast(b). \]
    Moreover, if ${\rm jump}_i(b) \geq 2$, then $\sumb_{m^\ast}^\ast(b') > \sumb_{1}^\ast(b')$ and $m_i^\ast(b') = m^\ast = m_i^\ast(b)$ holds. 
    Thus, if ${\rm jump}_i(b) \geq 2$, then
    \[ \widetilde{f}_i\widetilde{f}_i^\ast(b) = \widetilde{f}_i^\ast \widetilde{f}_i (b)
        = b + \sum_{k=1}^{m^\ast} ({\bf e}_{k,i+1-k} - {\bf e}_{k-1,i+1-k}) + {\bf e}_{1,i}. \]
    
    Second, we assume $m \neq 1$ and $m^\ast = 1$. 
    By changing the role of usual crystal structure and star crystal structure with using \eqref{eqn:mstar f}, 
    we can prove the statement similarly as the first case. 
    Note that if ${\rm jump}_i(b) \geq 2$, we have
    \[ \widetilde{f}_i\widetilde{f}_i^\ast(b) = \widetilde{f}_i^\ast \widetilde{f}_i (b) = b + {\bf e}_{1,i} + {\bf e}_{m,i}. \]

    Finally, we assume $m \neq 1$ and $m^\ast \neq 1$.
    Then \eqref{eqn:mstar f} and \eqref{eqn:m fstar} directly yield our statement with
    \[ \widetilde{f}_i\widetilde{f}_i^\ast(b) = \widetilde{f}_i^\ast \widetilde{f}_i (b) = b + {\bf e}_{m,i} + {\bf e}_{m^\ast,i}. \]
    Note that the condition ${\rm jump}_i(b) \geq 2$ always holds in this case.
\end{proof}

\begin{lem} \label{lem:comm sumb A}
    Assume that $\mathfrak{g}$ is of type $A_n$ and take $b \in \mathcal{B}(\infty)$. 
    For distinct $i, j \in I$, let $m_i(b) = m$ for some $1 \leq m \leq n+1-i$ and $m_j^\ast(b) = {m^\ast}$ for some $1 \leq m^\ast \leq j$. 
    When $i > j$, we have
    \[ \sumb_{s}^\ast(\widetilde{f}_i(b)) = \sumb_{s}^\ast(b), \quad \sumb_{t}(\widetilde{f}_j^\ast(b)) = \sumb_{t}(b) \]
    for $1 \leq s \leq j$ and $1 \leq t \leq n+1-i$. 
    When $i < j$, the followings hold:
    \[ \sumb_{s}^\ast(\widetilde{f}_i(b)) = \begin{cases}
        \sumb_{s}^\ast(b) - 1 & \mbox{if $m=j-i$ and $s = m$,} \\
        \sumb_{s}^\ast(b) + 1 & \mbox{if $m=j-i+1$ and $s = m$,} \\
        \sumb_{s}^\ast(b) & \mbox{otherwise}
    \end{cases} \]
    for $1 \leq s \leq j$ and 
    \[ \sumb_{t}(\widetilde{f}_j^\ast(b)) = \begin{cases}
        \sumb_{t}(b) - 1 & \mbox{if $m^\ast = j-i$ and $t=m^\ast$,} \\
        \sumb_{t}(b) + 1 & \mbox{if $m^\ast = j-i+1$ and $t=m^\ast$,} \\
        \sumb_{t}(b) & \mbox{otherwise}
    \end{cases} \]
    for $1 \leq t \leq n+1-i$.
\end{lem}
As we can observe in the above lemma, for fixed $m$, $\sumb_{s}^\ast(\widetilde{f}_i(b))$ and $\sumb_{s}^\ast(b)$ are different for at most one $s$, 
and, for fixed $m^\ast$, $\sumb_{t}(\widetilde{f}_j^\ast(b))$ and $\sumb_{t}(b)$ are different for at most one $t$, 
and hence most inequalities among $\sumb_{s}^\ast$ for $1 \leq s \leq n+1-j$ (resp. $\sumb_{t}$ with $1 \leq t \leq i$) are identical 
when evaluating $\sumb_{s}^\ast$ at $b$ and $\widetilde{f}_i(b)$ (resp. evaluating $\sumb_t$ at $b$ and $\widetilde{f}_j^\ast(b)$).
\begin{prop} \label{prop:comm A}
    For distinct $i, j \in I$ and $b \in \mathcal{B}(\infty)$, we have $\widetilde{f}_i \widetilde{f}_j^\ast(b) = \widetilde{f}_j^\ast \widetilde{f}_i(b)$. 
\end{prop}
\begin{proof}
    Put $\widetilde{f}_i(b) = b', \widetilde{f}_j^\ast(b) = b''$, and suppose $m_i(b) = m$ for some $1 \leq m \leq n+1-i$ 
    and $m_j^\ast(b) = {m^\ast}$ for some $1 \leq m^\ast \leq j$. 
    When $i > j$, we directly check $m_j^\ast(b') = m_j^\ast(b), m_i(b'') = m_i(b)$ by Lemma \ref{lem:comm sumb A}.
    Thus, we have $\widetilde{f}_j^\ast(b') = \widetilde{f}_i(b'')$. 

    Now assume $i < j$. As similarly as above, if $m \neq j-i$ and $m \neq j-i+1$, then we directly obtain 
    $m_j^\ast(b') = m_j^\ast(b)$ and $m_i(b'') = m_i(b)$, which implies $\widetilde{f}_j^\ast(b') = \widetilde{f}_i(b'')$. 

    Suppose $m = j-i+1$ and we have $\partb_{m-1, i}(b) = \partb_{j-i+1,i}^\ast(b) < 0$ by definition of $m$.
    Then we have $\sumb_{m}^\ast(b) < \sumb_{m-1}^\ast(b)$ and so $m^\ast \neq m$. 
    Recall $\sumb_{m}^\ast(b') = \sumb_{m}^\ast(b) + 1$ and $\sumb_{s}^\ast(b') = \sumb_{s}^\ast(b)$ for all $s \neq m$, 
    and we obtain $\sumb_{m}^\ast(b') \leq \sumb_{m-1}^\ast(b')$. 
    
    When $m^\ast > m$, we have
    \[ \sumb_{m^\ast}^\ast(b') = \sumb_{m^\ast}^\ast(b) > \sumb_{m-1}^\ast(b) = \sumb_{m-1}^\ast(b') \geq \sumb_{m}^\ast(b'), \]
    which implies $m_j^\ast(b') = m_j^\ast(b)$. 
    In addition, we have $\sumb_{t}(b'') = \sumb_{t}(b)$ for all $1 \leq t \leq n+1-i$ 
    and we have $m_i(b'') = m_i(b)$, and hence $\widetilde{f}_j^\ast(b') = \widetilde{f}_i(b'')$. 

    When $m^\ast < m$, if $\sumb_{m^\ast}^\ast(b) = \sumb_{m}^\ast(b)$, then $\sumb_{m^\ast}^\ast(b) = \sumb_{m}^\ast(b) < \sumb_{m-1}^\ast(b)$, 
    which is a contradiction to the choice of $m^\ast$. 
    It means that $\sumb_{m^\ast}^\ast(b) > \sumb_{m}^\ast(b)$ and so $\sumb_{m^\ast}^\ast(b') \geq \sumb_{m}^\ast(b')$, which implies $m_j^\ast(b') = m_j^\ast(b)$. 
    In addition, we have $\sumb_{t}(b'') = \sumb_{t}(b)$ for all $1 \leq t \leq n+1-i$ unless $m^\ast = m-1 (= j-i)$ and $t = m^\ast$. 
    If $m^\ast = j-i$, then
    \[ \sumb_{m}(b'') = \sumb_{m}(b) > \sumb_{m^\ast}(b) > \sumb_{m^\ast}(b''). \]
    In both cases, we have $m_i(b'') = m_i(b)$. 
    Thus, we always obtain $m_j^\ast(b') = m_j^\ast(b)$ and $m_i(b'') = m_i(b)$, which implies $\widetilde{f}_j^\ast(b') = \widetilde{f}_i(b'')$.

    On the other hand, suppose $m = j-i$ and we have $\partb_{m,i}(b) = \partb_{m+1,i}^\ast(b) \geq 0$ by definition of $m$. 
    When $m^\ast > m$ (resp. $m^\ast < m$), we have
    \begin{eqnarray*}
        && \sumb_{m^\ast}^\ast(b') = \sumb_{m^\ast}^\ast(b) > \sumb_{m}^\ast(b) > \sumb_{m}^\ast(b') \\
        & \mbox{(resp. } & \sumb_{m^\ast}^\ast(b') = \sumb_{m^\ast}^\ast(b) \geq \sumb_{m}^\ast(b) > \sumb_{m}^\ast(b') ),
    \end{eqnarray*}
    which implies $m_j^\ast(b') = m_j^\ast(b)$. 
    In addition, we have $\sumb_{t}(b'') = \sumb_{t}(b)$ for all $1 \leq t \leq n+1-i$ unless $m^\ast = m+1 (= j-i+1)$ and $t = m^\ast$. 
    If $m^\ast = m+1$, then we have $\partb_{m+1, j+1-(m+1)}^\ast(b) = \partb_{m,i}(b) > 0$ by definition of $m^\ast$ and $\sumb_{m}(b) > \sumb_{m+1}(b)$. 
    Then we have
    \[ \sumb_{m}(b'') = \sumb_{m}(b) \geq \sumb_{m^\ast}(b) + 1 = \sumb_{m^\ast}(b''). \]
    In both cases, we have $m_i(b'') = m_i(b)$. 
    Thus, we obtain $m_j^\ast(b') = m_j^\ast(b)$ and $m_i(b'') = m_i(b)$, which implies $\widetilde{f}_j^\ast(b') = \widetilde{f}_i(b'')$.

    When $m^\ast = m$, we know $\partb_{m+1, i}^\ast(b) \leq 0$ by definition of $m^\ast$. 
    Then $\partb_{m+1, i}^\ast(b) = \partb_{m,i}(b) = 0$, which implies $\sumb_{m}(b) = \sumb_{m+1}(b)$ and $\sumb_{m^\ast}^\ast(b) = \sumb_{m^\ast+1}^\ast(b)$. 
    By Lemma \ref{lem:fixed sumb A}, we get the following inequalities;
    \begin{eqnarray*}
        \sumb_{m^\ast+1}^\ast(b') = \sumb_{m}^\ast(b) &>& \sumb_{m}^\ast(b'), \\
        \sumb_{m^\ast+1}^\ast(b') = \sumb_{m}^\ast(b) &>& \sumb_{s}^\ast(b') \quad \mbox{for } s < m^\ast, \\
        \sumb_{m^\ast+1}^\ast(b') = \sumb_{m}^\ast(b) &\geq& \sumb_{s}^\ast(b') \quad \mbox{for }s > m^\ast,
    \end{eqnarray*}
    which implies $m_i^\ast(b') = {m^\ast+1}$. 
    Similarly, we get other inequalities;
    \begin{eqnarray*}
        \sumb_{m+1}(b'') = \sumb_{m}(b) &>& \sumb_{m}(b''), \\
        \sumb_{m+1}(b'') = \sumb_{m}(b) &>& \sumb_{t}(b'') \quad \mbox{for } t < m, \\
        \sumb_{m+1}(b'') = \sumb_{m}(b) &\geq& \sumb_{t}(b'') \quad \mbox{for }t > m,
    \end{eqnarray*}
    which implies $m_i(b'') = {m+1}$. 
    By the direct calculation, we obtain
    \begin{eqnarray*}
        \widetilde{f}_j^\ast(b') &=& b' + \sum_{s=1}^{m^\ast+1} {\bf v}_{s,j+1-s} \\
            &=& \left( b + {\bf e}_{m,i} \right) + \sum_{s=1}^{m^\ast} {\bf v}_{s,j+1-s} + {\bf v}_{m+1, j-m} \\
            &=& \left( b + \sum_{s=1}^{m^\ast} {\bf v}_{s,j+1-s} \right) + ({\bf e}_{m,i} + {\bf v}_{m+1, i}) \\
            &=& b'' + {\bf e}_{m+1, i} \ =\ \widetilde{f}_i(b'').
    \end{eqnarray*}
    Therefore, we get $\widetilde{f}_i^\ast(b') = \widetilde{f}_i(b'')$, which completes the proof.
\end{proof}

As a closing remark, we briefly explain the proof for type $BD$ cases and there are $3$ nontrivial statements. 
Note that we can easily prove Proposition \ref{prop:jump zero A} for type $BD$ as similarly as for type $A$. 
The first one is Proposition \ref{prop:bicrystal def A} which shows the well-definedness of crystal operators. 
The second one is Proposition \ref{prop:bicrystal inverse A} which explains that $\widetilde{e}_i^\ast$ and $\widetilde{f}_i^\ast$ are inverse to each other. 
The last one is Proposition \ref{prop:jump nonzero A} and Proposition \ref{prop:comm A} 
which are arguments on the evaluation of linear functions at $b$ and $\widetilde{f}_i(b)$ and $\widetilde{f}_i^\ast(b)$. 

The key idea to prove those statements is to find positive values from the definition of $m_i(b)$ or $m_i^\ast(b)$. 
In the proof, we have to evaluate linear functions $\sumb_\lambda$ ($\lambda \in \Pi_i$) and $\sumb_\mu^\ast$ ($\mu \in \Pi_i^\ast$) at $b, \widetilde{f}_i(b)$, and $\widetilde{f}_i^\ast(b)$, 
and we have to show that their valuations are nonnegative (sometimes positive); for example, $\boxed{j}_s^A(b')$ or $\sumb_\mu^\ast(b') - \sumb_\lambda^\ast(b')$ for $b' = \widetilde{f}_i(b)$ or $\widetilde{f}_i^\ast(b)$. 
We can similarly apply our arguments to the proof for type $BD$ if the value at $b'$ is still nonnegative. 
However, this nonnegativity seems to fail for special choices of $m_i(b)$ or $m_i^\ast(b)$ when the value we consider decreases by $-1$ after substituting $b'$. 
Fortunately, these specific assumptions on $m_i(b)$ or $m_i^\ast(b)$ ensure that the value at $b$ becomes (strictly) positive, 
which shows that the value is still nonnegative after substituting $b'$. 

Compared to type $A$, we have to consider additional special cases for type $BD$ causing mentioned obstructions. 
We verify that Theorem \ref{thm:bicrystal} for type $BD$ still holds with resolving all special cases. 
However, it is laborious and so we only remain lemmas necessary to prove Theorem \ref{thm:bicrystal} for type $BD$ in Appendix without proof.

\appendix
\section{Lemmas for Theorem \ref{thm:bicrystal} for type $BD$}
\begin{defn}[{\cite[Definition 3.3]{KaNa20}}] \hfill

    (1) Suppose that $\mathfrak{g}$ is of type $D_n$. For $s \in \mathbb{Z}_+$, we set
    \begin{eqnarray*}
        \boxed{j}_s^D &=& \begin{cases}
            x_{s,j} - x_{s+1, j-1} & \mbox{if } 1 \leq j \leq n-2 \mbox{ or } j = n, \\
            x_{s, n-1} + x_{s, n} - x_{s+1, n-2} & \mbox{if } j = n-1,
        \end{cases} \\
        \boxed{\overline{j}}_s^D &=& \begin{cases}
            x_{s+n-j,j-1} - x_{s+n-j, j} & \mbox{if } 1 \leq j \leq n-2, \\
            x_{s+1, n-2} - x_{s+1, n-1} - x_{s+1, n} & \mbox{if } j = n-1, \\
            x_{s, n-1} - x_{s+1, n} & \mbox{if } j = n, \\
            x_{s, n} & \mbox{if } j = n+1,
        \end{cases}
    \end{eqnarray*}
    where $x_{m, 0} = 0$ for $m \in \mathbb{Z}$.
    
    (2) Suppose that $\mathfrak{g}$ is of type $B_n$. For $1 \leq j \leq n$ and $s \in \mathbb{Z}_+$, we set
    \begin{eqnarray*}
        \boxed{j}_s^B &=& x_{s, j} - x_{s+1, j-1} \\
        \boxed{\overline{j}}_s^B &=& x_{s+n-j+1, j-1} - x_{s+n-j+1, j},
    \end{eqnarray*}
    where $x_{m, 0} = 0$ for $m \in \mathbb{Z}$.
\end{defn}
In \cite[Corollary 3.9]{KaNa20}, the set $\Xi$ is combinatorially described as follows.
\begin{eqnarray*} 
    && \mathcal{B}(\infty) = \{\, {\bf x} \in \mathbb{Z}_+^\infty \ |\ \boxed{j}_s^D({\bf x}) \geq 0 \mbox{ for all } s \geq 1, j \in \{ 1, \dots, n, \overline{n+1}, \overline{n}, \dots, \overline{1} \} \,\} \\
    &\mbox{and} & \mathcal{B}(\infty) = \{\, {\bf x} \in \mathbb{Z}_+^\infty \ |\ \boxed{j}_s^B({\bf x}) \geq 0 \mbox{ for all } s \geq 1, j \in \{ 1, \dots, n, \overline{n}, \dots, \overline{1} \} \,\}
\end{eqnarray*}

\begin{lem}[{\cite[Lemma 3.4]{KaNa20} (cf. Lemma \ref{lem:tab eqn A} for type $A$)}] \label{lem:tab eqn BD}\hfill

    (1) When $\mathfrak{g}$ is of type $D_n$, we have
    \begin{eqnarray*}
        \boxed{j+1}_s^D &=& \boxed{j}_s^D - \partb_{s, j} \quad (1 \leq j \leq n-1,\ s \geq 1), \\
        \boxed{\overline{n}}_s^D &=& \boxed{n-1}_s^D - \partb_{s, n} \quad (s \geq 1), \\
        \boxed{\overline{n-1}}_s^D &=& \boxed{n}_s^D - \partb_{s, n} \quad (s \geq 1), \\
        \boxed{\overline{j-1}}_s^D &=& \boxed{\overline{j}}_s^D - \partb_{s+n-j, j-1} \quad (2 \leq j \leq n,\ s \geq 1+j-n), \\
        \boxed{\overline{n+1}}_{s+2}^D + \boxed{\overline{n}}_{s+1}^D + \boxed{\overline{n-1}}_s^D &=& \boxed{\overline{n+1}}_s^D - \partb_{s,n} \quad (s \geq 1).
    \end{eqnarray*}

    (2) When $\mathfrak{g}$ is of type $B_n$, we have
    \begin{eqnarray*}
        \boxed{j+1}_s^B &=& \boxed{j}_s^B - \partb_{s,j} \quad (1 \leq j \leq n-1,\ s \geq 1), \\
        \boxed{\overline{n}}_s^B &=& \boxed{n}_s^B - \partb_{s,n} \quad (s \geq 1), \\
        \boxed{\overline{j-1}}_s^B &=& \boxed{\overline{j}}_s^B - \partb_{s+n-j+1, j-1} \quad (2 \leq j \leq n,\ s \geq j-n).
    \end{eqnarray*}
\end{lem}
Using Lemma \ref{lem:tab eqn BD}, we can prove Proposition \ref{prop:bicrystal def A} for type $BD$. \vspace{1em}

For $\lambda \in \Pi_i^\ast$, a cell $(s, t)$ is called a strict-addable (resp. a strict-removable) node of $\lambda$ 
if $(s, t) \not \in \lambda$ (resp. $(s, t) \in \lambda$) and $\lambda \cup \{ (s, t) \}$ (resp. $\lambda \backslash \{ (s, t) \}$) is a strict partition. 
For given $\lambda \in \Pi_i^\ast$, and we enumerate the strict-removable nodes of $\mu$ as $R_p = (s_p, t_p)$ with $s_1 < \dots < s_r$, and 
the strict-addable nodes of $\mu$ as $A_q = (u_q, v_q)$ with $u_1 < \dots < u_a$. 
Note that we only count strict-addable nodes contained in the shape of $T_i^\ast$. 
Let
\[ R_p' = (s_p, t_{p+1}+(s_{p+1}-s_p)), \quad A_q' = (u_{q+1}, v_q + (u_q - u_{q+1})) \]
for $1 \leq p < r$ and $1 \leq q < a$ (if exists). 

We obtain the following lemma on the difference between $\sumb_\lambda^\ast(b)$ and $\sumb_\lambda^\ast(\widetilde{f}_i^\ast(b))$. 
\begin{lem}[cf. {Lemma \ref{lem:sumb eqn A}} for type $A$] \label{lem:sumb eqn BD}
    Suppose that $\mathfrak{g}$ is of type $D_n$ or type $B_n$. 
    For $i \in I$, let $m_i^\ast(b) = \mu \in \Pi_i^\ast$. 
    For $\lambda \in \Pi_i^\ast$, denote by $\rho$, $\alpha$, $\rho'$, $\alpha'$ 
    the number of nodes $R_i$, $A_i$, $R_i'$, $A_i'$, respectively.

    (1) When $\mathfrak{g}$ is of type $D_n$, we have
    \[ \sumb_\lambda^\ast(\widetilde{f}_i^\ast(b)) = \sumb_\lambda^\ast(b) + 2 - \rho - \alpha + \rho' + \alpha'. \]

    (2) When $\mathfrak{g}$ is of type $B_n$, we have
    \[
        \sumb_\lambda^\ast(\widetilde{f}_i^\ast(b)) \ =\ \begin{cases}
            \sumb_\lambda^\ast(b) + 2 - \rho - \alpha + \rho' + \alpha' & \mbox{if } i \neq n, \\
            \sumb_\lambda^\ast(b) + 2 - 2\rho - 2\alpha + 2\rho' + 2\alpha' \\
            \hspace{0.5cm} + \delta((\ell, 1) \in \lambda) + \delta((\ell+1, 1) \in \lambda) \phantom{a} & \mbox{if } i = n,
        \end{cases}
    \]
    where $\delta(P) = 1$ if the statement $P$ is true and $\delta(P) = 0$ otherwise.
    
    In both cases, we have $\sumb_\mu^\ast(\widetilde{f}_i^\ast(b)) = \sumb_\mu^\ast(b) + 1$ for $\mu = m_i^\ast(b)$.
\end{lem}
Using Lemma \ref{lem:sumb eqn BD}, we can prove Proposition \ref{prop:bicrystal inverse A} for type $BD$. 
As a remark, we can easily observe that, when we choose $i \leq n-2$ for type $D_n$ or $i \leq n-1$ for type $B_n$, 
the proof of Lemma \ref{lem:sumb eqn BD} is parallel to that of Lemma \ref{lem:sumb eqn A}. 
This phenomenon also occurs when we prove the remaining lemmas. 
In addition, when $\mathfrak{g}$ is of type $D_n$, statements for $i = n-1$ are proved using the same argument as for $i = n$, 
by swapping the roles of $n-1$ and $n$. 
Thus, we focus on presenting statements for $i = n$ from now on. \vspace{1em}

We obtain the following lemma on values $\sumb_\lambda^\ast(\widetilde{f}_n(b))$ and $\sumb_{(s)}(\widetilde{f}_n^\ast(b))$. 
\begin{lem}[{cf. Lemma \ref{lem:fixed sumb A}}] \label{lem:fixed sumb D}
    Assume that $\mathfrak{g}$ is of type $D_n$ and take $b \in \mathcal{B}(\infty)$. 
    Let $m_n(b) = \eta_m$ for some $1 \leq m \leq n-1$ and $m_n^\ast(b) = \mu$ for some $\mu \in \Pi_n^\ast$.
    \begin{enumerate}
        \item The followings hold.
        \begin{enumerate}
            \item $\sumb_\lambda^\ast(\widetilde{f}_n(b)) = \sumb_\lambda^\ast(b) + 1$ if
            $m$ is odd and $\lambda$ is such that $(m, 1) \in \lambda$ and $(m, 2) \not \in \lambda$;
            \item $\sumb_\lambda^\ast(\widetilde{f}_n(b)) = \sumb_\lambda^\ast(b) - 1$ if
            $m$ is even and $\lambda$ is such that $(m, 2) \in \lambda$ and $(m+1, 1) \not \in \lambda$;
            \item $\sumb_\lambda^\ast(\widetilde{f}_n(b)) = \sumb_\lambda^\ast(b)$ otherwise.
        \end{enumerate}

        \item The followings hold.
        \begin{enumerate}
            \item $\sumb_{(s)}(\widetilde{f}_n^\ast(b)) = \sumb_{(s)}(b) + 1$ 
            if $s$ is odd and $(s, 1) \in \mu$ and $(s, 2) \not \in \mu$;
            \item $\sumb_{(s)}(\widetilde{f}_n^\ast(b)) = \sumb_{(s)}(b) - 1$ 
            if $s$ is even and $(s, 2) \in \mu$ and $(s+1, 1) \not \in \mu$;
            \item $\sumb_{(s)}(\widetilde{f}_n^\ast(b)) = \sumb_{(s)}(b)$ otherwise.
        \end{enumerate}
    \end{enumerate}
\end{lem}
\begin{lem}[{cf. Lemma \ref{lem:fixed sumb A}}] \label{lem:fixed sumb B}
    Assume that $\mathfrak{g}$ is of type $B_n$ and take $b \in \mathcal{B}(\infty)$. 
    Let $m_n(b) = \eta_m$ for some $1 \leq m \leq n$ and $m_n^\ast(b) = \mu$ for some $\mu \in \Pi_n^\ast$.

    \begin{enumerate}
        \item The followings hold.
        \begin{enumerate}
            \item $\sumb_\lambda^\ast(\widetilde{f}_n(b)) = \sumb_\lambda^\ast(b) + 1$
            if $(m, 1) \in \lambda$ and $(m, 2) \not \in \lambda$
            \item $\sumb_\lambda^\ast(\widetilde{f}_n(b)) = \sumb_\lambda^\ast(b) - 1$ 
            if $(m, 2) \in \lambda$ and $(m+1, 1) \not \in \lambda$
            \item $\sumb_\lambda^\ast(\widetilde{f}_n(b)) = \sumb_\lambda^\ast(b)$ otherwise.
        \end{enumerate}

        \item The followings hold.
        \begin{enumerate}
            \item $\sumb_{(s)}(\widetilde{f}_n^\ast(b)) = \sumb_{(s)}(b) + 1$ 
            if $(s, 1) \in \mu$ and $(s, 2) \not \in \mu$;
            \item $\sumb_{(s)}(\widetilde{f}_n^\ast(b)) = \sumb_{(s)}(b) - 1$ 
            if $(s, 2) \in \mu$ and $(s+1, 1) \not \in \mu$;
            \item $\sumb_{(s)}(\widetilde{f}_n^\ast(b)) = \sumb_{(s)}(b)$ otherwise.
        \end{enumerate}
    \end{enumerate}
\end{lem}
Using Lemma \ref{lem:fixed sumb D} and Lemma \ref{lem:fixed sumb B}, we can prove Proposition \ref{prop:jump nonzero A} for type $BD$. \vspace{1em}

Finally, we calculate values $\sumb_\lambda^\ast(\widetilde{f}_i(b))$ and $\sumb_{(t)}(\widetilde{f}_j^\ast(b))$ for distinct $i, j \in I$. 
By the same reason as Lemma \ref{lem:comm sumb A}, we verify that
\begin{itemize}
    \item when $i > j$ (except $(i, j) = (n, n-1)$ for type $D_n$),
    \[ \sumb_\lambda^\ast(\widetilde{f}_i(b)) = \sumb_\lambda^\ast(b), \quad \sumb_{(t)}(\widetilde{f}_j^\ast(b)) = \sumb_{(t)}(b) \]
    for $\lambda \in \Pi_j^\ast$ and $1 \leq t \leq n-1$ (resp. $1 \leq t \leq n$) for type $D_n$ (resp. $B_n$);
    \item when $i < j$ in $\{ 1, 2, \dots, n-2 \}$ for type $D_n$ and in $\{ 1, \dots, n-1 \}$ for type $B_n$, 
    \begin{eqnarray*}
        \sumb_\lambda^\ast(\widetilde{f}_i(b)) &=& \begin{cases}
            \sumb_\lambda^\ast(b) - 1 & \mbox{if } m_i(b)=\eta_{j-i} \mbox{ and } \lambda = \eta_{j-i}, \\
            \sumb_\lambda^\ast(b) + 1 & \mbox{if } m_i(b)=\eta_{j-i+1} \mbox{ and } \lambda = \eta_{j-i+1}, \\
            \sumb_\lambda^\ast(b) & \mbox{otherwise,}
        \end{cases} \\
        \sumb_{(t)}(\widetilde{f}_j^\ast(b)) &=& \begin{cases}
            \sumb_{(t)}(b) - 1 & \mbox{if } m_j^\ast(b) = \eta_{j-i} \mbox{ and } t=j-i, \\
            \sumb_{(t)}(b) + 1 & \mbox{if } m_j^\ast(b) = \eta_{j-i+1} \mbox{ and } t=j-i+1, \\
            \sumb_{(t)}(b) & \mbox{otherwise}
        \end{cases}
    \end{eqnarray*}
    for $\lambda \in \Pi_j^\ast$ and $1 \leq t \leq n-1$.
\end{itemize}
We obtain the following lemma to address the remaining cases.
\begin{lem}[{cf. Lemma \ref{lem:comm sumb A}}] \label{lem:comm sumb D}
    Assume that $\mathfrak{g}$ is of type $D_n$ and take $b \in \mathcal{B}(\infty)$. 
    For $i \leq n-2$, let $m_i(b) = \eta_m$ for some $1 \leq m \leq n-1$ and $m_n^\ast(b) = \mu$ for some $\mu \in \Pi_n^\ast$. 
    
    (1) When $i \leq n-2$, we have
    \[ \sumb_\lambda^\ast(\widetilde{f}_i(b)) = \begin{cases}
        \sumb_\lambda^\ast(b) - 1 & \mbox{if } [m=n-i-1+k \mbox{ for some $0 < k < i$ and } \\
            & \phantom{if } (k, n-i+1) \in \lambda, (k+1, n-i-1) \in \lambda, (k+1, n-i) \not \in \lambda] \\
            & \phantom{if } \mbox{or } [m = n-1-i, (1, m) \in \lambda, (1, m+1) \not \in \lambda], \\
        \sumb_\lambda^\ast(b) + 1 & \mbox{if } [m=n-i-1+k \mbox{ for some $0 < k < i$ and } \\
            & \phantom{if } (k, n-i) \in \lambda, (k, n-i+1) \not \in \lambda, (k+1, n-i-1) \not \in \lambda] \\
            & \phantom{if } \mbox{or } [m = n-1, (i, n-i) \in \lambda, (i+1, n-i-1) \not \in \lambda], \\
        \sumb_\lambda^\ast(b) & \mbox{otherwise}
    \end{cases} \]
    for $\lambda \in \Pi_n^\ast$ and
    \[ \sumb_{(t)}(\widetilde{f}_n^\ast(b)) = \begin{cases}
        \sumb_{(t)}(b) - 1 & \mbox{if } [t = n-i-1+k \mbox{ for some $0 < k < i$ and } \\
            & \phantom{if } (k, n-i+1) \in \mu, (k+1, n-i-1) \in \mu, (k+1, n-i) \not \in \mu] \\
            & \phantom{if } \mbox{or } [t = n-i-1, (1, t) \in \mu, (1, t+1) \not \in \mu], \\
        \sumb_{(t)}(b) + 1 & \mbox{if } [t=n-i-1+k \mbox{ for some $0 < k < i$ and } \\
        & \phantom{if } (k, n-i) \in \mu, (k, n-i+1) \not \in \mu, (k+1, n-i-1) \not \in \mu] \\
        & \phantom{if } \mbox{or } [t = n-1, (i, n-i) \in \mu, (i+1, n-i-1) \not \in \mu], \\
        \sumb_{(t)}(b) & \mbox{otherwise}
    \end{cases} \]
    for $1 \leq t \leq n-1$.

    (2) When $i = n-1$, we have
    \[ \sumb_\lambda^\ast(\widetilde{f}_i(b)) = \begin{cases}
        \sumb_\lambda^\ast(b) - 1 & \mbox{if $m$ is odd and } (m, 2) \in \lambda, (m+1, 1) \not \in \lambda, \\
        \sumb_\lambda^\ast(b) + 1 & \mbox{if $m$ is even and } (m, 1) \in \lambda, (m, 2) \not \in \lambda, \\
        \sumb_\lambda^\ast(b) & \mbox{otherwise,}
    \end{cases} \]
    for $\lambda \in \Pi_n^\ast$ and
    \[ \sumb_{(t)}(\widetilde{f}_n^\ast(b)) = \begin{cases}
        \sumb_{(t)}(b) - 1 & \mbox{if $t$ is odd and } (t, 2) \in \mu, (t+1, 1) \not \in \mu, \\
        \sumb_{(t)}(b) + 1 & \mbox{if $t$ is even and } (t, 1) \in \mu, (t, 2) \not \in \mu, \\
        \sumb_{(t)}(b) & \mbox{otherwise}
    \end{cases} \]
    for $1 \leq t \leq n-1$.
\end{lem}
\begin{lem}[{cf. Lemma \ref{lem:comm sumb A}}] \label{lem:comm sumb B}
    Assume that $\mathfrak{g}$ is of type $B_n$ and take $b \in \mathcal{B}(\infty)$. 
    For $i \leq n-1$, let $m_i(b) = \eta_m$ for some $1 \leq m \leq n$ and $m_n^\ast(b) = \mu$ for some $\mu \in \Pi_n^\ast$. 
    Then we have
    \[ \sumb_\lambda^\ast(\widetilde{f}_i(b)) = \begin{cases}
        \sumb_\lambda^\ast(b) - 2 & \mbox{if } [m=n-i+k \mbox{ for some $0 < k < i$ and } \\
            & \phantom{if } (k, n-i+2) \in \lambda, (k+1, n-i) \in \lambda, (k+1, n-i+1) \not \in \lambda] \\
            & \phantom{if } \mbox{or } [m = n-i, (1, m) \in \lambda, (1, m+1) \not \in \lambda], \\
        \sumb_\lambda^\ast(b) + 2 & \mbox{if } [m=n-i+k \mbox{ for some $0 < k < i$ and } \\
            & \phantom{if } (k, n-i+1) \in \lambda, (k, n-i+2) \not \in \lambda, (k+1, n-i) \not \in \lambda] \\
            & \phantom{if } \mbox{or } [m = n, (i, n-i+1) \in \lambda, (i+1, n-i) \not \in \lambda], \\
        \sumb_\lambda^\ast(b) & \mbox{otherwise}
    \end{cases} \]
    for $\lambda \in \Pi_n^\ast$ and
    \[ \sumb_{(t)}(\widetilde{f}_n^\ast(b)) = \begin{cases}
        \sumb_{(t)}(b) - 1 & \mbox{if } [t = n-i+k \mbox{ for some $0 < k < i$ and } \\
            & \phantom{if } (k, n-i+2) \in \mu, (k+1, n-i) \in \mu, (k+1, n-i+1) \not \in \mu] \\
            & \phantom{if } \mbox{or } [t = n-i, (1, t) \in \mu, (1, t+1) \not \in \mu], \\
        \sumb_{(t)}(b) + 1 & \mbox{if } [t=n-i+k \mbox{ for some $0 < k < i$ and } \\
        & \phantom{if } (k, n-i+1) \in \mu, (k, n-i+2) \not \in \mu, (k+1, n-i) \not \in \mu] \\
        & \phantom{if } \mbox{or } [t = n, (i, n-i+1) \in \mu, (i+1, n-i) \not \in \mu], \\
        \sumb_{(t)}(b) & \mbox{otherwise}
    \end{cases} \]
    for $1 \leq t \leq n$.
\end{lem}
Using Lemma \ref{lem:comm sumb D} and Lemma \ref{lem:comm sumb B}, we can prove Proposition \ref{prop:comm A} for type $BD$.
Note that Lemma \ref{lem:comm sumb D} does still not cover the entire case, 
but we deduce these remaining cases by swapping the roles of $n-1$ and $n$.

\bibliographystyle{plain}
\bibliography{poly_extended.bib}

\end{document}